\documentclass[%
noccpublish,
  multilingual,german,english]{cc}

\usepackage{tilde}
\usepackage{maps}
\usepackage{mathell}
\usepackage{set}
\usepackage{color}
\usepackage{pst-node, pst-tree}

\newcommand{\degOne}{{\color{red}l}}
\newcommand{\degTwo}{{\color{red}k}}
\renewcommand{\degOne}{{\color{red}k}}
\renewcommand{\degTwo}{{\color{red}\kappa}}
\renewcommand{\degOne}{k}
\renewcommand{\degTwo}{\kappa}
\newcommand{\loglog}{\log \! \log}
\newcommand{\C}{\mathbb{C}}
\newcommand{\F}{\mathbb{F}}

\newcommand{\nodiv}{\nmid}

\newcommand{\softO}{\mathcal{O}^{\sim}}
\renewcommand{\softO}{{O}^{\sim}}

\newcommand{\dg}{d}
\renewcommand{\dg}{\text{\textcolor{blue}{$d$}}}
\renewcommand{\dg}{n}

\newcommand{\parTwo}{{\color{red}CC}}
\renewcommand{\parTwo}{a}
\makeatletter
\newcounter{enumabc}
\def\theenumabc{\@alph\c@enumabc}
\def\labelenumabc{\theenumabc.}
\def\p@enumabc{.}
\newenvironment{enumabc}{%
  \begin{list}{\labelenumabc}{%
      \usecounter{enumabc}%
      \def\makelabel##1{\hss\llap{##1}}%
    }%
  }{%
  \end{list}
}
\newcounter{enumEsub}
\def\theenumEsub{\ensuremath{E_{\@arabic\c@enumEsub}}}
\def\labelenumEsub{\theenumEsub:}
\def\p@enumEsub{}
\newenvironment{enumEsub}{%
  \begin{list}{\labelenumEsub}{%
      \usecounter{enumEsub}%
      \def\makelabel##1{\hss\llap{##1}}%
    }%
  }{%
  \end{list}
}
\makeatother

\title{Counting decomposable\\univariate polynomials}

\author{Joachim von~zur~Gathen\\
  B-IT\\
  Universit\"at Bonn\\
  D-53113 Bonn\\
  \email{gathen@bit.uni-bonn.de}\\
  \homepage{http://cosec.bit.uni-bonn.de/}
}

\begin{abstract}
  A univariate polynomial $f$ over a field is \emph{decomposable}
  if it is the composition $f=g \circ h$ of two polynomials $g$ and $h$ 
whose degree is at least $2$. We
  determine an approximation to the number of decomposables over a
  finite field. 
The tame case, where the field characteristic $p$ does not divide the degree $n$ of $f$,
is reasonably well understood, and we obtain exponentially decreasing relative error bounds.
The wild case, where $p$ divides $n$, is more challenging and our error bounds are weaker.
\end{abstract}


\begin{subject}
\end{subject}

\begin{document}


\section{Introduction}

It is intuitively clear that the decomposable polynomials form a small
minority among all polynomials (univariate over a field). The goal in
this work is to give a quantitative version of this intuition.

Our question has two facets: in the \emph{geometric} view, we want to
determine the dimension of the algebraic set of decomposable
polynomials, say over an algebraically closed field. The
\emph{combinatorial} task is to approximate the number of
decomposables over a finite field, together with a good relative error
bound.

The first task is easy. For the second task, one readily obtains an
upper bound. The challenge then is to find an essentially matching
lower bound. \Citet{gat90c,gat90d} introduced the notion of
\emph{tame} for the case where the field characteristic does not
divide the degree of the left component, and \emph{wild} for the
complementary case. (\cite{sch00c}, \S ~1.5, uses \emph{tame} in a
different sense.) Algorithmically, the tame case is well understood
since the breakthrough result of \cite{kozlan86}; see also
\cite*{gatkoz87}; \cite{kozlan89}; \cite*{kozlan96}; \cite{gutsev06},
and the survey articles of \cite{gat02c} and
\cite{gutkoz03}\nocite{grakal03} with further references. This leads
to good estimates of the number of decomposable polynomials, provided
that we can also apply a central tool in this area, namely Ritt's
Second Theorem. This provision is satisfied if the square of the
smallest prime divisor $l$ of the degree $n$ does not divide $n$.

In the wild case, the methods from the literature do not yield a
satisfactory lower bound. We present in \ref{sec:usd} a decomposition
``algorithm'' which fails on some inputs but works on sufficiently
many ones.  The algorithm is a centerpiece of this paper and yields
lower bounds on the number of decomposable polynomials in the wild
case.

An important tool for estimating the number of ``collisions'', where
different pairs of components yield the same composition, is Ritt's
Second Theorem.  Ritt worked with $F=\mathbb{C}$ and used analytic
methods. Subsequently, his approach was replaced by algebraic methods,
in the work of \cite{lev42} and \cite{dorwha74}, and \cite{sch82c}
presented an elementary but long and involved argument. Thus Ritt's
Second Theorem was also shown to hold in positive characteristic
$p$. The original versions of this required $p>~deg(g\circ
h)$. \cite{zan93} reduced this to the milder and more natural
requirement $g'(g^{*})'\neq 0$. His proof works over an algebraic
closed field, and Schinzel's \citeyear{sch00c} monograph adapts it to
finite fields. In \ref{sec:collcomp}, we provide a precise
quantitative version of this Theorem, by determining exactly the
number of such collisions in the tame case, assuming that $p \nmid
n/l$. This is based on a unique normal form for the polynomials
occurring in the Theorem.  Furthermore, we give (less precise)
substitutes in those cases where the Theorem is not applicable.

A uniqueness property in Ritt's Second Theorem is not obvious, and
indeed \cite{beang00} are puzzled by its absence. On their page 128,
they write, translated to the present notation, ``Now these rules
are a little less transparent, and a little less independent, than may
appear at first sight. First, we note that [the First Case], which is
stated in its conventional form, is rather loosely defined, for the
$k$ and $w$ are not uniquely determined by the form $x^{k}w(x^{l})$;
for instance, if $w(0)=0$, we can equally well write this expression
in the form $x^{k+l}\tilde w(x^{l})$, where $\tilde w = w/x$. Next,
$T_{2}(x,1)=x^{2}-2$ differs by a linear component from $x^{2}$, so
that in some circumstances it is possible to apply [the Second Case]
to $T_{2}(x,1)$, then [a linear composition], and then (on what is
essentially the same factor) [the Second Case]. These observations
perhaps show why it is difficult to use Ritt's result.'' These
well-motivated concerns are settled by the result of the present paper.

\ref{sec:cdup} presents the resulting estimates in the tame case.
\ref{sec:cgdp} puts together all our bounds in the general case,
resulting in a veritable jungle of case distinctions. It is not clear
whether this is the nature of the problem or an artifact of our
approach. The following is proved at the very end of the paper and
provides a pr\'ecis of our results---by necessity less precise than the
individual bounds, in particular when $q\leq 4$ or $\dg$ is (close to)
$l^{2}$. The basic statement is that $\alpha_{\dg}$ is an
approximation to the number of decomposable polynomials of degree
$\dg$, with relative error bounds of varying quality.
\begin{namedtheorem*}{Main Theorem}\label{cor:Fq}
  Let $\mathbb{F}_{q}$ be a finite field with $q$ elements and
  characteristic $p$, let $l$ be the smallest prime divisor of the
  composite integer $\dg \geq 2$, $D_n$ the set of decomposable
  polynomials in $\F_q[x]$ of degree $n$, and
  \begin{align*}
    \alpha_{\dg} =
    \begin{cases}
      2q^{l+\dg/l}(1-q^{-1})
      & \text{if } \dg \neq l^{2}, \\
      q^{2l}(1-q^{-1})
      & \text{if } \dg = l^{2}.\\
    \end{cases}
  \end{align*}
  Then the following hold.
  \begin{enumerate}
  \item\label{cor:Fq-6} $ q^{2\sqrt{\dg}}/2 \leq \alpha_{\dg}
    <2q^{\dg/2+2}.  $
  \item\label{cor:Fq-1} $ \alpha_{\dg}/2 \leq \#D_{\dg} \leq
    \alpha_{\dg}(1+q^{-\dg/3l^{2}}) < 2\alpha_{\dg} < 4q^{\dg/2+2}$.
  \item\label{cor:Fq-2} If $n\neq p^{2}$ and $q> 5$, then $\#D_{\dg} \geq
    (3-2q^{-1})\alpha_{\dg}/4 \geq q^{2\sqrt{\dg}}/2$.
  \item\label{cor:Fq-4} Unless $p=l$ and $p$ divides $\dg$ exactly
    twice, we have $\#D_{\dg} \geq \alpha_{\dg}(1-2q^{-1})$.
  \item\label{cor:Fq-5} If $p \nmid \dg$, then $| \#D_{\dg}-
    \alpha_{\dg} | \leq \alpha_{\dg}\cdot q^{-\dg/3l^{2}}$.
  \end{enumerate}
\end{namedtheorem*}
The upper and lower bounds in \short\ref{cor:Fq-1} and
\short\ref{cor:Fq-5} differ by a factor of $1 + \epsilon$, with
$\epsilon$ exponentially decreasing in the input size $n\log q$, in
the tame case and for growing $\dg/3l^{2}$. When the field
characteristic is the smallest prime divisor of $\dg$ and divides
$\dg$ exactly twice, then we have a factor of about $2$, provided that
the condition in \short\ref{cor:Fq-2} is satisfied.  In all other
cases, the factor is $1+O(q^{-1})$ over $\mathbb{F}_{q}$. It remains
a challenge whether these gaps can be reduced.

\cite{gie88} was the first to consider our counting problem. He showed
that the decomposable polynomials form an exponentially small fraction
of all univariate polynomials. My interest, dating back to the
supervision of this thesis, was rekindled by a study of similar (but
multivariate) counting problems \citep{gat08-incl-gat07} and during a
visit to Pierre D\`ebes' group at Lille, where I received a preliminary
version of \cite*{boddeb09}. Multivariate decomposable polynomials are
counted in \cite{gat08b}.

We use the methods from \cite{gat08-incl-gat07}, where the
corresponding counting task was solved for reducible, squareful,
relatively irreducible, and singular bivariate
polynomials. \Citet*{gatvio09} extends those results to multivariate
polynomials. Recently, \cite{ziemue08} found interesting characterizations of
complete decompositions, where all components are indecomposable.

\section{Decompositions}\label{secDec}

A nonzero polynomial $f\in F[x]$ over a field $F$ is \emph{monic} if
its leading coefficient $~lc(f)$ equals $1$.  We call $f$
\emph{original} if its graph contains the origin, that is, $f(0)=0$.
\begin{definition}
  \label{defComp}
  For $g, h \in F[x]$,
  $$
  f = g \circ h = g(h) \in F[x]
  $$
  is their \emph{composition}.  If $\deg g, \deg h \geq 2$, then
  $(g,h)$ is a \emph{decomposition} of $f$. A polynomial $f \in F[x]$
  is \emph{decomposable} if there exist such $g$ and $h$, otherwise
  $f$ is \emph{indecomposable}. The decomposition $(g,h)$ is
  \emph{normal} if $h$ is monic and original.
\end{definition}
\begin{remark}
  \label{invariance1}
  Multiplication by a unit or addition of a constant does not change
  decomposability, since
  $$
  f = g \circ h \Longleftrightarrow a f+b = (a g+b) \circ h
  $$
  for all $f$, $g$, $h$ as above and $a,b \in F$ with $a\neq 0$.  In
  other words, the set of decomposable polynomials is invariant under
  this action of $F^{\times} \times F$ on $F[x]$.

  Furthermore, any decomposition $(g,h)$ can be normalized by this
  action, by taking $a = ~lc (h)^{-1} \in F^{\times}$, $b=-a \cdot
  h(0) \in F$, $g^{*} = g((x-b)a^{-1}) \in F[x]$, and $h^{*} = ah+b$.
  Then $g\circ h = g^{*} \circ h^{*}$ and $(g^{*}, h^{*})$ is normal.
\end{remark}

We fix some notation for the remainder of this paper. For $\dg \geq
0$, we write
$$
P_{\dg}= \{f \in F [x] \colon \deg f \leq \dg\}
$$
for the vector space of polynomials of degree at most $\dg$, of
dimension $n+1$.  Furthermore, we consider the subsets
\begin{align*}
  P_{\dg}^{=}  & = \{f \in P_{\dg} \colon \deg f = \dg \}, \\
  P^{0}_{\dg} & = \{f \in P_{\dg}^{=} \colon f\text{ monic and
    original}\}.
\end{align*}

Over an infinite field, the first of these is the Zariski-open subset
$ P_{\dg} \smallsetminus P_{\dg-1}$ of $P_{\dg}$, and thus
irreducible, taking $P_{-1} = \{0\}$. The second one is obtained by
further imposing one equation and working modulo multiplication by
units, so that
\begin{align*}
  \dim P_{\dg}^{=}&=n+1,\\
  \dim P^{0}_{\dg}& = n-1,
\end{align*}
with $P^{0}_{0}= \varnothing$. For any divisor $e$ of $\dg$, we have
the normal composition map
$$
\map[\gamma_{\dg,e}] {P_{e}^{=} \times P^{0}_{\dg/e}} {P_{\dg}^{=}}
{(g,h)} {g \circ h,}
$$
corresponding to \ref{defComp}, and set
\begin{equation}\label{eq:Zan}
  D_{\dg,e}= ~im \gamma_{\dg,e}.
\end{equation}

The set $D_{\dg}$ of all decomposable polynomials in $P_{\dg}^{=}$
satisfies
\begin{equation}\label{substack}
  D_{\dg}= \bigcup_{\substack{e\mid \dg\\1<e<\dg}} D_{\dg,e}.
\end{equation}
In particular, $D_{\dg} = \varnothing$ if $\dg$ is prime. We also let
$I_{\dg}=P^{=}_{\dg} \smallsetminus D_{\dg}$ be the set of
indecomposable polynomials. Over a finite field $\mathbb{F}_{q}$ with
$q$ elements, we have
\begin{align*}
  \#P_{\dg}^{=}  &= q^{\dg+1}(1-q^{-1}),\\
  \#P_{\dg}^{0} &= q^{\dg-1},\\
  \#D_{\dg,e}  &\leq q^{e+n/e}(1-q^{-1}).
\end{align*}
\begin{remark}
  \label{invariant}
  By \ref{invariance1}, over an algebraically closed field, the
  codimension of $D_{\dg}$ in $P_{\dg}^{=}$ equals that of $D_{\dg}
  \cap P_{\dg}^{0}$ in $P_{\dg}^{0}$.  The same holds for $I_{\dg}$,
  and over a finite field for the corresponding fractions:
  $$
  \frac{\#D_{\dg}}{\#P_{\dg}^{=}} = \frac{\#(D_{\dg} \cap
    P_{\dg}^{0})}{\#P_{\dg}^{0}}.
  $$
\end{remark}
\begin{example}\label{ex:sl}
  We look at normal decompositions $(g,h)$ of univariate quartic
  polynomials $f$, so that $\dg=4$. By \ref{invariance1}, we may
  assume $f\in P_{4}^{0}$, and then also $g$ is monic with constant
  coefficient 0.  Thus the general case is
  $$
  (x^{2}+ax) \circ (x^{2}+bx) = x^{4} +ux^{3} + vx^{2} + wx \in F[x],
  $$
  with $a,b,u,v,w \in F$.  We find that with $a=2w/u$ and $b=u/2$
  (assuming $2u \neq 0$), the cubic and linear coefficients match, and
  the whole decomposition does if and only if
  $$u^{3} -4uv + 8w =0.$$
  This is a defining equation for the hypersurface of decomposable
  polynomials in $P_{4}^{0}$ (if $~char F \neq 2$). Translating back
  to $P_{4}^{=}$, we have
  $$
  \dim D_{4} = 4 < 5 = \dim P_{4}^{=}.
  $$ 
  This example is also in \cite{barzip76,barzip85}.
\end{example}

\section{Equal-degree collisions}\label{sec:usd}

A decomposition $(g,h)$ of $f=g\circ h$ over a field of characteristic
$p$ is called \emph{tame} if $p\nmid \deg g$, and \emph{wild}
otherwise, in analogy with ramification indices. The polynomial $f$
itself is \emph{tame} if $p\nmid \deg f$, and \emph{wild} otherwise. The tame case is well understood, both theoretically and
algorithmically. The wild case is more difficult and less well
understood; there are polynomials with superpolynomially many
``inequivalent'' decompositions \citep{gie88}.

For $u,v\in F[x]$ and $j\in\mathbb{N}$, we write
$$
u=v+O(x^{j})
$$
if $\deg (u-v)\leq j$.  We start with two facts from the literature
concerning the injectivity of the composition map. When $p\mid\dg$, a
polynomial $f=x^{\dg}+f_{i}x^{i}+O(x^{i-1})$ with $f_{i}\neq 0$ is
called \emph{simple} if $p\nmid i$ or $i<\dg-p$.
\begin{fact}\label{cor:inj}
  Let $F$ be a field of characteristic $p$, and $e$ a divisor of $\dg
  \geq 2$.
  \begin{enumerate}
  \item \label{cor:inj-1} If $p$ does not divide $e$, then
    $\gamma_{\dg,e} $ is injective, and
    $$
    \#D_{\dg,e}= q^{e+\dg/e}(1-q^{-1}).
    $$ 
  \item \label{cor:inj-3} If $p$ divides $\dg$ exactly $d$ times and
    $f\in F[x]$ is simple, then $f$ has at most $s < 2p^{d} \leq 2
    \dg$ normal decompositions, where $s=(p^{d+1}-1)/(p-1) = 1 + p +
    \cdots + p^{d}$.
  \end{enumerate}
\end{fact}
\begin{proof}
  The uniqueness in \ref{cor:inj-1} is well-known, see e.g.,
  \cite{gat90c} and the references therein.  \ref{cor:inj-3} follows
  from \cite{gat90d}, where the above notion of a simple polynomial is
  defined, and (the proof of) Corollary 3.6 of that paper shows that
  there are at most $s $ such decompositions of $f$.
\end{proof}
The paper cited for \ref{cor:inj-3} also gives an algorithm to decide
decomposability and, in that case, to compute all such decompositions.
This only applies to ``simple'' polynomials, and no nontrivial general
upper bound on the number of decompositions seems to be known.

\ref{algoWd} below uses a similar approach. On the one hand, it
applies to more restricted inputs. On the other hand, it is faster
(roughly, $n^{2}$ vs. $n^{4}$), more transparent and hence easier to
analyze, and yields a lower bound on the number of decomposables at
fixed component degrees.

In \ref{sec:cdup}, we find an upper bound $\alpha_{\dg}$ on
$\#D_{\dg}$, up to some small relative error. When the exact size of
the error term is not a concern, then this is quite easy. Furthermore,
\ref{cor:inj} immediately yields a lower bound of $\alpha_{\dg}/2$ if
$p$ is not the smallest prime divisor $l$ of $n$, and of about
$\alpha_{\dg}/4\dg$ in general, since ``most'' polynomials are simple.

Our goal in this paper is to improve these estimates. For this
purpose, we have to address the uniqueness (or lack thereof) of normal
compositions
\begin{equation}\label{eq:circ}
  g \circ h = g^{*}\circ h^{*}
\end{equation}
in two situations.  We call $\{(g,h),(g^{*},h^{*})\}$ satisfying
\ref{eq:circ} with $h \neq h^{*}$ an \emph{equal-degree collision} if
$\deg g = \deg g^{*}$ (and hence $\deg h = \deg h^{*}$), and a
\emph{distinct-degree collision} if $\deg g = \deg h^{*} \neq \deg h$
(and hence $\deg h = \deg g^{*}$). The present section deals with
equal-degree collisions, and \ref{sec:collcomp} with distinct-degree
collisions.

By \ref{cor:inj-1}, there are no equal-degree collisions when $p \nmid
\deg g$. In the more interesting case $p \mid \deg g$, collisions are
well-known to exist; \ref{ex:coll} exhibits all collisions over
$\mathbb{F}_{3}$ at degree $9$. Our goal, then, is to show that there
are few of them, so that the decomposable polynomials are still
numerous.  \ref{algoWd} provides a constructive proof of this. For
many, but not all, $(g,h)$ it reconstructs $(g,h)$ from $g \circ
h$. To quantify the benefit provided by the algorithm, we rely on a
result by Antonia \cite{blu04a}.

Distinct-degree collisions are classically taken care of by Ritt's
Second Theorem. Some versions put a restriction on $p$ that would make
our task difficult, but Umberto \cite{zan93} has cut this restriction
down to the bare minimum. The additional common restriction that $\gcd
(\deg g, \deg h)= 1$ has essentially been removed by
\cite{tor88a}, but only if $p$ does not divide the degree. If, in
addition, the composition is wild, then a look at derivatives provides
a reasonable bound. It is useful to single out a special case of wild
compositions.  
\begin{definition}\label{rem:coll}
  We call \emph{Frobenius composition} any $f \in F[x^{p}]$, since
  then $f=x^{p} \circ h^{*}$ for some $h^{*} \in P^{=}_{\dg/p}$, and
  any decomposition $(g,h)$ of $f = g \circ h$ is a \emph{Frobenius
    decomposition}. A \emph{Frobenius collision} is the following
  example of a collision \ref{eq:circ}. For any integer $j$, we denote
  by $\varphi_{j} \colon F \longrightarrow F$ the $j$th power of the
  Frobenius automorphism over a field $F$ of characteristic $p$, with
  $\varphi_{j}(a)= a^{p^{j}}$ for all $a \in F$, and extend it to an
  $\mathbb{F}_{p}$-linear isomorphism $\varphi_{j} \colon F[x]
  \longrightarrow F[x]$ with $\varphi_{j}(x)=x$.  Then if $h \in
  F[x]$, we have
  \begin{equation}\label{eq:frob}
    x^{p^{j}} \circ h = \varphi_{j} (h) \circ x^{p^{j}}.
  \end{equation}
\end{definition}

Thus any Frobenius composition except $x^{p^{2}}$ is the result of a
collision. Over $F = \mathbb{F}_{q}$, there are $q^{p^{j}-1}-1$ many
$h \in P^{0}_{p^{j}}$ with $h \neq x^{p^{j}}$ and for $m \neq p^{j}$,
this produces $q^{m-1}$ collisions with $h \in P_{m}^{0}$. By
composing with a linear function, we obtain
$q^{p^{j}+1}(1-q^{-1})(1-q^{-p^{j}+1})$ and $q^{m+1}(1-q^{-1})$
Frobenius collisions for $m=p^{j}$ and $m\neq p^{j} $, respectively.
This example is noted in \cite{sch82c}, Section I.5, page 39.

The Frobenius compositions from \ref{rem:coll} are easily described
and counted. It is useful to separate them from the others. If $p \mid
\dg$ and $l$ is a proper divisor of $\dg$, we set
\begin{align}\label{al:DD}
  \begin{aligned}
    D_{\dg}^{\varphi} & = D_{\dg} \cap F[x^{p}],\\
    D_{\dg}^{+}  &= D_{\dg} \smallsetminus D_{\dg}^{\varphi},\\
    D_{\dg,l}^{+}  &= D_{\dg,l} \cap D_{\dg}^{+},
  \end{aligned}
\end{align}
so that $D_{\dg}^{\varphi}$ comprises exactly the Frobenius
compositions of degree $\dg$.

\Citet{gat90d} presents an algorithm for certain ``wild''
decompositions $f= g \circ h$ with
$$
\deg f = \dg = \degOne \cdot m = \deg g \cdot \deg h
$$
and $p\mid \degOne$. It first makes coefficient comparisons to compute
$h$, and then a Taylor expansion to find $g$. We now take a simplified
version of that method. It does not work for all inputs, but for
sufficiently many for our counting purpose. In general, decomposing a
polynomial can be done by solving the corresponding system of
equations in the coefficients of the unknown components, say, using
Gr\"obner bases.

To fix some notation, we have integers
\begin{equation}\label{eq:int}
  d \geq 1,\;
  r=p^{d},\;
  \degOne = ar,\;
  m \geq 2,\;
  \dg = \degOne m,\;
  \degTwo\text{ with } 0 \leq  \degTwo < \degOne
  \text{ and }
  p \nmid a\degTwo,
\end{equation}
and polynomials
\begin{equation}\label{eq:ghf}
  \begin{aligned}
    g  &= x^{\degOne} + \sum_{1 \leq i \leq \degTwo}g_{i}x^{i},\\
    h & = \sum_{1 \leq i \leq m} h_{i}x^{i},\\
    f & = g \circ h = h^{\degOne}+ \sum_{1 \leq i \leq
      \degTwo}g_{i}h^{i},
  \end{aligned}
\end{equation}
with $h_{m}=1$, $h_{m-1} \neq 0$, and either $g_{\degTwo} \neq 0$ or
$g=x^{\degOne}$; the latter case corresponds to $\degTwo=0$. The idea
is to compute $h_{i}$ for $i=m-1$, $m-2$, $\ldots$, $1$ by comparing
the known coefficients of $f$ to the unknown ones of $h^{\degOne}$ and
$g_{\degTwo}h^{\degTwo}$. Special situations arise when the latter two
polynomials both contribute to a coefficient.  We denote by
$$
h^{(i)} = \sum_{i < b < m}h_{b}x^{b}
$$
the top part of $h$, so that $h^{(m-1)}=0$. Furthermore, we write
$~coeff (v,j)$ for the coefficient of $x^{j}$ in a polynomial $v$, and
$$
c_{i,j}(v) = ~coeff (v\circ (h-h^{(i)}),j).
$$
Thus $c_{m-1,j} (x^{\degOne}) = ~coeff (h^{\degOne},j)$, and in
particular, we have $c_{m-1,j}(g) = f_{j}$ for all $j$. To illustrate
the usage of these $c_{ij}$, we consider $E_{1}$ below. At some point
in the algorithm, we have determined
$g_{\degTwo},h_{m},\dots,h_{i+1}$. The appropriate $c_{ij}$ exhibits
$h_{i}$ in a simple fashion, meaning that we can compute it from
$f_{j}$ and $h^{(i)}$. Lastly we define the rational number
\begin{equation}\label{eq:ratio}
  i_{0}= m(\frac{\degTwo-a}{r-1}-a+1)= \frac{\degTwo m-\dg}{r-1}+m;
\end{equation}
thus $i_{0}<m$, and $i_{0}$ is an integer if and only if
$$
r-1 \mid (\degTwo-a)m.
$$
\begin{lemma}\label{lem:int}
  For $ 1 \leq i \leq m$ and $0 \leq j \leq \dg$, we have the
  following.
  \begin{enumEsub}
  \item\label{lem:int-1} If $i < m$, then
    \begin{equation}\label{eq:m-2}
      c_{i,(\degTwo-1)m+i}  (g_{\degTwo}x^{\degTwo})= \degTwo g_{\degTwo}h_{i},
    \end{equation}
    and $c_{m-1,\degTwo m}(g_{\degTwo}x^{\degTwo})= g_{\degTwo}$.
  \item\label{lem:int-2} If $i < m$, then
    \begin{equation}\label{eq:m-1}
      c_{i, \dg-r(m-i)} (x^{\degOne})= ah^{r}_{i}.
    \end{equation}
    If $r \nmid j$, then $~coeff(h^{\degOne},j)=0$.
  \item\label{lem:int-3} If $i_{0} \in \mathbb{N}$, then
    \begin{equation}\label{eq:m-4}
      c_{i_{0},(\degTwo-1)m+i_{0}}
      (x^{\degOne}+g_{\degTwo}x^{\degTwo})= ah_{i_{0}}^{r}+ \degTwo
      g_{\degTwo}h_{i_{0}}. 
    \end{equation}
  \item\label{lem:int-4} If $m=r$ and $\degTwo=\degOne -1$, then
    \begin{equation}\label{eq:m-3}
      \begin{aligned}
        c_{m-1,\degTwo m}(x^{\degOne}+g_{\degTwo}x^{\degTwo})&= ah_{m-1}^{r}+g_{\degTwo},\\
        c_{m-1,\degTwo m-1}(x^{\degOne}+g_{\degTwo}x^{\degTwo})&=
        -g_{\degTwo}h_{m-1}.
      \end{aligned}
    \end{equation}
  \end{enumEsub}
\end{lemma}
\begin{proof}
  For $E_{1}$, we have to consider
  $$ g_{\kappa}(x^{m}+h_{i}x^{i}+O(x^{i-1}))^{\kappa}=g_{\kappa}x^{\kappa a}+\kappa^{'}g_{\kappa} h_{i}x^{(\kappa-1)m+i}+O(x^{(\kappa-1)m+i-1}),
  $$
  furthermore
  \begin{align*}
    c_{i,(\degTwo-1)m+i}(g_{\degTwo}x^{\degTwo}) &= g_{\degTwo}\cdot
    \degTwo
    h_{i},\\
    c_{m,\degTwo m}(g_{\degTwo}x^{\degTwo})
    &=~coeff(g_{\degTwo}h^{\degTwo},\degTwo m)= g_{\degTwo},
  \end{align*}
  and $E_{1}$ follows.  For $E_{2}$, we have
  $$
  h^{a} = x^{am}+ ah_{m-1}x^{am-1}+ O(x^{am-2}).
  $$
  When $i < m$, then in the coefficient of $x^{(a-1)m+i}$, we have the
  contribution $ah_{i}$, which comes from taking in the expansion of
  $h^{a}$ the factor $x^{m}$ exactly $a-1$ times and the factor
  $h_{i}x^{i}$ exactly once; there are $a$ ways to make these choices.
  The largest degree to which a summand $h_{j}x^{j}$ contributes in
  $h^{a}$ is $(a-1)m+j$, so that those with $j < i$ do not appear in
  the coefficient under consideration, and $c_{i,(a-1)m+i}(x^{a})=
  ah_{i}$.  Raising $h^{a}$ to the $r$th power yields
  $$
  c_{i,((a-1)m+i)r} (x^{\degOne})= c_{i,((a-1)m+i)r}((x^{a})^{r})=
  a^{r}h_{i}^{r}= ah_{i}^{r}
  $$
  and proves $E_{2}$, since $((a-1)m+i)r = \dg-r(m-i)$.
  
  For $E_{3}$, we have
  \begin{align*}
    (\degTwo-1)m+i_{0}&= \dg- r(m-i_{0}),\\
    c_{i_{0},(\degTwo-1)m+i_{0}}(x^{\degOne}+g_{\degTwo}x^{\degTwo})&
    = c_{i_{0},\dg-r(m-i_{0})}(x^{\degOne})+
    c_{i_{0},(\degTwo-1)m+i_{0}}(g_{\degTwo}x^{\degTwo})\\
     &= ah_{i_{0}}^{r}+\degTwo g_{\degTwo}h_{i_{0}}.
  \end{align*}
  For $E_{4}$, we have $\degTwo m=\dg-m$ and from $E_{1}$ and $E_{2}$
  \begin{align*}
    c_{m-1, \degTwo m}(x^{\degOne}+g_{\degTwo}x^{\degTwo})  &=
    c_{m-1,\dg-m}(x^{\degOne})+c_{m-1,\degTwo m}(g_{\degTwo}x^{\degTwo})=ah_{m-1}^{r}+g_{\degTwo},\\
    c_{m-1,\degTwo m-1}(x^{\degOne}+g_{\degTwo}x^{\degTwo}) &=
    ~coeff(h^{\degOne}, \degTwo m-1)+c_{m-1,\degTwo m-1}(g_{\degTwo}x^{\degTwo})\\
    &= 0+\degTwo g_{\degTwo}h_{m-1} = -g_{\degTwo}h_{m-1}.\qed
  \end{align*}
\end{proof}

In the following algorithm, the instruction ``determine $h_{i}$ (or
$g_{\degTwo}$) by $E_{\mu}$ (at $x^{j}$)'', for $1 \leq \mu \leq 4$,
means that the property $E_{\mu}$ involves some quantity
$c_{ij}(\cdot)$ which is a summand in $~coeff(g \circ h,j)= f_{j}$,
the other summands are already known, and we can solve for $h_{i}$ (or
$g_{\degTwo}$).
When we use $E_{2}$, we first compute $y=h_{i}^{r}$ and then $h_{i}$
by extracting the $r$th root of $y$. Over a finite field, this always
yields a unique answer, since $r$ is a power of $p$. But in general,
$y$ might not have an $r$th root. We say ``compute $h_{i}^{r}$ by
$E_{2}$, then $h_{i}$ if possible'' to mean that first $y$ is
determined, then $h_{i}$ as its $r$th root; if $y$ does not have an
$r$th root, then the empty set is returned.

The main effort in the correctness proof is to show that all data
required are available at that point in the algorithm, and that the
equation can indeed be solved. The algorithm's basic structure is
driven by the relationship between the degrees $\degTwo m$ of
$g_{\degTwo}h^{\degTwo}$ and $\dg-r$ of
$h^{\degOne}-x^{\dg}$.

\begin{namedalgorithm}{Algorithm}[Input,Output]{algoWd}[Wild
  decomposition]
\item $f \in {F}[x]$ monic and original of degree $\dg=
  \degOne m$, where $F$ is a field of characteristic $p\geq 2$, $d \geq 1$, $r=p^{d}$,
  and $\degOne = ar$ with $p \nmid a$.
\item Either a set of at most $r+1$ pairs $(g,h)$ with $g,h \in
  {F}[x]$ monic and original of degrees $\degOne$ and $m$,
  respectively, and $f=g \circ h$, or ``failure''.
\item\label{algoWd:step1} Let $j$ be the largest integer for which
  $f_{j} \neq 0$ and $p \nmid j$. If no such $j$ exists then if $d
  \geq 2$ call \ref{algoWd} recursively and else call a tame
  decomposition algorithm, in either case with input $f^{*}=f^{1/p}$
  and $\degOne^{*}= \degOne/p$. If a set of $(g^{*}, h^{*})$ is output
  by the call, then return the set of all Frobenius compositions
  $(x^{p} \circ g^{*}, h^{*})$.
\item\label{algoWd:step2} If $p \nmid m$ then if $m \nmid j$ then
  return ``failure'' else set $\degTwo=j/m$. If $p \mid m$ then if $m
  \nmid j+1$ then return ``failure'' else set $\degTwo=(j+1)/m$. If $p
  \mid \degTwo$, then return ``failure''. Calculate $i_{0}= (\degTwo
  m-\dg)/(r-1)+m$.
\item\label{algoWd:step3} If $\degTwo m \geq \dg-r+2$ then do the
  following.
  \begin{enumabc}
  \item\label{algoWd:step3-a} Set $g_{\degTwo}= f_{\degTwo m}$.
  \item\label{algoWd:step3-b} Determine $h_{i}$ for $i = m-1, \ldots,
    1$ by $E_{1}$.
  \end{enumabc}
\item\label{algoWd:step4} If $\degTwo m = \dg-r+1$ then do the
  following.
  \begin{enumabc}
  \item\label{algoWd:step4-a} Set $g_{\degTwo} = f_{\degTwo m}$.
  \item\label{algoWd:step4-b} Determine $h_{m-1}$ by $E_{3}$. If
    \short\ref{eq:m-4} does not have a unique solution, then return
    ``failure''.
  \item\label{algoWd:step4-c} Determine $h_{i}$ for $i = m-2, \ldots,
    1$ by $E_{1}$.
  \end{enumabc}
\item\label{algoWd:step5} If $\degTwo m = \dg-r$ then do the
  following.
  \begin{enumabc}
  \item\label{algoWd:step5-a} Determine $h_{m-1}$ by $E_{4}$, in the
    following way. Compute the set $S$ of all nonzero $s \in
    \mathbb{F}_{q}$ with
    \begin{equation}\label{eq:as}
      as^{r+1}- f_{\degTwo m}s -f_{\degTwo m-1}=0.
    \end{equation}
    If $S = \varnothing$ then return the empty set, else do steps 5.b
    and 5.c for all $s \in S$, setting $h_{m-1} = s$.
  \item\label{algoWd:step5-b} Determine $g_{\degTwo}$ by $E_{1}$ and
    $E_{2}$ at $x^{\degTwo m}$, from $f_{\degTwo m}=
    ah^{r}_{m-1}+g_{\degTwo}$.
  \item\label{algoWd:step5-c} For $i=m-2, \ldots, 1$ determine $h_{i}$
    by $E_{1}$.
  \end{enumabc}
\item\label{algoWd:step6} If $\degTwo m < \dg-r$ then do the
  following.
  \begin{enumabc}
  \item\label{algoWd:step6-a} Determine $h^{r}_{m-1}$ by $E_{2}$, then
    $h_{m-1}$ if possible.
  \item\label{algoWd:step6-b} If $r \nmid m$ then determine
    $g_{\degTwo}$ by $E_{1}$ at $x^{\degTwo m}$ (as
    $g_{\degTwo}=f_{\degTwo m}$), else by $E_{1}$ at $x^{\degTwo m-1}$
    (via $\degTwo g_{\degTwo}h_{m-1}= f_{\degTwo m-1}$).
  \item\label{algoWd:step6-c} Determine $h^{r}_{i}$ by $E_{2}$,
    then $h_{i}$ if possible, for decreasing $i$
    with $m-2 \geq i >i_{0}$.
  \item\label{algoWd:step6-d} If $i_{0}$ is a positive integer, then
    determine $h_{i_{0}}$ by $E_{3}$. If $E_{3}$ does not yield a
    unique solution, then return ``failure''.
  \item\label{algoWd:step6-e} Determine $h_{i}$ for decreasing $i$
    with $i_{0}> i \geq 1$ by $E_{1}$.
  \end{enumabc}
\item\label{algoWd:step7} [We now know $h$.] Compute the remaining
  coefficients $g_{1}, \ldots, g_{\degTwo-1}$ as the ``Taylor
  coefficients'' of $f$ in base $h$.
\item\label{algoWd:step8} Return the set of all $(g,h)$ for which $g
  \circ h = f$. If there are none, then return the empty set.
\end{namedalgorithm}

The Taylor expansion method determines for given $f$ and $h$ the
unique $g$ (if one exists) so that $f = g \circ h$; see \cite{gat90c}.

We first illustrate the algorithm in some examples.
\begin{example}
  We let $p=5$, $\dg=50$, and $\degOne=r=5$, so that $a=d=1$ and
  $m=10$, and start with $\degTwo=4=r-1$. We assume $f_{39}= g_{4}
  h_{9} \neq 0$. Then
  $$
  h^{5}+g_{4}h^{4}= x^{50}+ h_{9}^{5}x^{45}+ (h_{8}^{5}+g_{4})x^{40}
  +4g_{4}h_{9} x^{39}+ g_{4} (4h_{8}+h^{2}_{9})x^{38}
  $$
  $$
  + x^{36}\cdot O(x) + (h_{7}^{5}+ g_{4}(4h_{5}+h_{9}h_{6} +h_{8}h_{7} +
  h_{9}^{2}h_{7}+ h_{9}h_{8}^{2}+ h_{9}^{3}h_{8}))x^{35}+O(x^{34}).
  $$

  Step \short\ref{algoWd:step1} determines $j=39$, and step
  \short\ref{algoWd:step2} finds $\degTwo= (39+1)/10$ and $i_{0}=15/2
  \not\in \mathbb{N}$. Since $\degTwo m=40 < 45 = \dg-r$, we go to
  step \short\ref{algoWd:step6}. Step
  \short\ref{algoWd:step6}\short\ref{algoWd:step6-a} computes $h_{9} $
  at $x^{45}$, step \short\ref{algoWd:step6}\short\ref{algoWd:step6-b}
  yields $g_{4}$ at $x^{39}$, step
  \short\ref{algoWd:step6}\short\ref{algoWd:step6-c} determines
  $h_{8}$ at $x^{40}$ by $E_{2}$, step
  \short\ref{algoWd:step6}\short\ref{algoWd:step6-d} is skipped, and
  then step \short\ref{algoWd:step6}\short\ref{algoWd:step6-e} yields
  $h_{7},...,h_{1}$ at $x^{37},...,x^{31},$ respectively, all using
  $E_{1}$. Step \short\ref{algoWd:step7} determines $g_{1}$, $g_{2}$,
  $g_{3}$, and step \short\ref{algoWd:step8} checks whether indeed
  $f=g \circ h$, and if so, returns $(g,h)$.

  With the same values, except that $\degTwo=3$, we have
  \begin{align*}
    h^{5}+ g_{3} h^{3}  &= x^{50} + h_{9}^{5} x^{45 } + h_{8}^{5} x^{40}+ h_{7}^{5} x^{35}\\
    &\quad + (h_{6}^{5} + g_{3}) x^{30} + 3 g_{3 }h_{9} x^{29}+ g_{3}
    (3h_{9}^{2}+3h_{8}) x^{28} + x^{26} \cdot O(x)\\
    &\quad + (h_{5}^{5}+ g_{3}(3h_{5}+ 3 h_{9} h_{6}+ 3 h_{8}h_{7} + 3
    h_{9}^{2} h_{7} + 3 h_{9}h_{8}^{2} ))x^{25 } + O(x^{24}).
  \end{align*}

  Assuming that $f_{29}=3g_{3}h_{9}\neq 0$, the algorithm computes
  $j=29$, $\degTwo=(29+1)/10$, $i_{0}=5 \in \mathbb{N}$, goes to step
  \short\ref{algoWd:step6}, determines $h_{9}$ at $x^{45}$,
  $g_{3}$ at $ x^{29}$, $h_{8}$, $h_{7}$, $h_{6}$ according to
  $E_{2}$, then $h_{5}$ at $x^{25}$ via the known value for $h_{5}^{5}
  + 3 g_{3} h_{5}$ in step
  \short\ref{algoWd:step6}\short\ref{algoWd:step6-d} with
  $E_{3}$. Condition \ref{eq:pc} below requires that
  $(-3g_{3})^{(q-1)/4} \neq 1$ and guarantees that $h_{5}$ is uniquely
  determined, as shown in the proof of \ref{thm:kg} below.  Finally
  $h_{4},...,h_{1}$ and $g_{1}, g_{2}$ are computed.

  As a last example, we take $p=5$, $\dg = 25$, $\degOne=r=m=5$ and
  $\degTwo = 4$, so that $a=1$ and
  $$
  h^{5}+g_{4}h^{4}= x^{25}+ (h_{4}^{5} + g_{4}) x^{20}+
  4g_{4}h_{4}x^{19} + O(x^{18}).
  $$
  Again we assume $f_{19}= 4g_{4}h_{4} \neq 0$. Then steps
  \short\ref{algoWd:step1} and \short\ref{algoWd:step2} determine
  $j=19$, $\degTwo=4$, and $i_{0}=15/4 \not\in \mathbb{N}$. We have
  $\degTwo m=20=\dg-r$, so that we go to step
  \short\ref{algoWd:step5}. In step
  \short\ref{algoWd:step5}\short\ref{algoWd:step5-a}, we have to solve
  \ref{eq:as}. The number of solutions is discussed starting with
  \ref{bluher} below. We consider two special cases, namely $q=5$ and
  $q=125$. For $q=5$, we have $25$ pairs $(v, w)=(f_{20}, f_{19})\in
  \mathbb{F}_{5}^{2}$ to consider, with $w \neq 0$.  When $v \neq 0$,
  then the number of solutions is
  \begin{align*}
    \begin{cases}
      2 & \text{if } wv^{-2} \in \{2,0\},\\
      1 & \text{if } wv^{-2} = 1,\\
      0 & \text{otherwise},
    \end{cases}
  \end{align*}
  and when $v=0$:
  \begin{align*}
    \begin{cases}
      2 & \text{for the squares } w=1,4,\\
      0 & \text{otherwise}.
    \end{cases}
  \end{align*}

  Over $\mathbb{F}_{125}$, we have the following numbers of nonzero
  solutions $s$ when $v \neq 0$:
  \begin{align*}
    \begin{cases}
      6 & \text{for } 1 \cdot 124 \text{ values} ~(v,w),\\
      2 & \text{for } 47 \cdot 124 \text{ values}~ (v,w),\\
      1 & \text{for } 25 \cdot 124 \text{ values} ~(v,w),\\
      0 & \text{for } 52 \cdot 124 \text{ values}~ (v,w),
    \end{cases}
  \end{align*}
  and when $v=0$:

  \begin{align*}
    \begin{cases}
      2 & \text{for } 62 \text{ values of } w, \text{ namely the squares},\\
      0 & \text{for } 62 \text{ values of } w.
    \end{cases}
  \end{align*}
  These numbers are explained below. We run the remaining steps in
  parallel for each value $h_{4}= s$ with $s \in S$. This yields
  $g_{4}$ in step \short\ref{algoWd:step5}\short\ref{algoWd:step5-b},
  $h_{3}$, $h_{2}$, $h_{1}$ in step
  \short\ref{algoWd:step5}\short\ref{algoWd:step5-c}, and $g_{1}$,
  $g_{2}$, $g_{3}$ in step \short\ref{algoWd:step7}.
\end{example}

We denote by ${\sf M}(\dg)$ a multiplication time, so that polynomials
of degree at most $\dg$ can be multiplied with ${\sf M}(\dg)$
operations in ${F}$. Then ${\sf M} (\dg)$ is in $O(\dg \log
\dg \loglog \dg)$; see \cite{gatger03}, Chapter 8, and \cite{fue07}
for an improvement.

For an input $f$, we set $\sigma(f) = \#S$ if the precondition of step
\short\ref{algoWd:step5} is satisfied and $S$ computed there, and
otherwise $\sigma(f)=1$.

\begin{theorem}\label{thm:kg}
  Let $f$ be an input polynomial with parameters $\dg$, $p$,
  $q=p^{e}$, $d$, $r$, $a$, $\degOne$, $m$ as specified, $g$, $h$,
  $\degTwo$, $i_{0}$ as in \ref{eq:ghf} and \ref{eq:ratio}, so that
  $f=g \circ h$, set $c=\gcd(d,e)$ and suppose further that
  \begin{equation}
    \label{eq:pc}
    \text{if } i_{0} \in \mathbb{N}\text{ and } 1 \leq i_{0}<m, 
      \text{ then } (-\degTwo g_{\degTwo}/a)^{(q-1)/(p^{c}-1)} \neq 1.
  \end{equation}
  On input $f$, \ref{algoWd} returns either ``failure'' or a set of at
  most $\sigma(f)$ normal decompositions $(g^{*}, h^{*})$ of $f$, and
  $(g,h)$ is one of them. Except if returned in step
  \short\ref{algoWd:step1}, none of them is a Frobenius
  decomposition. If $F=\mathbb{F}_{q}$ is finite, then the algorithm uses
  $$
  O\bigl( {\sf M}(\dg) \log \degOne \, (m+ \log(\degOne q))\bigr)
  $$
  or $\softO(\dg(m + \log q))$ operations in $\F_{q}$.
\end{theorem}
\begin{proof}
  Since $ r = p^{d} \mid \degOne$, we have $~coeff(h^{\degOne},j)=0$
  unless $r \mid j$. Furthermore $g_{\degTwo}h^{\degTwo}=
  g_{\degTwo}x^{\degTwo m} + \degTwo g_{\degTwo}h_{m-1} x^{\degTwo
    m-1}+O(x^{\degTwo m-2})$ and $\degTwo g_{\degTwo}h_{m-1} \neq 0$,
  so that $j$ from step \short\ref{algoWd:step1} equals $\degTwo m$
  (if $p \nmid m$) or $\degTwo m-1$ (if $p\mid m$). Thus $\degTwo$ is
  correctly determined in step \short\ref{algoWd:step2}. In
  particular, $f$ is not a Frobenius composition.

  We denote by $G$ the set of $(g,h)$ allowed in the theorem.  We
  claim that the equations used in the algorithm involve only
  coefficients of $f$ and previously computed values, and usually have
  a unique solution. It follows that most $f \in
  \gamma_{\dg,\degOne}(G)$ are correctly and uniquely decomposed by
  the algorithm. The only exception to the uniqueness occurs in
  \ref{eq:as}.

  In the remaining steps, we use various coefficients $f_{j}$ for
  $j=(\degTwo-1)m+i$ with $1 \leq i \leq m$ or $j= \dg-r(m-i)$ with
  $i_{0} \leq i < m$. The value $i_{0}$ is defined so that $\dg-r
  (m-i_{0})=(\degTwo-1)m+i_{0}$, and thus
  \begin{equation}\label{eq:i}
    \dg -r(m-i) \geq (\degTwo-1)m+i ~\text{if and only if } i \geq i_{0},
  \end{equation}
  since the first linear function in $i$ has the slope $r > 1$,
  greater than for the second one. Since $i \geq 1$, it follows that
  $j > (\degTwo-1)m$ for all $j$ under consideration.  For the
  low-degree part of $g$ we have
  $$
  \deg ((g-(x^{\degOne}+g_{\degTwo}x^{\degTwo}))\circ h) \leq
  (\degTwo-1)m <j,
  $$
  so that
  $$
  f_{j} = ~coeff (g \circ h,j)=
  ~coeff((x^{\degOne}+g_{\degTwo}x^{\degTwo})\circ h,j)=
  ~coeff(h^{\degOne}+g_{\degTwo}h^{\kappa},j)
  $$
  for all $j$ in the algorithm.

  We have to see that the application of $E_{3}$ in steps
  \short\ref{algoWd:step4}\short\ref{algoWd:step4-b} (where
  $i_{0}=m-1$) and \short\ref{algoWd:step6}\short\ref{algoWd:step6-d}
  (where $m-2\geq i_{0}\geq 1$) always has a unique solution. The
  right hand side of \ref{eq:m-4}, say $as^{r}+\degTwo g_{\degTwo}s$,
  is an $\F_{p}$-linear function of $s$. The equation has a unique
  solution if and only if its kernel is $\{0\}$. (\citealp{seg64},
  Teil 1, � 3, and \citealp{wan90a} provide an explicit solution in
  this case.)  But when $s \in \F_{q}$ is nonzero with $as^{r} +
  \degTwo g_{\degTwo}s = 0$, then $-\degTwo g_{\degTwo}/a=
  s^{r-1}$. Writing $z = p^{c}$, so that $z-1=~gcd(q-1,r-1)$, we have
  $$
  (-\degTwo g_{\degTwo}/a)^{(q-1)/(z-1)}=(s^{r-1})^{(q-1)/(z-1)}=
  (s^{(r-1)/(z-1)})^{q-1}=1,
  $$ 
  contradicting the condition \ref{eq:pc}.

  For the correctness it is sufficient to show that all required
  quantities are known, in particular $c_{i,j}
  (g_{\degTwo}x^{\degTwo})$  in $ E_{1}$ and
  $c_{i,j}(x^{\degOne})$ in $ E_{2}$, and that the equations
  determine the coefficient to be computed. We have
  \begin{equation}\label{am}
    \deg (h^{\degOne}-x^{\dg})= \deg((h^{a}-x^{am})^{r}) \leq (am-1)r=\dg-r,
  \end{equation}
  so that $g_{\degTwo}= f_{\degTwo m}$ in steps
  \short\ref{algoWd:step3}\short\ref{algoWd:step3-a} and
  \short\ref{algoWd:step4}\short\ref{algoWd:step4-a}.

  The precondition of step \short\ref{algoWd:step3} implies that for
  all $i<m$ we have
  $$
  (\kappa-1)m\geq\dg-r-m+2>\dg-mr+(r-1)(m-1)\geq\dg-rm+(r-1)i,
  $$
  $$
  n-r(m-i) < (\degTwo-1)m+i.
  $$
  Thus from $E_{1}$ we have with $j=(\degTwo-1)m-i$
  \begin{align*}
    f_{(\degTwo-1)m+i}  &= ~coeff (h^{\degOne}, j) + ~coeff (g_{\degTwo}h^{\degTwo},j)\\
    &= ~coeff ((h^{(i)})^{\degOne}, j)+\degTwo g_{\degTwo}h_{i}
  \end{align*}
  with $\degTwo g_{\degTwo} \neq 0$, so that $h_{i}$ can be computed
  in step \short\ref{algoWd:step3}\short\ref{algoWd:step3-b}.

  The precondition in step \short\ref{algoWd:step4} implies that
  $i_{0}= m-1$, and hence $(r-1)\mid (a-\degTwo)m$. $E_{3}$ says that
  $f_{\degTwo m-1}=c_{m-1, \degTwo
    m-1}(x^{\degOne}+g_{\degTwo}x^{\degTwo})= ah_{m-1}^{r}+ \degTwo
  g_{\degTwo}h_{m-1}$. We have seen above that under our assumptions
  the equation $f_{\degTwo m-1}= as^{r}+\degTwo g_{\degTwo}s$ has
  exactly one solution $s$.  By an argument as for step
  \short\ref{algoWd:step3}\short\ref{algoWd:step3-b}, also step
  \short\ref{algoWd:step4}\short\ref{algoWd:step4-c} works correctly.

  The only usage of $E_{4}$ occurs in step
  \short\ref{algoWd:step5}\short\ref{algoWd:step5-a}, where
  $\degTwo=(n-r)/m= \degOne-r/m.$ Since $p\mid \degOne$, $r$ is a
  power of $p$, and $p \nmid \degTwo$, this implies that $r=m$ and
  $\degTwo=\degOne-1$.  We have from $E_{4}$
  \begin{align*}
    f_{\degTwo m} &= ah_{m-1}^{r}+g_{\degTwo},\\
    f_{\degTwo m-1} &= -g_{\degTwo}h_{m-1}= -(f_{\degTwo
      m}-ah_{m-1}^{r})h_{m-1}= ah_{m-1}^{r+1}- f_{\degTwo m}h_{m-1}.
  \end{align*}
  Thus $h_{m-1} \in S$ as computed in step
  \short\ref{algoWd:step5}\short\ref{algoWd:step5-a} and $g_{\degTwo}$
  is correctly determined in step
  \short\ref{algoWd:step5}\short\ref{algoWd:step5-b}. The precondition
  of step \short\ref{algoWd:step5} implies that $i_{0}= m-1-1/(r-1)$,
  which is an integer only for $r=2$. In that case, $i_{0}=m-2=0$ and
  no further $h_{i}$ is needed. Otherwise, $m-2 < i_{0}<m-1$ and step
  \short\ref{algoWd:step5}\short\ref{algoWd:step5-c} works correctly
  since $i < i_{0}$.

  The precondition of step \short\ref{algoWd:step6} implies that
  $i_{0} < m-1$. If $r \nmid m$, then $~coeff(h^{\degOne},\degTwo
  m)=0$ by $E_{2}$, and otherwise $~coeff(h^{\degOne}, \degTwo
  m-1)=0$. Thus $g_{\degTwo}$ is correctly computed in step
  \short\ref{algoWd:step6}\short\ref{algoWd:step6-b}. Correctness of
  the remaining steps follows as above.

  For the cost of the algorithm over $F=\mathbb{F}_{q}$, two contributions are from
  calculating $(h^{(j)})^{\degTwo}$ for some $j < m$ and the various
  $r$th roots. The first comes to $O (m \cdot \log \degTwo \cdot {\sf
    M}(\dg))$ and the second one to $O (m \cdot \log_{p}q)$ operations
  in $\mathbb{F}_{q}$. $E_{3}$ and $E_{4}$ are applied at most
  once. We then have to find all roots of a univariate polynomial of
  degree at most $r+1$. This can be done with $O({\sf M}(r) \log r
  \log rq)$ operations (see \cite{gatger03}, Corollary 14.16). The
  Taylor coefficients in step \short\ref{algoWd:step7} can be
  calculated with $O({\sf M}(\dg) \log \degOne)$ operations (see
  \cite{gatger03}, Theorem 9.15).  All other costs are dominated by
  these contributions, and we find the total cost as
  $$
  O \bigl({\sf M}(\dg) \log \degOne \cdot (m+\log(\degOne
  q))\bigr).\qed
  $$
\end{proof}
A more direct way to compute $h$ (say, in step 3) is to consider its
reversal as the $\kappa$th root of the reversal of
$(f-h^{k})/g_{\kappa}$, feeding the contribution of $h^{k}$ into the
Newton iteration as in \cite{gat90c}. I have not analyzed this
procedure.

Our next task is to determine the number $N$ of decomposable $f$
obtained as $g \circ h$ in \ref{thm:kg}. Since \ref{eq:as} is an
equation of degree $r+1$, it has at most $r+1$ solutions, and $\sigma
(f) \leq r+1$. $N$ is at least the number of $(g,h)$ permitted by
\ref{thm:kg}, divided by $r+1$. The following considerations lead to a
much better lower bound on $N$.

In the following we write, as usually, $p = ~char \mathbb{F}_{q}$, and
also
\begin{equation}\label{eq:N}
  q = p^{e}, r=p^{d}, c= \gcd(d,e), z=p^{c},
\end{equation}
so that $\mathbb{F}_{q} \cap \mathbb{F}_{r} = \mathbb{F}_{z}$
(assuming an embedding of $\mathbb{F}_{q}$ and $\mathbb{F}_{r}$ in a
common superfield) and $\gcd(q-1,r-1)= z-1$ (see \ref{lem:even}). We
have to understand the number of solutions $s$ of \ref{eq:as}, in
other words, the size of
$$
S(v,w)= \{s \in \mathbb{F}^{\times}_{q} \colon s^{r+1}-vs-w=0\}
$$
for $v=f_{\degTwo m}/a$, $w=f_{\degTwo m-1}/a \in
\mathbb{F}_{q}$. \ref{eq:as} is only used in step
\short\ref{algoWd:step5}, where $m=r$, as noted above. We have
$\degTwo=(j+1)/m$ in step \short\ref{algoWd:step2} and hence
$f_{\degTwo m-1} \neq 0$ and $w\neq 0$.  Furthermore, we define for $u
\in \mathbb{F}_{q}$
\begin{equation}\label{eq:Tu}
  T(u) = \{t \in \mathbb{F}^{\times}_{q} \colon t^{r+1}-ut+u=0\}.
\end{equation}

In \ref{eq:as}, we have $w \neq 0$, but $v$ might be zero. In order to
apply a result from the literature, we first assume that also $v$ is
nonzero, make the invertible substitution $s = -v^{-1}wt$, and set
$u=v^{r+1}(-w)^{-r} = -v^{r+1}w^{-r} \in \mathbb{F}_{q}$. Then $u \neq
0$ and
\begin{align}\label{eq:w1}
  s^{r+1} -vs-w &= (-v^{-1}w)^{r+1}(t^{r+1}-ut+u),\\
  \#S(v,w) &= \#T(u).\nonumber
\end{align}

This reduces the study of $ S(v,w) $, with two parameters, to the
one-parameter problem $T(u)$.  The polynomial $t^{r+1}-ut+u$ has
appeared in other contexts such as the inverse Galois problem,
difference sets, and M\"uller-Cohen-Matthews polynomials. \cite{blu04a}
has determined the combinatorial properties that we need here; see her
paper also for further references. \citeauthor{blu04a} allows an
infinite ground field $F$, but we only use her results for $F=\mathbb{F}_{q}$.

For $i \geq 0$, let
\begin{equation}\label{eq:c}
  \begin{aligned}
    C_{i}  &= \# \{u \in \mathbb{F}^{\times}_{q} \colon \#T (u)=i\},\\
    c_{i} &=\#C_{i}.
  \end{aligned}
\end{equation}
Then $C_{i}= \varnothing$ for $i>r+1$. \cite{blu04a} completely
determines these $c_{i}$, as follows.
\begin{fact}\label{bluher}
  With the notations \ref{eq:N} and \ref{eq:c}, let $I= \{0,1,2,
  z+1\}$. Then
  \begin{equation}\label{eq:c1}
    \begin{aligned}
      c_{1}  &= \frac{q}{z}-\gamma,\\
      c_{i}  &= 0 ~\text{unless } i \in I,\\
      c_{z+1}  &= \left\lfloor \frac{q}{z^{3}-z}\right\rfloor,
    \end{aligned}
  \end{equation}
  where
  \begin{equation}\label{eq:ec}
    \gamma = 
    \begin{cases}
      1 & \textrm{if } q \textrm{ is even and }e/c\textrm{ is odd },\\
      0  & \textrm{otherwise},\\
    \end{cases}
  \end{equation}
  and furthermore
  \begin{equation}\label{eq:c2}
    q = 1+ \sum_{i \in I}c_{i}= 2+\sum_{i \in I} ic_{i}.
  \end{equation}
\end{fact}
\begin{proof}
  The claims are shown in \cite{blu04a}, Theorem 5.6. Her statement
  assumes $tu \neq 0$, which is equivalent to our assumption $t \neq
  0$. \ref{eq:c2} corresponds to the fact that the numbers $c_{i}$
  form the preimage statistic of the map from $\mathbb{F}_{q}
  \smallsetminus \{0,1\}$ to $\mathbb{F}_{q}\smallsetminus \{0\}$
  given by the rational function $x^{r+1}/(x-1)$.
\end{proof}

\ref{eq:c1} and \ref{eq:c2} also determine the remaining two values
$c_{0}$ and $c_{2}$, namely $c_{2}=
\frac{1}{2}(q-2-c_{1}-(z+1)c_{z+1})$ and $c_{0}= 1+
c_{2}+zc_{z+1}$. For large $z$, we have
$$
c_{2} \approx \frac{q}{2}(1-\frac{1}{z}- \frac{z+1}{z^{3}-z}) =
\frac{q}{2}(1-\frac{1}{z-1}) \approx \frac{q}{2}.
$$
Thus $x^{r+1}/(x-1)$ behaves a bit like squaring: about half the
elements have two preimages, and about half have none.

For the case $v=0$, we have the following facts, which are presumably
well-known. For an integer $m$, we let the integer $\nu(m)$ be the
multiplicity of $2$ in $m$, so that $m = 2^{\nu(m)}m^{*}$ with an odd
integer $m^{*}$.
\begin{lemma}\label{lem:even}
  Let $\mathbb{F}_{q}$ have characteristic $p$ with $q=p^{e}$,
  $r=p^{d}$ with $d \geq 1$, $b = \gcd (q-1, r+1)$ and $w \in
  \mathbb{F}_{q}^{\times}$. Then the following hold.
  \begin{enumerate}
  \item\label{lem:even-1-1}
    \begin{equation*}
      \#S(0,w) = 
      \begin{cases}
        b & \textrm{if } w^{(q-1)/b}=1,\\
        0 & \textrm{otherwise}.
      \end{cases}
    \end{equation*}
  \item\label{lem:even-1-2} We let $c = \gcd (d,e)$, $z=p^{c}$,
    $\delta = \nu(d)$, $\epsilon = \nu(e)$, $\alpha = \nu(r^{2}-1)$,
    $\beta = \nu(q-1)$,
    \begin{align*}
      \lambda =
      \begin{cases}
        2 & \text{if } \delta < \epsilon,\\
        1 & \text{if } \delta \geq \epsilon,
      \end{cases}
    \end{align*}
    \begin{align*}
      \mu =
      \begin{cases}
        1 & \text{if } \alpha > \beta,\\
        0 & \text{if } \alpha \leq \beta.
      \end{cases}
    \end{align*}
    Then $\gcd (r-1, q-1)= z-1$ and
    \begin{equation*}\label{lem:even-1}
      b = \frac{(z^{\lambda }-1)\cdot 2^{\mu}}{z-1}=
      \begin{cases}
        2(z+1) & \textrm{if } \delta < \epsilon ~\textrm{and }
        \alpha > \beta,\\
        z+1 & \textrm{if } \delta < \epsilon ~\textrm{and } \alpha
        \leq \beta,\\
        2  & \textrm{if } \delta \geq \epsilon ~\textrm{and } \alpha > \beta,\\
        1 & \textrm{if }\delta \geq \epsilon ~\textrm{and } \alpha
        \leq \beta.
      \end{cases}
    \end{equation*}
  \item\label{lem:even-3} If $p$ is odd, then $\alpha > \beta$ if and
    only if $e/c$ is odd.
  \end{enumerate}
\end{lemma}
\begin{proof}
  \short\ref{lem:even-1-1} The power function $y \mapsto y^{r+1}$ from
  $\mathbb{F}_{q}^{\times}$ to $\mathbb{F}_{q}^{\times}$ maps $b$
  elements to one, and its image consists of the $u \in
  \mathbb{F}_{q}$ with $u^{(q-1)/b}=1$.

  \short\ref{lem:even-1-2} For the first claim that
  \begin{equation}\label{eq:claim}
    \gcd (q-1, r-1)= z-1,
  \end{equation}
  we may assume, by symmetry, that $d > e$ and let $d = ie +j$ be the
  division with remainder of $d$ by $e$, with $0 \leq j < e$. Then for
  \begin{align*}
    a = \frac{x^{j}(x^{d-j}-1)}{x^{e}-1}= x^{j} \cdot
    \frac{x^{ie}-1}{x^{e}-1}\in \mathbb{Z}[x],
  \end{align*}
  we have
  \begin{align*}
    x^{d}-1 = a \cdot (x^{e}-1)+(x^{j}-1).
  \end{align*}
  By induction along the Extended Euclidean Algorithm for $(d,e)$ it
  follows that all quotients in the Euclidean Algorithm for $(x^{d}-1,
  x^{e}-1)$ in $\mathbb{Q}[x]$ are, in fact, in $\mathbb{Z}[x]$, hence
  also the B{\'{e}}zout coefficients, and that all remainders are of the
  form $x^{y}-1$, where $y$ is some remainder for $d$ and $e$. For $c=
  \gcd(d,e)$, there exist $u$, $v$, $s$, $t \in \mathbb{Z}[x]$ so that
  \begin{align*}
    u \cdot (x^{c}-1)  &= x^{d}-1,\\
    v \cdot(x^{c}-1)  &= x^{e}-1,\\
    s \cdot(x^{d}-1)+t \cdot(x^{e}-1)  &= x^{c}-1.
  \end{align*}
  Substituting any integer $q$ for $x$ into these equations shows the
  claim \ref{eq:claim}.

  We note that $\gcd(2d,e)= \lambda c$ and
  \begin{equation*}
    \gcd(p^{d}-1, p^{d}+1) =
    \begin{cases}
      2 & \text{if } p\text{ is odd},\\
      1 & \text{if } p\text{ is even}.
    \end{cases}
  \end{equation*}

  When $p$ is even, then applying \ref{eq:claim} to $q=p^{e}$ and
  $r^{2}=p^{2d}$, we find
  \begin{align*}
    p^{\lambda c}-1  &= \gcd((p^{d}-1)(p^{d}+1), p^{e}-1)\\
    &= \gcd (p^{d}-1, p^{e}-1)\cdot \gcd (p^{d}+1, p^{e}-1)\\
    &= (p^{c}-1)\cdot b,\\
    b &= \frac{p^{\lambda c}-1}{p^{c}-1} =
    \begin{cases}
      z+1 & \text{if } \delta < \epsilon,\\
      1 & \text{if } \delta \geq \epsilon.
    \end{cases}
  \end{align*}
  For odd $p$, the second equation above is still almost correct,
  except possibly for factors which are powers of $2$. We note that
  exactly one of $\nu(p^{d}-1)$ and $\nu(p^{d}+1)$ equals $1$, and
  \begin{align*}
    p^{\lambda c}-1  &=  \gcd((p^{d}-1)(p^{d}+1), p^{e}-1)\\
     &= \gcd(p^{d}-1, p^{e}-1)\cdot \gcd(p^{d}+1, p^{e}-1)\cdot
    2^{-\mu}\\
     &= (p^{c}-1) \cdot b \cdot 2^{-\mu},\\
    b  &= \frac{(p^{\lambda c}-1)\cdot 2^{\mu}}{p^{c}-1}.
  \end{align*}

  \short\ref{lem:even-3} We define the integers $k_{q}$ and $k_{r}$ by
  \begin{align*}
    \frac{q-1}{z-1} &= \frac{z^{e/c}-1}{z-1} = z^{e/c-1}+ \cdots + 1
    =
    k_{q},\\
    \frac{r^{2}-1}{z-1} &= \frac{(r+1)(z^{d/c}-1)}{z-1}=
    (r+1)(z^{d/c-1}+ \cdots +1)= (r+1)k_{r}.
  \end{align*}
  Now $r+1$ is even and $z$ is odd. If $e/c$ is odd, then $k_{q}$ is
  odd and hence $\alpha > \beta$. Now assume that $e/c$ is even. Then
  $d/c$ is odd, and so is $k_{r}$. Hence $\nu(r-1)= \nu (z-1)$, and we
  denote this integer by $\gamma$. If $\gamma \geq 2$, then
  $\nu(r+1)=1$ and $\alpha = \nu(r+1)+\gamma \leq
  \nu(k_{q})+\gamma=\beta$.

  Now suppose that $\gamma = 1$, and let $\tau = \nu (z+1)$ and
  $m=(z+1)\cdot 2^{-\tau}$. Then $\tau \geq 2$, $m$ is an odd integer,
  and
  \begin{align*}
    z^{2}  &= (m2^{\tau}-1)^{2} \equiv -2 \cdot 2^{\tau}+1 \equiv
    2^{\tau+1}+1 \bmod 2^{\tau+2},\\
    r^{2} &= (z^{2})^{d/c} = (2^{\tau+1}+1)^{d/c} \equiv 2^{\tau+1}+1
    \bmod 2^{\tau+2},\\
    q  &= (z^{2})^{e/2c} \equiv (2^{\tau+1}+1)^{e/2c}\bmod 2^{\tau+2}.
  \end{align*}
  The last value equals $2^{\tau+1}+1$ or $1 \text{ modulo
  }2^{\tau+2}$ if $e/2c$ is odd or even, respectively. In either case,
  it follows that $\alpha = \nu (r^{2}-1)= \tau+1 \leq \nu (q-1)=
  \beta$.
\end{proof}
\begin{theorem}\label{th:decom}
  Let $\F_{q} $ have characteristic $p$ with $q=p^{e}$, and take
  integers $d \geq 1$, $r=p^{d}$, $\degOne=ar$ with $p\nmid a$, $m
  \geq2$, $\dg=\degOne m$, $c= \gcd(d,e)$, $z=p^{c}$, $\mu =
  \gcd(r-1,m)$, $r^{*}=(r-1)/\mu$, and let $G$ consist of the $(g,h)$
  as in \ref{thm:kg}. Then we have the following lower bounds on the
  cardinality of $\gamma_{\dg,\degOne}(G)$.
  \begin{enumerate}
  \item\label{th:decom-1} If $r\neq m$ and $\mu=1$:
    $$
    q^{\degOne+m-2}(1-q^{-1}(1+q^{-p+2}
    \frac{(1-q^{-1})^{2}}{1-q^{-p}})) (1-q^{-\degOne}),
    $$
  \item\label{th:decom-2} If $r\neq m$:
    $$
    q^{\degOne+m-2}\bigl((1-q^{-1}(1+q^{-p+2}
    \frac{(1-q^{-1})^{2}}{1-q^{-p}}))(1-q^{-\degOne})
    $$
    $$
    -q^{-\degOne-r^{*}-c/e+2}\frac{(1-q^{-1})^{2}(1-q^{-r^{*}(\mu-1)})}{(1-q^{-c/e})
      (1-q^{-r^{*}})}(1+q^{-r^{*}(p-2)}) \bigr).
    $$ 
  \item\label{th:decom-3} If $r=m$:
    $$
    q^{\degOne+m-2} (1-q^{-1})(\frac 1 2 + \frac{1+q^{-1}}{2z+2}
    +\frac{q^{-1}} 2 -q^{-\degOne} \frac{1-q^{-p+1}}{1-q^{-p}} -
    q^{-p+1}\frac{1-q^{-1}} {1-q^{-p}}).
    $$
  \end{enumerate}
\end{theorem}
\begin{proof}
  We have seen at the beginning of the proof of \ref{thm:kg} that
  steps \short\ref{algoWd:step1} and \short\ref{algoWd:step2}
  determine $j$ and $\degTwo$. We also know that, given $g_{\degTwo}$
  and $h_{m-1}$, the remaining coefficients of $g$ and $h$ are
  uniquely determined by those of $f$.

  We count the number of compositions $g \circ h$ according to the
  four mutually exclusive conditions in steps \short\ref{algoWd:step3}
  through \short\ref{algoWd:step6}, for a fixed $\degTwo$. The
  admissible $\degTwo$ are those with $1 \leq \degTwo < \degOne$ and
  $p \nmid \degTwo$. $E_{3}$ or $E_{4}$ are used if and only if either
  $i_{0} \in \mathbb{N}$ or $\degTwo m = \dg-r$, respectively. If
  neither happens, then the number of $(g,h)$ is
  \begin{equation}\label{eq:qq}
    q^{\degTwo}(1-q^{-1})\cdot q^{m-1}(1-q^{-1})= q^{\degTwo+m-1}(1-q^{-1})^{2}.
  \end{equation}
  $E_{3}$ is used if and only if $\degTwo \in K$, where
  \begin{align*}
    K = \{\degTwo \in \mathbb{N} \colon 1 \leq \degTwo < \degOne, p
    \nmid \degTwo, i_{0} \in \mathbb{N}, 1 \leq i_{0} < m\},
  \end{align*}
  which corresponds to steps
  \short\ref{algoWd:step4}\short\ref{algoWd:step4-b} (where
  $i_{0}=m-1$) and \short\ref{algoWd:step6}\short\ref{algoWd:step6-d}
  (where $i_{0} \in \mathbb{N}$ and $1 \leq i_{0} \leq m-2$).  For
  $\degTwo \in K$, we have the condition \ref{eq:pc} that $(-\degTwo
  g_{\degTwo}/a)^{(q-1)/(z-1)} \neq 1$.  The exponent is a divisor of
  $q-1$, and there are exactly $(q-1)/(z-1)$ values of $g_{\degTwo}$
  that violate \ref{eq:pc}. Thus for $\kappa \in K$ the number of
  $(g,h)$ equals
  \begin{equation}
    \label{eq:q-1}
    (q-1-\frac{q-1}{z-1}) q^{\degTwo-1} \cdot q^{m-1}(1-q^{-1})= q^{\degTwo+m-1}(1-\frac{1}{z-1})(1-q^{-1})^{2}.
  \end{equation}

  The only usage of $E_{4}$ occurs in step
  \short\ref{algoWd:step5}\short\ref{algoWd:step5-a}, where
  $\degTwo=(n-r)/m= \degOne-r/m$. We have seen in the proof of
  \ref{thm:kg} that this implies $r=m$ and $\degTwo=\degOne-1$. We
  split $G$ according to whether $\degTwo=\degOne-1$ or
  $\degTwo<\degOne-1$, setting
  $$
  G^{*}=\{(g,h) \in G \colon \degTwo =\degOne -1 ~\text{in
    \ref{eq:ghf}}\}.
  $$

  We define three summands $S_{12}$, $S_{3}$, and $S_{4}$ according to
  whether only $E_{1}$ and $E_{2}$, or also $E_{3}$, or $E_{4}$ are
  used, respectively:
  \begin{align*}
    S_{12} &= \sum_{\substack{1 \leq \degTwo < \degOne\\p \nmid
        \degTwo}}q ^{\degTwo+m-1}
    (1-q^{-1})^{2},\\
    S_{3} &= \sum_{\degTwo \in K}(q^{\degTwo+m-1}(1-q^{-1})^{2}
    -q^{\degTwo+m-1}(1-q^{-1})^{2}(1-\frac{1}{z-1})),\\
    S_{4}&=q^{\degOne+m-2}(1-q^{-1})^{2}-\#\gamma_{\dg,\degOne}(G^{*}).
  \end{align*}
  We will see below that $K= \varnothing$ if $r=m$. Thus
  \begin{align*}
    \#\gamma_{\dg,\degOne}(G) \geq
    \begin{cases}
      S_{12} & \text{if } r \neq m \text{ and } K = \varnothing,\\
      S_{12}-S_{3} & \text{if } r\neq m,\\
      S_{12}-S_{4} & \text{if } r=m.
    \end{cases}
  \end{align*}
  The subtraction of $S_{3}$ corresponds to replacing the summand
  \ref{eq:qq} by \ref{eq:q-1} for $\degTwo \in K$. Similarly, $S_{4}$
  replaces \ref{eq:qq} for $\degTwo=\degOne-1$ by the correct value if
  $E_{4}$ is applied.

  Since $p\mid \degOne$, the first sum equals
  \begin{align*}
    S_{12} &= q^{m-1}(1-q^{-1})^{2}(\sum_{1 \leq \degTwo <
      \degOne}q^{\degTwo}- \sum_{\substack{1 \leq \degTwo < \degOne
        \\
        p\mid \degTwo}}q^{\degTwo}) \\
    &= q^{m-1} (1-q^{-1})^{2}(\frac{q^{\degOne}-1}{q-1}-1-
    \frac{(q^{p})^{\degOne/p}-1}{q^{p}-1}+1)
    \\
    &=
    q^{\degOne+m-2}(1-q^{-1})(1-q^{-\degOne})\frac{1-q^{-p+1}}{1-q^{-p}}\\
    &= q^{\degOne+m-2}(1-q^{-1}(1+q^{-p+2}
    \frac{(1-q^{-1})^{2}}{1-q^{-p}}))(1-q^{-\degOne}).
  \end{align*}

  For $S_{3}$, we describe $K$ more transparently.  From
  \ref{eq:ratio} we find
  \begin{equation}
    \begin{aligned}\label{eq:K-1}
      1  \leq i_{0} &= \frac{\degTwo m-\dg}{r-1}+m \leq m-1\\
      & \Longleftrightarrow \degOne-(r-1)+ \frac{r-1}{m}\leq \degTwo
      \leq \degOne - \frac{r-1}{m},
    \end{aligned}
  \end{equation}
  \begin{align}
    \label{eq:K-2} i_{0} \in \mathbb{Z}& \Longleftrightarrow (r-1)\mid
    (\degTwo-a)m.
  \end{align}
  We have $\mu = \gcd (r-1,m)$ and $r^{*}= (r-1)/\mu$, and set
  $m^{*}=m/\mu$, so that $\gcd(r^{*}, m^{*})= 1$ and
  \begin{align*}
    \ref{eq:K-1}& \Longleftrightarrow \degOne- (r-1) +
    \frac{r^{*}}{m^{*}}
    \leq \degTwo \leq \degOne - \frac{r^{*}}{m^{*}},\\
    \ref{eq:K-2}& \Longleftrightarrow r^{*}\mid (\degTwo-a)m^{*}
    \Longleftrightarrow r^{*} \mid (\degTwo-a).
  \end{align*}
  
  Since $r^{*}\mid \degOne-a=a(r-1)$, we have
  \begin{align}\label{eq:k-a}
    \ref{eq:K-2} & \Longleftrightarrow  \exists j \in \mathbb{Z} \quad \degTwo=\degOne-(r-1)+jr^{*},\\
    \label{eq:a+j}
    \ref{eq:K-1}& \Longleftrightarrow \frac{1}{m^{*}} \leq j \leq
    \frac{r-1}{r^{*}}- \frac{1}{m^{*}} \Longleftrightarrow 1 \leq j
    \leq \mu-1.
  \end{align}
  Since $\mu \mid (r-1)$ and $r=p^{d}$, we have $p \nmid \mu$. Thus
  \begin{align}\label{modp}
    p \mid \degTwo & \Longleftrightarrow 1- \frac{j}{\mu} \equiv 1+
    \frac{j(r-1)}{\mu} \equiv \degOne-(r-1)+jr^{*}= \degTwo \equiv 0 \bmod p\\
    & \Longleftrightarrow j \equiv \mu \bmod p \Longleftrightarrow
    \exists
    i \in \mathbb{Z} \quad j = \mu - ip,\nonumber\\
    \ref{eq:K-1} & \Longleftrightarrow 1 \leq j = \mu - ip \leq \mu-1
    \Longleftrightarrow 1 \leq i \leq
    \lfloor\frac{\mu-1}{p}\rfloor.\nonumber
  \end{align}

  Abbreviating $\mu^{*}= \lfloor(\mu -1)/p \rfloor$, it follows that
  \begin{align*}
    K = \{\degOne-(r-1) +jr^{*} \colon 1 \leq j \leq \mu -1\}
    \smallsetminus \{\degOne-ipr^{*} \colon 1 \leq i \leq \mu^{*}\}.
  \end{align*}

  In particular, we have $K = \varnothing$ if $\mu = 1$.  Assuming
  $\mu \geq 2$ and using $z=p^{c}= q^{c/e}$, we can evaluate $S_{3}$
  as follows.
  \begin{align*}
    S_{3}&= \sum_{\degTwo \in K} \frac{q^{\degTwo+m-1}}{z-1}(1-q^{-1})^{2}\\
    &= \frac{q^{m-1}(1-q^{-1})^{2}}{z-1} \sum_{\degTwo \in K}q^{\degTwo} \\
    &= \frac{q^{m-1} (1-q^{-1})^{2}} {z-1}(q^{\degOne-(r-1)+r^{*}}
    \frac{(q^{r^{*}})^{\mu-1}-1} {q^{r^{*}}-1} -q^{\degOne-pr^{*}}
    \frac{(q^{-pr^{*}})^{\mu^{*}}-1 }{q^{-pr^{*}}-1})\\
    &=q^{\degOne+m-1-r^{*}-c/e}\frac{(1-q^{-1})^{2}(1-q^{-r^{*}(\mu-1)})}{(1-q^{-c/e})(1-q^{-r^{*}})}\\
    & \quad\cdot (1-q^{-r^{*}(p-1)}
    \frac{(1-q^{-r^{*}})(1-q^{-pr^{*}\mu^{*}}) }{(1-q^{-r^{*}(\mu-1)})
      (1-q^{-pr^{*}})})\\
   &\leq
    q^{\degOne+m-1-r^{*}-c/e}\frac{(1-q^{-1})^{2}(1-q^{-r^{*}(\mu-1)})}
    {(1-q^{-c/e}) (1-q^{-r^{*}})}.
  \end{align*}

  In order to evaluate $S_{4}$, we first recall from the above that we
  have $\degTwo m=\dg-r$, $\degTwo=\degOne-1$, $m=r$, and any $(g,h)
  \in G^{*}$ is uniquely determined by $f=g \circ h$, $g_{\degOne-1}$,
  and $h_{m-1}$. To any $(g,h) \in G^{*}$, we associate the field
  elements
  \begin{equation}
    \begin{aligned}
      V(g,h)&= h^{r}_{m-1} +g_{\degOne-1}/a,\\
      W(g,h) &= -g_{\degOne-1}h_{m-1}/a,\\
      U(g,h) &= -V(g,h)^{r+1} W(g,h)^{-r}.
    \end{aligned}
  \end{equation}
  Then if $f = g \circ h$, we have $aV(g,h)=f_{n-r}$, $aW(g,h)=
  f_{n-r-1} \neq 0$, and for nonzero $s \in \mathbb{F}_{q}$ and
  $t=-V(g,h)\cdot W(g,h)^{-1}s$, \ref{eq:w1} says that
  $$
  \ref{eq:as} \text{ holds } \Longleftrightarrow s \in S(V(g,h),
  W(g,h)) \Longleftrightarrow t \in T(U(g,h)).
  $$
  We recall the sets $C_{i}$ from \ref{eq:c} and for $i \in
  \{1,2,z+1\}$, we set
  \begin{align*}
    G_{i} &= \{(g,h)\in G \colon V(g,h) \neq 0, U(g,h) \in
    C_{i}\},\\
    G_{0} &= \{(g,h)\in G \colon V(g,h)=0\}.
  \end{align*}

  Now let $v \in \mathbb{F}_{q}^{\times}$, $i \in \{1,2,z+1\}$, $u \in
  C_{i}$, and $g_{\degOne-2}$, $\ldots$, $g_{1}$, $h_{m-2}$, $\ldots$,
  $h_{1} \in \mathbb{F}_{q}$. From these data, we construct $(g,h)\in
  G_{i}$ with $g = \sum_{1 \leq i \leq \degOne} g_{i}x^{i}$ and $h =
  \sum_{1 \leq i \leq m} h_{i}x^{i}$ and $g_{\degOne}= h_{m}=1$, so
  that only $g_{\degOne-1}$ and $h_{m-1}$ still need to be determined.
  Furthermore, if $f=g \circ h$, we show that different data lead to
  different $f$. This will prove that
  \begin{equation}\label{data}
    \gamma_{\dg,\degOne}(G_{i}) \geq (q-1)c_{i} \cdot q^{\degOne+m-4}.
  \end{equation}
  By assumption, we have $u \neq 0$ and $\#T(u)=i \geq 1$.  We choose
  some $t \in T(u)$ and define $w,s \in \mathbb{F}_{q}^{\times}$ by
  \begin{align*}
    w^{r} &= -v^{r+1}u^{-1},\\
    s &= -v^{-1}wt.
  \end{align*}
  Then $s \in S(v,w)$ by \ref{eq:w1}. We set $h_{m-1}=s$ and
  $g_{\degOne-1} = av-as^{r}$. Now $g$ and $h$ are determined, and
  $E_{1}$ and $E_{2}$ imply that
  \begin{align*}
    f_{n-r} &= ah^{r}_{m-1}+ g_{\degTwo} = aV(g,h)=av,\\
    f_{n-r-1}  &= -g_{\degTwo}h_{m-1} = aW (g,h)= -a(v-s^{r})s = a(s^{r+1}-vs)=aw,\\
    U(g,h)&= - v^{r+1}w^{-r} = -v^{r+1}(-v^{r+1}u^{-1})^{-1} =u.
  \end{align*}
  Suppose that $(u,v)$ and $(\tilde{u}, \tilde{v})$ lead to $(f_{n-r},
  f_{n-r-1})=(av,aw)$ and $(\widetilde{f_{n-r}},
  \widetilde{f_{n-r-1}})= (a\tilde{v}, a\tilde{w})$, and that the
  latter pairs are equal. Then $v = \tilde{v}$ and $u = -v^{r+1}
  w^{-r} = - \tilde{v}^{r+1}\tilde{w}^{-r}= \tilde{u}$.  This
  concludes the proof of \ref{data}.

  A similar argument works for $G_{0}$. We let $b = \gcd(q-1,r+1)$,
  take $w \in \mathbb{F}_{q}$ with $w^{(q-1)/b}=1$, and some $s \in
  \mathbb{F}_{q}$ with $s^{r+1}=w$. There are $(q-1)/b$ such $w$, and
  according to \ref{lem:even-1-1}, $b$ such values $s$ for each $w$.
  We set $h_{m-1}=s$ and $g_{\degOne-1}= -ah^{r}_{m-1}$ and, as above,
  complete them with arbitrary coefficients to $(g,h)\in G_{0}$. When
  $f= g\circ h$, then $f_{\dg-r}=0$ and $f_{n-r-1}= -
  g_{\degOne-1}h_{m-1}= ah_{m-1}^{r+1}= aw = aW(g,h)$, and different
  $w$ lead to different $f$. It follows that
  \begin{align}\label{eq:w2}
    \gamma_{\dg,\degOne}(G_{0}) \geq \frac{q-1}{b}.
  \end{align}

  The images of $G_{1}$, $G_{2}$, $G_{z+1}$, and $G_{0}$ under
  $\gamma_{\dg,\degOne}$ are pairwise disjoint, since the map $V
  \times W \times U \colon \bigcup_{i=0,1,2,z+1}G_{i} \longrightarrow
  \mathbb{F}_{q}^{3}$ is injective, and its value together with the
  lower coefficients of $g$ and $h$ determines $f$, again injectively.
  It follows that
  \begin{align}\label{al:GG}
    \sum_{i=0,1,2,z+1}\#\gamma_{\dg,\degOne}(G_{i}) &\geq
    \sum_{i=1,2,z+1}(q-1)c_{i} \cdot q^{\degOne+m-4}+
    \frac{q-1}{b}\cdot
    q^{\degOne+m-4}\\
    &= (q-1)q^{\degOne+m-4}(\sum_{i=1,2,z+1}c_{i}+
    \frac{1}{b}).\nonumber
  \end{align}
  We write $q=p^{e}$ and set
  \begin{align*}
    z^{*}  &=
    \begin{cases}
      z & \text{if }e/c\text{ is odd},\\
      z^{2} & \text{if }e/c\text{ is even}.
    \end{cases}
  \end{align*}
  \ref{bluher} yields
  $$
  c_{z+1} = \displaystyle \left\lfloor\frac{q}{z^{3}-z}
  \right\rfloor=\frac{q-z^{*}}{z^{3}-z}, 
  $$ 
  \begin{align*}
    2 \sum_{i=1,2,z+1}c_{i}  &=
    2c_{1}+(q-2-c_{1}-(z+1)c_{z+1})+2c_{z+1}\\
    &= q-2+ \frac{q}{z}- \gamma - (z-1)\frac{q-z^{*}}{z^{3}-z}\\
    &= q-2+ \frac{q}{z}- \gamma - \frac{q-z^{*}}{z^{2}+z},
  \end{align*}
  \begin{align*}
    \#\gamma_{\dg,\degOne}(G^{*}) &\geq
    q^{\degOne+m-3}(1-q^{-1})(\frac{1}{2}(q-2+\frac{q}{z}-\gamma
    -\frac{q-z^{*}}{z^{2}+z})+\frac{1}{b}).
  \end{align*}

  We call the last factor $B$. If $e/c$ is odd, then, in the notation
  of \ref{lem:even}, $\delta = \nu (d)\geq \nu(e)= \epsilon$, so that
  $b \in \{1,2\}$, and
  \begin{align*}
    b =
    \begin{cases}
      2 & \text{if }p\text{ is odd},\\
      1 & \text{if }p=2.
    \end{cases}
  \end{align*}
  If $p$ is odd, then $\gamma = 0$ and $2/b-\gamma=1$. If $p=2$, then
  $\gamma=1$ and again $2/b-\gamma = 2-1=1$.  It follows that
  \begin{align*}
    2B &= q-2+\frac{q}{z}-\frac{q-z}{z^{2}+z}+ \frac{2}{b}-\gamma =
    q(1+\frac{1}{z+1}(1-\frac{z}{q})).
  \end{align*}
  If $e/c$ is even, then $\gamma =0$, $b=z+1$ and
  \begin{align*}
    2B &= q-2 + \frac{q}{z}- \frac{q-z^{2}}{z^{2}+z}+ \frac{2}{z+1} =
    q(1+\frac{1}{z+1}(1-\frac{z}{q})).
  \end{align*}
  It follows that in all cases
  \begin{align*}
    \#\gamma_{\dg,\degOne}(G^{*})  &\geq \frac{1}{2}q^{\degOne+m-2}(1-q^{-1})(1+\frac{1}{z+1}(1-\frac{z}{q})), \\
    S_{4} &\leq q^{\degOne+m-2}(1-q^{-1})(1-q^{-1}
    -\frac{1}{2}(1+  \frac{1}{z+1}(1- \frac{z}{q}) ))\\
    &= q^{\degOne+m-2}(1-q^{-1})(\frac{1}{2}-q^{-1}-
    \frac{1}{2z+2}(1- \frac{z}{q}) ).
  \end{align*}

  Together we have found the following lower bounds on
  $\#\gamma_{\dg,\degOne}(G)$.  If $r \neq m$ and $\mu = 1$, then
  \begin{align*}
    \#\gamma_{\dg,\degOne}(G) \geq S_{12} =
    q^{\degOne+m-2}(1-q^{-1}(1+q^{-p+2}
    \frac{(1-q^{-1})^{2}}{1-q^{-p}}))(1-q^{-\degOne}).
  \end{align*}

  If $r \neq m$, then
  \begin{align*}
    \#\gamma_{\dg,\degOne}(G)  &\geq S_{12} - S_{3} \geq
    q^{\degOne+m-2}(1-q^{-1}(1+q^{-p+2}
    \frac{(1-q^{-1})^{2}}{1-q^{-p}}))  (1-q^{-\degOne})\\
    & \quad-
    q^{\degOne+m-\degOne-2r^{*}-c/e}\frac{(1-q^{-1})^{2}(1-q^{-r^{*}(\mu-1)})}{(1-q^{-c/e})(1-q^{-r^{*}})}(1+q^{-r^{*}(p-2)})\\
    &= q^{\degOne+m-2}\bigl((1-q^{-1}(1+q^{-p+2}
    \frac{(1-q^{-1})^{2}}{1-q^{-p}})) (1-q^{-\degOne})\\
    &\quad -q^{-\degOne-r^{*}-c/e+2}\frac{(1-q^{-1})^{2}(1-q^{-r^{*}(\mu-1)})}{(1-q^{-c/e})(1-q^{-r^{*}})}(1+q^{-r^{*}(p-2)})\bigr).
  \end{align*}

  If $r=m$, then
  \begin{align*}
    \#\gamma_{\dg,\degOne}(G)  &\geq S_{12} - S_{4}\geq
    q^{\degOne+m-2}(1-q^{-1})(1-q^{-\degOne})\frac{1-q^{-p+1}}{1-q^{-p}}\\
    &\quad- q^{\degOne+m-2}(1-q^{-1})(\frac{1}{2}-q^{-1}- \frac{1}{2z+2}(1- \frac{z}{q})  )\\
    &= q^{\degOne+m-2}(1-q^{-1})(\frac 1 2 + \frac{1+q^{-1}}{2z+2}
    +\frac{q^{-1}} 2 \\
    &\quad -q^{-\degOne} \frac{1-q^{-p+1}}{1-q^{-p}} -
    q^{-p+1}\frac{1-q^{-1}} {1-q^{-p}}).  \qed
  \end{align*}
\end{proof}
\begin{corollary}\label{cor:decom}
  With the assumptions and notation of \ref{th:decom}, the set
  $D^{+}_{\dg,k}$ of non-Frobenius compositions has at least the
  following size.
  \begin{enumerate}
  \item\label{cor:decom-1} If $ r \neq m$ and $\mu=1$:
   $$q^{\degOne+m}(1-q^{-1})
   (1-q^{-\degOne})(1-q^{-1}(1+q^{-p+2}\frac{(1-q^{-1})^{2}}{1-q^{-p}})).$$
  \item\label{cor:decom-2} If $r \neq m$:
    \begin{align*}
      & q^{\degOne+m}(1-q^{-1})\bigl((1-q^{-1}(1+q^{-p+2}
      \frac{(1-q^{-1})^{2}}{1-q^{-p}}))(1-q^{-\degOne})\\
      &\quad
      -q^{-\degOne-r^{*}-c/e+2}\frac{(1-q^{-1})^{2}(1-q^{-r^{*}(\mu-1)})}{(1-q^{-c/e})
        (1-q^{-r^{*}})}(1+q^{-r^{*}(p-2)})\bigr)\\
      &\geq  q^{\degOne+m}(1-q^{-1})\bigl((1-q^{-1}(1+q^{-p+2}
      \frac{(1-q^{-1})^{2}}{1-q^{-p}}))(1-q^{-\degOne})\\
      &\quad -q^{-\degOne-r^{*}+2}
      \frac{(1-q^{-1})^{2}(1-q^{-r^{*}(\mu-1)})}{1-q^{-r^{*}}}
      (1+q^{-r^{*}(p-2)})\bigr).
    \end{align*}
    If furthermore $r^{*}\geq 2$ and $p > \mu$, then the latter
    quantity is at least
    $$
    q^{\degOne+m}(1-q^{-1})\bigl((1-q^{-1}(1+q^{-p+2}
    \frac{(1-q^{-1})^{2}}{1-q^{-p}}))
    (1-q^{-\degOne})-\frac{4}{3}q^{-\degOne}(1-q^{-1})^{2}\bigr).
    $$
  \item\label{cor:decom-3} If $r=m$:
    $$
    q^{\degOne+m} (1-q^{-1})^2 (\frac 1 2 + \frac{1+q^{-1}}{2z+2}
    +\frac{q^{-1}} 2 -q^{-\degOne} \frac{1-q^{-p+1}}{1-q^{-p}} -
    q^{-p+1}\frac{1-q^{-1}} {1-q^{-p}}).
    $$
  \end{enumerate}
\end{corollary}
\begin{proof}
  All $g$ and $h$ considered in \ref{th:decom} are monic and original,
  and so are their compositions $f$.  We may replace the left hand
  component $g$ of any $(g,h) \in G$ by $(ax+b) \circ g$, where $a,b
  \in \mathbb{F}_{q}$ are arbitrary with $a \neq 0$. Hence
  \begin{align*}
    \#D^{\#}_{\dg,\degOne} \geq q^{2}(1-q^{-1}) \cdot \#
    \gamma_{\dg,\degOne}(G),
  \end{align*}
  and the claims follow from \ref{th:decom}. For the first inequality
  in \short\ref{cor:decom-2}, we observe that $c \geq 1$ and

  \begin{align}\label{frac:leqp}
    \frac{q^{-c/e}}{1-q^{-c/e}}= \frac{p^{-c}}{1-p^{-c}} \leq 1.
  \end{align}

  For the last estimate, we have
  $$
  q^{-r^{*}}\leq 1/4,
  $$
  $$
  q^{-r^{*}(p-2)} \leq q^{-r^{*}(\mu-1)},
  $$
  $$
  (1-q^{-r^{*}(\mu-1)})(1+q^{-r^{*}(p-2)})\leq
  \frac{4}{3}(1-q^{-r^{*}}).\qed
  $$
\end{proof}

The algorithm works over any field of characteristic $p$ where each
element has a $p$th root; in $\mathbb{F}_{q}$, this is just the $(q /
p)$th power. It even works over an arbitrary extension of
$\mathbb{F}_{p}$, rather than just the separable ones, provided we
have a subroutine that tests whether a field element is a $p$th power,
and if so, returns a $p$th root. Then where a $p$th root is requested
in the algorithm (steps
\bare\ref{algoWd:step3}\bare\ref{algoWd:step3-a},
\bare\ref{algoWd:step6}\bare\ref{algoWd:step6-a}, and
\bare\ref{algoWd:step6}\bare\ref{algoWd:step6-c}), we either return
``no decomposition'' or the root, depending on the outcome of the
test.
\begin{example}\label{ex:com}
  When $\dg = p^{2}$, then we have $\degOne=r=m=p$ in
  \ref{cor:decom-3}, and including the Frobenius compositions
  (\ref{lem:propdivi-2}), we obtain
  \begin{align*}
    \#D_{\dg}  &\geq
    \frac{1}{2}q^{2p}(1-q^{-1})^{2}(1+\frac{1+q^{-1}}{p+1}+q^{-1}-2q^{-p+1})
    + q^{p+1}(1-q^{-1})\\
    &= \alpha_{\dg} \cdot \bigl(\frac 1 2
    (1+\frac{1}{p+1})(1-q^{-2})+{q^{-p}} \bigr).
  \end{align*}

  In characteristic 2, the estimate is exact, since we have accounted
  for all compositions and a monic original polynomial of degree $2$
  is determined by its linear coefficient. Thus
  \begin{align*}
    {\#D_{4}} &= {\alpha_{4}} \cdot (\frac{2}{3} \cdot (1-q^{-2}) +
    q^{-2})
    = \alpha_4 \cdot \frac{2+q^{-2}}{3}, \\
    {\#D_{4}} &= \frac{3}{4} \alpha_4
    \text{ over } \mathbb{F}_{2} , \\
    {\#D_{4}} &= \frac{11}{16} \alpha_4 \text{ over } \mathbb{F}_{4}.
  \end{align*}

  Over an algebraically closed field, a quartic polynomial is
  decomposable if and only if its cubic coefficient vanishes; compare
  to \ref{ex:sl}.  For $p=3$, we find
  \begin{align*}
    \#D_{9} &\geq \alpha_{9} \cdot (\frac{5}{8}(1-q^{-2})+{q^{-3}})
    = \alpha_{9} \cdot (\frac{5}{8}-q^{-2}(\frac 5 8 -{q^{-1}})), \\
    \#D_{9} &\geq \frac{16}{27} \cdot \alpha_{9} > 0.59259 \,
    \alpha_9
    \text{ over } \mathbb{F}_{3}, \\
    \#D_{9} &\geq \frac{451}{3^6} \cdot \alpha_{9} > 0.61065 \,
    \alpha_{9} \text{ over } \mathbb{F}_{9}.
  \end{align*}

 \ref{tab:numforsm} shows that these are serious underestimates of the
 actual ratios $\approx 0.8518$ and $0.9542$.
  In the same vein we find, when $p=\degOne $ and $\dg = ap^{2} >p^2$
  with $p \nmid a$, that
  \begin{align*}
    \#D_{n,n/p} \geq \frac{ \alpha_{\dg}} 2 \cdot (\frac 1 2
    (1+\frac{1}{p+1})(1-q^{-2})+{q^{-p}}).\qed
  \end{align*}
\end{example}
\begin{example}\label{ex:coll}
  In $\mathbb{F}_{3}[x]$, we have, besides the eight Frobenius
  collisions according to \ref{rem:coll}, four two-way collisions of degree $9$:
  \begin{gather*}
    (x^{3}+x) \circ (x^{3}-x) = (x^{3}-x) \circ (x^{3}+x) = x^{9}-x,\\
    (x^{3}+x^{2})\circ (x^{3}-x^{2}-x) = (x^{3}-x^{2}+x) \circ
    (x^{3}+x^{2}) = x^{9}+x^{5}-x^{4}+x^{3}+x^{2},\\
    (x^{3}+x^{2}+x)\circ (x^{3}-x^{2}) = (x^{3}-x^{2})\circ
    (x^{3}+x^{2}-x) = x^{9}+x^{5}+x^{4}+x^{3}-x^{2},\\
    (x^{3}+x^{2}+x) \circ (x^{3}-x^{2}+x) =
    (x^{3}-x^{2}+x)\circ(x^{3}+x^{2}+x) = x^{9}+x^{5}+x.
  \end{gather*}

  Our general bounds of \ref{thm:Estimate-1}, \ref{cor:decom}, and
  \ref{ex:com} say that
  $$
  18 \cdot 16 = 288 < 18\cdot 17 = 306 < \#D_{9}= 414 = 18 \cdot 23 <
  486 = 18 \cdot 27 = \alpha_{9}.\qed
  $$
\end{example}

\section{Distinct-degree collisions of
  compositions}\label{sec:collcomp}

In this section, we turn to the last preparatory task. Namely, for a
lower bound on $D_{\dg}$ we have to understand $D_{\dg,l} \cap
D_{\dg,\dg/l}$, that is, the distinct-degree collisions \ref{eq:circ}
when $\deg g^{*} = \deg h = l$. In our application, $l$ is the
smallest prime divisor of $n$.

The following is an example of a collision:
$$
x^{k}w^{l} \circ x^{l}= x^{kl}w^{l}(x^{l})= x^{l} \circ x^{k}w(x^{l}),
$$
for any polynomial $w \in F[x,y]$, where $F$ is a field (or even a
ring).  We define the (bivariate) \emph{Dickson polynomials of the
first kind} $T_{m} \in F[x,y]$ by $T_{0}=2$, $T_{1}=x$,
and
\begin{equation}\label{Trecursion}
  T_{m}=xT_{m-1}- yT_{m-2}
  \text{ for }m \geq 2.
\end{equation}
The monograph of \cite{lidmul93a} provides extensive information about
these polynomials. We have $T_{m}(x,0)= x^{m}$, and $T_{m}(x,1)$ is
closely related to the \emph{Chebyshev polynomial} $C_{\dg}=\cos(\dg
\arccos x)$, as $T_{\dg}(2x,1)=2C_{\dg}(x)$. $T_{m}$ is monic (for $m
\geq 1$) of degree $m$, and
$$
T_{m}= \sum_{0 \leq i \leq
  m/2}\frac{m}{m-i}\binom{m-i}{i}(-y)^{i} x^{m-2i} \in F[x,y].
$$
Furthermore,
\begin{equation}\label{eq:Trecursion}
T_{m}(x,y^{l})\circ T_{l}(x,y)=
T_{lm}(x,y)=T_{l}(x,y^{m})\circ T_{m}(x,y),
\end{equation}
and if $l \neq m$, then substituting any $z \in F$ for $y$ yields a
collision.

Ritt's Second Theorem is the central tool for understanding
distinct-degree collisions, and the following notions enter the
scene. The functional inverse $v^{-1}$ of a linear polynomial
$v=ax+b$ with $a$, $b \in F$ and $a \neq 0$ is defined as
$v^{-1}=(x-b)/a$.  Then $v^{-1} \circ v = v \circ v^{-1} =
x$. Two pairs $(g,h)$ and $(g^{*},h^{*})$ of polynomials are called
\emph{equivalent} if there exists a linear polynomial $v$ such that
$$
g^{*}=g \circ v, \ h^{*}=v^{-1} \circ h.
$$
Then $g \circ h = g^{*}\circ h^{*}$, and we write $(g,h) \sim
(g^{*},h^{*})$.  The following result says that, under certain
conditions, the examples above are essentially the only
distinct-degree collisions. It was first proved by \cite{rit22} for
$F=\C$. We use the strong version of \cite{zan93}, adapted to finite
fileds. The adaption uses \cite{sch00c}, Section 1.4, Lemma 2, and
leads to his Theorem 8. Further references can be found in this
monograph as well. 
\begin{fact}\label{RST}(Ritt's Second Theorem)
  Let $l$ and $m$ be integers, $F$ a field, and $g$, $h$, $g^{*}$,
  $h^{*}\in F[x]$ with
  \begin{equation}\label{eq:RST-1}
    m > l \geq 2, \gcd(l,m)=1, \deg g= \deg h^{*} = m,\\
    \deg h = \deg g^{*} = l,
  \end{equation}
  \begin{equation}\label{eq:RST-2}
    g'(g^{*})' \neq 0,
  \end{equation}
  where $g'= \partial g/ \partial x$ is the derivative of $g$. Then
  \begin{equation}\label{eq:RST-3}
    g \circ h = g^{*} \circ h^{*}
  \end{equation}
  if and only if
  $$
  \exists k \in \mathbb{N}, v_{1}, v_{2} \in F[x] ~\text{linear}, w
  \in F[x] ~\text{with}~ k+l \deg w = m, z\in F^{\times},
  $$
  so that either
  \begin{equation}
    \tag*{First Case}
    \left\{
      \begin{aligned}
        (v_{1} \circ g, h \circ v_{2})  &\sim (x^{k} w^{l}, x^{l}),\\
        (v_{1} \circ g^{*}, h^{*} \circ v_{2}) &\sim (x^{l},
        x^{k}w(x^{l})),
      \end{aligned}
    \right.
  \end{equation}
  or
  \begin{equation}
    \tag*{Second Case}
    \left\{
      \begin{aligned}
        (v_{1} \circ g, h \circ v_{2})&\sim ( T_{m}(x,z^{l}), T_{l}(x,z)),\\
        (v_{1} \circ g^{*}, h^{*} \circ v_{2}) &\sim (T_{l}(x,z^{m}), T_{m}(x,z)).
      \end{aligned}
    \right.
  \end{equation}
\end{fact}

In principle, one also has to consider the First Case with $(g,h,m)$
and $(g^{*}, h^{*}, l)$ interchanged; see \cite{zan93}, Main Theorem
(ii).  Then $k+m \deg w = l$ and hence deg $w= 0$. But this situation
is covered by the First Case in \ref{RST}, with $k=m$. We note that
the conclusion of the First Case is asymmetric in $l$ and $m$, but in
the Second Case it is symmetric, so that there the assumption $m > l$
does not intervene.

According to \ref{invariance1}, we may assume $h$ and $h^{*}$ to be
monic and original. If one of $g$ or $g^{*}$ is also monic and
original, then so is the other one, and also the composition
\ref{eq:RST-3}. It is convenient to add this condition:
\begin{equation}\label{eq:intersec}
  f = g \circ h,\text{ and } g, h, g^{*}, h^{*} \text{ are monic and original}.
\end{equation}
The transition between the general and this special case is by left
composition with a linear polynomial.

The following lemma about Dickson polynomials will be useful for
determining the number of collisions exactly. We write
$T_{\dg}'(x,y)=\partial T_{\dg}(x,y)/\partial x$ for the derivative
with respect to $x$.
\begin{lemma}\label{Tsqfree}
  Let $F$ be a field of characteristic $p \geq 0$, $n \geq 1$, and $z
  \in F^{\times}$.
  \begin{enumerate}
  \item
    \label{Tsqfree-1}
    If $p=0$, or $p \geq 3$ and $\gcd(n,p)=1$, then the derivative
    $T'_n(x,z)$ is squarefree in $F[x]$.
  \item
    \label{Tsqfree-1b}
    If $p=0$ or $\gcd(n,p)=1$, and $n$ is odd, then there exists some
    monic squarefree $u\in F[x]$ of degree $(n-1)/2$ so that
    $T_n(x,z^{2})=(x-2z) \cdot u^{2}+2z^{\dg}$.
  \item
    \label{Tsqfree-2}
    Let $\gamma = (-y)^{\lfloor n/2 \rfloor}$.  $T_n$ is an odd
    or even polynomial in $x$ if $n$ is odd or even, respectively.  It has
    the form
    \begin{align*}
    T_n =
    \begin{cases}
      x^n - nyx^{n-2} +- \cdots + \gamma x & \text{if }n\text{ is odd},\\
      x^n - nyx^{n-2} +- \cdots + 2\gamma & \text{if }n\text{ is even}.
    \end{cases}
    \end{align*}
  \item
    \label{Tsqfree-3}
    If $p \geq 2$, then $T_{p^j} = x^{p^j}$ for $j\geq 0$.
  \item \label{Tsqfree-3b} If $p \geq 2$ and $p \mid \dg$, then
    $T_{\dg}'=0$.
  \item
    \label{Tsqfree-4}
    For a new indeterminate $t$, we have
    $t^{\dg}T_{\dg}(x,y)=T_{\dg}(tx,t^{2}y)$.
  \item
    \label{Tsqfree-5}
    $T_{\dg}(2z,z^{2})=2z^{\dg}$.
  \end{enumerate}
\end{lemma}
\begin{proof}
  \ref{Tsqfree-1} \cite{wil71} and Corollary 3.14 of \cite{lidmul93a}
  show that if $F$ contains a primitive $n$th root of unity $\rho$,
  then $T'_n(x,z)/nc$ factors over $F$ completely into a product of
  quadratic polynomials $(x^{2}-\alpha_{k}^{2}z)$, where $1\leq k
  <\dg/2$, the $\alpha_k = \rho^k+ \rho^{-k}$ are Gau\ss\ periods
  derived from $\rho$, and the $\alpha_k^2$ are pairwise distinct,
  with $c=1$ if $n$ is odd and $c=x$ otherwise. We note that
  $\alpha_{k}=\alpha_{\dg-k}$.  We take an extension $E$ of $F$ that
  contains a primitive $n$th root of unity and a square root $z_{0}$
  of $z$. This is possible since $p=0$ or $\gcd(n,p)=1$.
  ~Thus $x^{2}- \alpha_{k}^{2}z = (x- \alpha_{k}z_{0})(x+
  \alpha_{k}z_{0})$, and the $\pm \alpha_{k}z_{0}$ for $1 \leq k < n/2$ are
  pairwise distinct, using that $p \neq 2$. It follows that
  $T'_n(x,z)$ is squarefree over $E$. Since squarefreeness is a
  rational condition, equivalent to the nonvanishing of the
  discriminant, $T'_n(x,z)$ is also squarefree over $F$.

  For \ref{Tsqfree-1b}, we take a Galois extension field~$E$ of~$F$
  that contains a primitive $n$th root of unity $\rho$, and set
  $\alpha_k = \rho^k + \rho^{-k}$ and $\beta_k = \rho^{k} - \rho^{-k}$
  for all $k\in\mathbb{Z}$.  We have $T_{n}(2z,z^{2}) = 2z^{\dg}$ by
  \ref{Tsqfree-5}, proven below, and Theorem~3.12(i) of
  \cite{lidmul93} states that
  $$
   T_{n}(x,z^{2})-2z^{\dg} = (x-2z) \prod_{1 \leq k < n/2}
   (x^{2}-2\alpha_{k}zx+4z^{2}+\beta_{k}^{2}z^{2});
  $$
  see also \cite{tur95}, Proposition 1.7. Now $ -\alpha_{k}^{2} + 4 +
  \beta_{k}^{2} = -(\rho^{k}+\rho^{-k})^{2} + (\rho^{k}-\rho^{-k})^{2}
  + 4 = 0, $ so that $ x^{2}-2 \alpha_{k}z x+4z^{2}+\beta_{k}^{2}z^{2}
  = (x-\alpha_{k}z)^{2}$. We set $u= \prod_{1 \leq k < n/2} (x-
  \alpha_{k}z)\in E[x]$. Then $T_{n}(x,z^{2})-2z^{\dg} = (x-2z)u^{2}$,
  and $u$ is squarefree. It remains to show that $u \in F[x]$. We take
  some $\sigma \in~Gal(E:F)$.  Then $\sigma(\rho)$ is also a primitive
  $\dg$th root of unity, say $\sigma(\rho)=\rho^{i}$ with $1\leq
  i<\dg$ and $~gcd(i,\dg)=1$. We take some $k$ with $1\leq k<\dg/2$,
  and $j$ with $ik\equiv j \bmod \dg$ and $0<|j| <\dg/2$. Then
  $\sigma(\alpha_{k})=\alpha_{|j|}$. Hence, $\sigma$ induces a
  permutation on $\Set{\alpha_{1}, \dots, \alpha_{(n-1)/2}}$. It
  follows that
  
  $$u = \prod_{1 \leq k < n/2} (x-\alpha_{k}z) = \prod_{1 \leq k < n/2}
  (x-\sigma(\alpha_{k}z)) = \sigma u.
  $$
  Since this holds for all $\sigma$, we have $u\in F[x]$.

  \ref{Tsqfree-2} follows from the recursion \ref{Trecursion}, and
  \ref{Tsqfree-3} from \cite{lidmul93a}, Lemma
  2.6(iii). \ref{Tsqfree-3b} follows from \ref{eq:Trecursion} and
  \ref{Tsqfree-3}. The claim in \ref{Tsqfree-4} is Lemma 2.6(ii) of
  \cite{lidmul93a}. It also follows inductively from \ref{Trecursion},
  as does \ref{Tsqfree-5}.
\end{proof}
In the following, we present several pairs of results. In each pair,
the first item is a theorem, valid over fairly general fields, that
describes the structure of distinct-degree collisions. The second one
is a corollary, valid over finite fields, giving bounds on the number
of such collisions. We start with the following normal form for the
decompositions in Ritt's Second Theorem. The uniqueness result is not
obvious, as witnessed by the quotes in the Introduction.
\begin{theorem}\label{th:fifi}
  Let $F$ be a field of characteristic $p$, let $m>l\geq 2$ be
  integers, and $n=lm$.  Furthermore, we have monic original $f, g, h,
  g^{*}, h^{*} \in F[x]$ satisfying \ref{eq:RST-1} through
  \ref{eq:intersec}. Then either \short\ref{th:fifi-1} or
  \short\ref{th:fifi-2} hold, and \short\ref{Tsqfree-2} is also
  valid.
  \begin{enumerate}
  \item\label{th:fifi-1} (First Case) There exists a monic
    polynomial $w \in F[x]$ of degree $s$ and $c \in F$ so that
    \begin{equation}\label{eq:mopo}
      f=(x-\parTwo^{kl}w^{l}(\parTwo^{l})) \circ x^{kl}w^{l}(x^{l}) \circ (x+\parTwo),
    \end{equation}
   where $m=sl+k$ is the division with remainder of $m$ by $l$, with
   $1 \leq k < l$. Furthermore
    \begin{align}\label{eq:unidet}
      &  kw+lxw' \neq 0 \text{ and } p \nmid l,\\
      g &= (x-\parTwo^{kl}w^{l}(\parTwo^{l})) \circ x^{k}w^{l} \circ (x+\parTwo^{l}), \nonumber\\
      h& = (x-\parTwo^{l}) \circ x^{l} \circ (x+\parTwo), \nonumber\\
      g^{*}& = (x-\parTwo^{kl}w^{l}(\parTwo^{l})) \circ x^{l} \circ (x+\parTwo^{k}w(\parTwo^{l})),\nonumber \\
      h^{*} &= (x-\parTwo^{k}w(\parTwo^{l})) \circ x^{k}w(x^{l}) \circ
      (x+\parTwo).\nonumber
    \end{align}

    Conversely, any $(w,\parTwo)$ as above for which \ref{eq:unidet}
    holds yields a collision satisfying \ref{eq:RST-1} through
    \ref{eq:intersec}, via the above formulas. If $p\nmid m$, then
    $(w,\parTwo)$ is uniquely determined by $f$ and $l$.
  \item\label{th:fifi-2} (Second Case) There exist $z,\parTwo \in F$
    with $z \neq 0$ so that
    \begin{align}\label{eq:TN}
      f=(x-T_{n}(\parTwo,z)) \circ T_{n}(x,z) \circ (x+\parTwo).
    \end{align}
    Now $(z,\parTwo)$ is uniquely determined by $f$. Furthermore we have
    \begin{align}\label{eq:ab}
      & p \nmid \dg,\\
      g &= (x-T_{n}(\parTwo,z)) \circ T_{m}(x,z^{l}) \circ (x + T_{l}(\parTwo,z)),\nonumber \\
      h &= (x-T_{l}(\parTwo,z)) \circ T_{l}(x,z) \circ (x+\parTwo), \nonumber\\
      g^{*} &=  (x-T_{n}(\parTwo,z)) \circ T_{l}(x,z^{m}) \circ (x+T_{m}(\parTwo,z)),\nonumber \\
      h^{*} &=(x-T_{m}(\parTwo,z)) \circ T_{m}(x,z) \circ (x+\parTwo).\nonumber
    \end{align}
    Conversely, if \ref{eq:ab} holds, then any $(z,\parTwo)$ as above yields a
    collision satisfying \ref{eq:RST-1} through \ref{eq:intersec}, via
    the above formulas.

  \item\label{th:fifi-3} When $l \geq 3$, the First and Second Cases
    are mutually exclusive. For $l=2$, the Second Case is included in
    the First Case.
  \end{enumerate}
\end{theorem}

\begin{proof}

   By assumption, either the First or the Second
  Case of Ritt's Second Theorem (\ref{RST}) applies. 

  \short\ref{th:fifi-1} From the First Case in \ref{RST}, we have a
  positive integer $K$, linear polynomials $v_{1}$, $v_{2}$, $v_{3}$,
  $v_{4}$ and a nonzero polynomial $W$ with $d = \deg W =(m-K)/l $ and
  (renaming $v_{2}$ as $v_{2}^{-1}$)
  \begin{align*}
    x^{K}W^{l} &=  v_{1} \circ g \circ v_{3},\\
    x^{l} &= v_{3}^{-1} \circ h \circ v_{2}^{-1},\\
    x^{l} &= v_{1} \circ g^{*} \circ v_{4},\\
    x^{K}W(x^{l}) &= v_{4}^{-1} \circ h^{*} \circ v_{2}^{-1}.
  \end{align*}
  We abbreviate $r=~lc(W)$, so that $r\neq 0$, and write $v_{i} =
  a_{i} x+b_{i}$ for $1 \leq i \leq 4$ with all $a_{i}$, $b_{i}\in F$
  and $a_{i} \neq 0$, and first express $v_{3} $, $v_{4}$, and $v_{1}$
  in terms of $v_{2}$. We have
  \begin{align*}
    h &= v_{3} \circ x^{l} \circ v_{2} = a_{3}(a_{2}x+b_{2})^{l}+b_{3},\\
    h^{*} &= v_{4} \circ x^{K}W(x^{l}) \circ v_{2} =
    a_{4}(a_{2}x+b_{2})^{K} \cdot W((a_{2}x+b_{2})^{l})+b_{4}.
  \end{align*}
  Since $h$ and $h^{*}$ are monic and original and $K+ld=m$, it
  follows that
  $$
  a_{3}= a_{2}^{-l}, \ b_{3}=-a_{2}^{-l}b_{2}^{l}, \ a_{4}=
  a_{2}^{-m}r^{-1}, \ b_{4}= -a_{2}^{-m}b_{2}^{K}r^{-1}W(b_{2}^{l}).
  $$ 
  Playing the same game with $g$, we find
  \begin{align*}
    g=v_{1}^{-1} \circ x^{K} W^{l} \circ v_{3}^{-1}&= a_{1}^{-1}
    \bigl( (\frac{x-b_{3}}{a_{3}})^{K} W^{l} (\frac{x-b_{3}}{a_{3}})
    -b_{1}
    \bigr),\\
    a_{1} &= a_{2}^{n} r^{l}, \\
    b_{1} &= b_{2}^{Kl}W^{l}(b_{2}^{l}).
  \end{align*}
  We note that then
  $$
  g^{*} = v_{1}^{-1} \circ x^{l} \circ v_{4}^{-1} = a_{1}^{-1} \bigl(
  (\frac{x-b_{4}}{a_{4}})^{l}-b_{1} \bigr)
  $$ 
  is automatically monic and original. Furthermore, we have $d =
  (m-K)/l \leq \lfloor m/l \rfloor = s$ and
  \begin{equation}
    \label{adm}
    f = v_{1}^{-1} \circ (v_{1} \circ g \circ
    v_{3}) \circ (v_{3}^{-1} \circ h \circ v_{2}^{-1}) \circ v_{2}= v_{1}^{-1}
    \circ x^{Kl}\cdot W^{l}(x^{l}) \circ v_{2}.
  \end{equation}
  We set
  \begin{align*}\parTwo&=\frac{b_{2}}{a_{2}}\in F, \quad
    u_{1}=x+\frac{b_{1}}{a_{1}}=\frac{v_{1}}{a_{1}}, \quad
    u_{2}=x+\parTwo=\frac{v_{2}}{a_{2}}, \quad\\
    w&=r^{-1}a^{-ld}_{2}x^{s-d}\cdot W(a^{l}_{2}x) \in
    F[x].
  \end{align*} 
  Then $b_{1}/a_{1}= a^{kl}w^{l}(a^{l})$, $w$ is monic of degree $s$,
  $ u_{1}^{-1}= x-b_{1}/a_{1}= x - \parTwo^{kl}
  w^{l}(\parTwo^{l})$, and
\begin{align}\label{eq:mondegs}
W(x)=~lc (W)a^{ls}_{2}x^{-(s-d)}w(a_{2}^{-l}x).
\end{align}

   Noting that $m = ld + K =ls + k$, the equation
  analogous to \ref{adm} reads
  \begin{align}
    u_{1}^{-1} \circ x^{kl} w^{l}(x^{l})\circ u_{2} &= a_{1}\cdot
    v^{-1}_{1} \circ x^{kl} \cdot
    \frac{x^{l^{2}(s-d)}W^{l}(a_{2}^{l}x^{l})}{a_{2}^{dl^{2}}r^{l}}
    \circ
    \frac{v_{2}}{a_{2}} \nonumber \\
    &= v_{1}^{-1}\circ a_{2}^{\dg}r^{l}\cdot
    \bigl(\frac{v_{2}}{a_{2}}\bigr)^{kl} \cdot
    \bigl(\frac{v_{2}}{a_{2}}\bigr)^{l^{2}(s-d)} \cdot
    \frac{W^{l}(v_{2}^{l})}{a_{2}^{dl^{2}}r^{l}} \nonumber \\
    &= v_{1}^{-1}\circ x^{Kl} \cdot W^{l}(x^{l})\circ v_{2} = f.
    \label{compFirst}
  \end{align}
  This proves the existence of $w$ and $\parTwo$, as claimed in
  \ref{eq:mopo}.

  In order to express the four components in the new parameters, we
  note that $K=k+l(s-d)$. Thus
  \begin{align*}
    g &= v_{1}^{-1} \circ x^{K}W^{l} \circ v_{3}^{-1} \\
    &= (r^{-l}a_{2}^{-n} x-\parTwo^{kl}w^{l}(\parTwo^{l})) \circ
    (a_{2}^{l}(x+\parTwo^{l}))^{K} \cdot
    W^{l}(a_{2}^{l}(x+\parTwo^{l})) \\
    &= r^{-l}a_{2}^{-n} \bigl(a_{2}^{Kl}(x+\parTwo^{l})^{K} \cdot
    r^{l}a_{2}^{l^{2}s} a_{2}^{-l^{2}(s-d)}
    (x+\parTwo^{l})^{-l(s-d)} w^{l}(x+\parTwo^{l})\bigr)\\
    & \quad-\parTwo^{kl}w^{l}(\parTwo^{l}) \\
    &= a_{2}^{-n+Kl+l^{2}s-l^{2}s+l^{2}d} (x+\parTwo^{l})^{K-ls+ld}
    w^{l}(x+\parTwo^{l})
    -\parTwo^{kl}w^{l}(c\parTwo^{l}) \\
    &= (x+\parTwo^{l})^{k} w^{l}(x+\parTwo^{l}) -\parTwo^{kl}
    w^{l}(\parTwo^{l})\\
    &= \bigl(x-\parTwo^{kl} w^{l}(\parTwo^{l})\bigr) \circ x^{k}w^{l} \circ
    (x+\parTwo^{l}), \\ 
    h &= v_{3} \circ x^{l} \circ v_{2}
    = a_{2}^{-l}(a_{2}x+b_{2})^{l} - a_{2}^{-l}b_{2}^{l}\\
    &= (x-\parTwo^{l}) \circ x^{l} \circ (x+\parTwo), \\
    g^{*} &= v_{1}^{-1} \circ x^{l} \circ v_{4}^{-1} \\
    &= (r^{-l}a_{2}^{-n} x-\parTwo^{kl} w^{l}(\parTwo^{l})) \circ
    \bigl(ra_{2}^{m}(x+r^{-1}a_{2}^{-m} b_{2}^{K} \cdot
    W(b_{2}^{l}))\bigr)^{l} \\
    &= (x+r^{-1} a_{2}^{-m} b_{2}^{K} \cdot r a_{2}^{ls}
    b_{2}^{-l(s-d)} w(\parTwo^{l}))^{l} -
    \parTwo^{kl}w^{l}(\parTwo^{l}) \\
    &= \bigl(x+a_{2}^{-k}b_{2}^{k}w(\parTwo^{l})\bigr)^{l}
    -\parTwo^{kl}w^{l}(\parTwo^{l}) \\
    &= (x-\parTwo^{kl}w^{l}(\parTwo^{l})) \circ x^{l} \circ (x+\parTwo^{k}w(\parTwo^{l})),\\
    h^{*} &= v_{4} \circ x^{K} W(x^{l}) \circ v_{2} \\
    &=\bigl(r^{-1}a_{2}^{-m} (x-b_{2}^{K} W(b_{2}^{l}))\bigr) \circ
    (a_{2}(x+\parTwo))^{K} W(a_{2}^{l}(x+\parTwo)^{l}) \\
    &= r^{-1}a_{2}^{-m} \cdot ra_{2}^{ls} \cdot \bigl((a_{2}^{K}
    (x+\parTwo)^{K} (a_{2}^{l}(x+\parTwo)^{l}))^{-(s-d)}
    w((x+\parTwo)^{l})\\
    &\quad  - b_{2}^{K} b_{2}^{-l(s-d)} w(\parTwo^{l})\bigr) \\
    &= a_{2}^{-k} \bigl(a_{2}^{K-l(s-d)} (x+\parTwo)^{K-l(s-d)} w((x+\parTwo)^{l})
    -b_{2}^{K-l(s-d)} w(\parTwo^{l}) \bigr)\\
    &= (x+\parTwo)^{k} w((x+\parTwo)^{l})- \parTwo^{k}w(\parTwo^{l})\\
    &= (x-\parTwo^{k}w(\parTwo^{l})) \circ x^{k}w(x^{l}) \circ (x+\parTwo).
  \end{align*}
  \ref{eq:mopo} has been shown above. We note that in the right hand
  component $x+\parTwo$, the constant $\parTwo$ is arbitrary. All other linear
  components follow automatically from the required form of $g$, $h$,
  $g^{*}$, $h^{*}$, namely, being monic and original, and from the
  condition that $g$ and $h$ (and $g^{*}$ and $h^{*}$) have to match
  up with their ``middle'' components.
  Furthermore, we have
  \begin{align}\label{eq:kw}
    \begin{aligned}
      0  = g'=(x^{k-1}w^{l-1}(kw+lxw')) \circ (x+\parTwo^{l})
      &\Longleftrightarrow
      kw+ lxw'=0,\\
      0  = (g^{*})' = lx^{l-1} \circ (x+\parTwo^{k}w(\parTwo^{l}))
     &\Longleftrightarrow p \mid l.
    \end{aligned}
  \end{align}

  Thus \ref{eq:unidet} follows from \ref{eq:RST-2}.

  In order to prove the uniqueness if $p\nmid\dg$, we take monic $w$,
  $ \tilde w \in F[x]$ of degree $s$, and $\parTwo$, $\tilde \parTwo \in F$ and
  the unique monic linear polynomials $v$ and $\tilde v$ for which
  \begin{align}\label{eq:unique}
    f = v \circ x^{kl}w^{l}(x^{l}) \circ (x+\parTwo)= \tilde v \circ x^{kl}
    \tilde{w}^{l}(x^{l}) \circ (x+\tilde \parTwo).
  \end{align}
  
  By composing on the left and right with $\tilde{v}^{-1}$ and
  $(x+\tilde \parTwo)^{-1}$, respectively, and abbreviating
  $u=\tilde{v}^{-1}\circ v$, we find
  \begin{align*}
    x^{kl}\tilde{w}^{l}(x^{l})&=\tilde{v}^{-1}\circ v\circ
    x^{kl}w^{l}(x^{l})\circ(x+\parTwo)\circ(x-\tilde \parTwo)\\
    &=u\circ x^{kl}w^{l}(x^{l})\circ(x+\parTwo-\tilde \parTwo).
  \end{align*}
  Since $l\geq 2$ and the left hand side is a polynomial in $x^{l}$,
  its second highest coefficient (of $x^{n-1}$) vanishes. Equating
  this with the same coefficient on the right, and abbreviating
  $a^{*}=\parTwo-\tilde \parTwo$, we find
  $$
  0=kla^{*}+sl^{2}a^{*}=na^{*},
  $$
  so that $a^{*}=0$, since $p\nmid\dg$. Thus $\parTwo=\tilde \parTwo$ and
  $$
  x^{k}\tilde{w}^{l}\circ x^{l}=x^{kl}\tilde{w}^{l}(x^{l})=u\circ
  x^{kl}w^{l}(x^{l})=u\circ x^{k}w^{l}\circ x^{l},
  $$
 $$x^k \tilde w^l =u \circ x^k w^l.$$
  Now $x^{k}\tilde w^{l}$ and $x^{k}w^{l}$ are monic and original, since $k\geq 1$.
  It follows that $u=x$ and $w^l= \tilde w^l$. Both polynomials are
  monic, so that $w=\tilde w$, as claimed. (The equation for $h$ in
  \ref{th:fifi-1} determines $a$ uniquely provided that $p \nmid l$,
  even if $p \mid m$. However, the value of $h$ is not unique in this case.)

  Conversely, we take some $(w,\parTwo)$ satisfying \ref{eq:unidet} and
  define $f$, $g$, $h$, $g^{*}$, $h^{*}$ via the formulas in
  \short\ref{th:fifi-1}. Then \ref{eq:RST-1}, \ref{eq:RST-3}, and
  \ref{eq:intersec} hold. As to \ref{eq:RST-2}, we have $p \nmid l $
  from \ref{eq:unidet}, and hence $(g^{*})' \neq 0$. Furthermore,
  $$
  (x^{k}w^{l})'= x^{k-1}w^{l-1}\cdot (kw+lxw')\neq 0,
  $$
  so that also $g' \neq 0$.

  \short\ref{th:fifi-2} In the Second Case, again renaming $v_{2}$ as
  $v_{2}^{-1}$, and also $z$ as $z_{2}$, we have from \ref{RST}
  \begin{align*}
    T_{m}(x,z_{2}^{l}) &= v_{1} \circ g \circ v_{3},\\
    T_{l}(x,z_{2})&= v_{3}^{-1} \circ h \circ v_{2}^{-1},\\
    T_{l}(x,z_{2}^{m})&= v_{1} \circ g^{*} \circ v_{4},\\
    T_{m}(x,z_{2})&= v_{4}^{-1} \circ h^{*} \circ v_{2}^{-1},\\
    h &= v_{3} \circ T_{l}(x,z_{2}) \circ v_{2} =
    a_{3}T_{l}(a_{2}x+b_{2},z_{2})+b_{3},\\
    h^{*} &= v_{4} \circ T_{m}(x,z_{2})\circ v_{2}=
    a_{4}T_{m}(a_{2}x+b_{2},z_{2})+b_{4}.
  \end{align*}
  As before, it follows that
  $$
  a_{3}=a_{2}^{-l}, \quad b_{3}=-a_{2}^{-l}T_{l}(b_{2},z_{2}), \quad
 a_{4}=a_{2}^{-m}, \quad b_{4}=-a_{2}^{-m}T_{m}(b_{2},z_{2}).
  $$
  Furthermore, we have
  \begin{align*}
    g &= v_{1}^{-1} \circ T_{m}(x,z_{2}^{l}) \circ v_{3}^{-1} =
    a_{1}^{-1}(T_{m}(a_{3}^{-1}(x-b_{3}),z_{2}^{l})-b_{1}),\\
    a_{1}&= a_{2}^{\dg},\\
    b_{1} &= T_{m}(T_{l}(b_{2}, z_{2}),z_{2}^{l})=T_{\dg}(b_{2}, z_{2}),\\
    f &= \bigl(a_{2}^{-\dg}(x-T_{\dg}(b_{2},z_{2}))\bigr) \circ
    T_{\dg}(x,z_{2}) \circ (a_{2}x+b_{2}).
  \end{align*}

  We now set $\parTwo=b_{2}/a_{2}$ and $z=z_{2}/a_{2}^{2}$ and show that the
  preceding equation holds with $(1,\parTwo,z)$ for $(a_{2}, b_{2},
  z_{2})$. \ref{Tsqfree-4} with $t=a_{2}^{-1}$ says that
\begin{align*}
a_{2}^{-\dg}T_{\dg}(a_{2}x+b_{2}, z_{2})&=T_{\dg}(x+\parTwo,z),\\
a_{2}^{-\dg}T_{\dg}(b_{2},z_{2})&= T_{\dg}(\parTwo,z),\\
f&=(x-T_{\dg}(\parTwo,z))\circ T_{\dg}(x,z)\circ(x+\parTwo).
\end{align*}
Thus the first claim in \short\ref{th:fifi-2} holds with these
values.
In the same vein, applying \ref{Tsqfree-4} with $t$ equal to
$a_{2}^{-1}, a_{2}^{-l}, a_{2}^{-m}, a_{2}^{-1}$, respectively, yields
\begin{align*}
a_{2}^{-l}T_{l}(a_{2}x+b_{2},z_{2})&=T_{l}(x+\parTwo,z),\\
a_{2}^{-\dg}T_{m}(a_{2}^{l}x+T_{l}(b_{2},z_{2}),z_{2}^{l})&=T_{m}(x+a_{2}^{-l}T_{l}(b_{2},z_{2}),z^{l})\\
&=T_{m}(x+T_{l}(\parTwo,z),z^{l}),\\
a_{2}^{-\dg}T_{l}(a_{2}^{m}x+T_{m}(b_{2},z_{2}),z_{2}^{m})&=T_{l}(x+a_{2}^{-m}T_{m}(b_{2},z_{2}),z^{m})\\
&=T_{l}(x+T_{m}(\parTwo,z),z^{m}),\\
a_{2}^{-m}T_{m}(a_{2}x+b_{2}, z_{2})&=T_{m}(x+\parTwo,z).
\end{align*}

For the four components, we have
  \begin{align*}
    g  &= v_{1}^{-1} \circ T_{m}(x,z_{2}^{l}) \circ v_{3}^{-1}\\
    &= a_{2}^{-\dg}(x-T_{\dg}(b_{2},z_{2})) \circ T_{m}(x,z_{2}^{l}) \circ (a_{2}^{l}x+T_{l}(b_{2},z_{2}))\\
    &=a_{2}^{-\dg}T_{m}(a_{2}^{l}x+T_{l}(b_{2},z_{2}),z_{2}^{l})-a_{2}^{-\dg}T_{m}(T_{l}(b_{2},z_{2}),z_{2}^{l})\\
    &=T_{m}(x+T_{l}(\parTwo,z),z^{l})-T_{\dg}(\parTwo,z)\\
    &=(x-T_{\dg}(\parTwo,z))\circ T_{m}(x,z^{l})\circ(x+T_{l}(\parTwo,z)),\\
    h &= v_{3} \circ T_{l}(x,z_{2}) \circ v_{2}  = a_{2}^{-l} T_{l}(a_{2}x+b_{2},z_{2})-a_{2}^{-l} T_{l}(b_{2},z_{2})\\
    &= a_{2}^{-l}(x-T_{l}(b_{2},z_{2})) \circ T_{l}(x,z_{2}) \circ (a_{2}x+b_{2})\\
    &=T_{l}(x+\parTwo,z)-T_{l}(\parTwo,z)\\
    &=(x-T_{l}(\parTwo,z))\circ T_{l}(x,z)\circ (x+\parTwo),\\
    g^{*}  &= v_{1}^{-1} \circ T_{l} (x,z_{2}^{m}) \circ v_{4}^{-1}\\
    &= a_{2}^{-\dg}(x-T_{\dg}(b_{2},z_{2})) \circ T_{l}(x,z_{2}^{m})
    \circ (a_{2}^{m}x+T_{m}(b_{2},z_{2}))\\
    &=a_{2}^{-\dg}T_{l}(a_{2}^{m}x+T_{m}(b_{2},z_{2}),z_{2}^{m})-a_{2}^{-\dg}T_{\dg}(b_{2},z_{2})\\
    &= T_{l}(x+T_{m}(\parTwo,z),z^{m})-T_{\dg}(\parTwo,z)\\
    &=(x-T_{\dg}(\parTwo,z))\circ T_{l}(\parTwo,z^{m})\circ (x+T_{m}(\parTwo,z)),\\
    h^{*}  &= v_{4} \circ T_{m}(x,z_{2}) \circ v_{2}\\
    &= a_{2}^{-m}(x-T_{m}(b_{2},z_{2})) \circ T_{m}(x,z_{2})\circ
    (a_{2}x+b_{2})\\
    &=a_{2}^{-m}T_{m}(a_{2}x+b_{2},z_{2})-a_{2}^{-m}T_{m}(b_{2},z_{2})\\
    &=T_{m}(x+\parTwo,z)-T_{m}(\parTwo,z)\\
    &=(x-T_{m}(\parTwo,z))\circ T_{m}(x,z)\circ (x+\parTwo).
  \end{align*}

  Since
  $$
  0 \neq g' = T_{m}'(x,z^{l})\circ (x+T_{l}(\parTwo,z)),
  $$
  \ref{Tsqfree-3b} implies that $p \nmid m$. Similarly, the
  non-vanishing of $(g^{*})'$ implies that $p \nmid l$, and
  \ref{eq:ab} follows.

  Next we claim that the representation of $f$ is unique.
So we take some $(z,\parTwo),(z^{*},\parTwo^{*})\in F^{2}$ with
$zz^{*}\neq 0$ and
\begin{align}\label{eq:klam}
(x-T_{\dg}(\parTwo,z))\circ T_{\dg}(x,z)\circ
(x+a)=(x-T_{\dg}(\parTwo^{*},z^{*}))\circ T_{\dg}(x,z^{*})\circ(x+\parTwo^{*}).
\end{align}

  Comparing the coefficients of $x^{\dg-1}$ in \ref{eq:klam} and using
  \ref{Tsqfree-2} yields $\dg \parTwo=\dg \parTwo^{*}$, hence $\parTwo=\parTwo^{*}$, since
  $p\nmid \dg$. We now compose \ref{eq:klam} with $x-\parTwo$ on the right
  and find
$$
(x-T_{\dg}(\parTwo,z))\circ T_{\dg}(x,z)=(x-T_{\dg}(\parTwo,z^{*}))\circ T_{\dg}(x,z^{*}).
$$
  Now the coefficients of $x^{\dg-2}$ yield $-\dg z=-\dg z^{*}$, so
  that $z=z^{*}$.

  The converse claim that any $(z,\parTwo)$ with $z\neq 0$ and
  \ref{eq:ab} yields a
  collision as prescribed follows since \ref{eq:ab} and
  \ref{Tsqfree-3b} imply that $T_{m}'(x,z^{l}) T_{l}'(x,z^{m}) \neq 0$.

  \short\ref{th:fifi-3} We first assume $l \geq 3$ and show that the
  First and Second Cases are mutually exclusive. Assume, to the
  contrary, that in our usual notation we have
  \begin{equation}
    \label{exptrig1}
    f = v_1 \circ x^{kl}w^{l}(x^{l}) \circ (x+\parTwo)=
    v_2 \circ T_n(x,z) \circ (x+\parTwo^{*}),
  \end{equation} 
  where $v_{1}$ and $v_{2}$ are the unique linear polynomials that
  make the composition monic and original, as specified in
  \short\ref{th:fifi-1} and \short\ref{th:fifi-2}. Then
  \begin{align*}
    f &= (v_1 \circ x^{k}w^{l}\circ (x+\parTwo^l)) \circ ((x+\parTwo)^l-\parTwo^l) \\
    &= \bigl(v_2 \circ T_{m} (x+T_l(\parTwo^{*},z),z^{l} )\bigr) \circ
    (T_l(x+\parTwo^{*},z)-T_l(\parTwo^{*},z)).
  \end{align*}

  These are two normal decompositions of $f$, and since $p \nmid m$ by
  \ref{eq:ab}, 
  the uniqueness of \ref{cor:inj-1} implies that
  \begin{align}
    \label{exptrig3}
    h &= (x+\parTwo)^l-\parTwo^l =T_l(x+\parTwo^{*},z)-T_l(\parTwo^{*},z),\\
    h' &= l(x+\parTwo)^{l-1} =T'_l(x+\parTwo^{*},z).  \nonumber
  \end{align}

  If $p = 0$ or $p \geq 3$, then according to \ref{Tsqfree-1}, $T'_l(x,z)$
  is squarefree, while $(x+\parTwo)^{l-1}$ is not, since $l\geq 3$. This
  contradiction refutes the assumption \ref{exptrig1}.

  If $p=2$, then $l$ is odd by \ref{eq:ab}. After adjoining a square
  root $z_{0}$ of $z$ to $F$ (if necessary), \ref{Tsqfree-1b} implies that
  $T'_{l}(x,z)=((x-2z_{0})u^{2}+2z_{0}^{\dg})'=u^{2}$ has $(l-1)/2$ distinct roots in an
  algebraic closure of $F$, while $(x+\parTwo)^{l-1}$ has only one. This
  contradiction is sufficient for $l \geq 5$. For $l=3$, we have
  $T_{3}= x^{3}-3 yx$ and there are no $a$, $\parTwo^{*}$, $z\in F$ with $z \neq 0$
  so that
  \begin{align*}
    x^{3} + ax^{2} + a^{2}x  &= (x+a)^{3}-a^{3}= (x+\parTwo^{*})^{3} -3z
    (x+\parTwo^{*})-((\parTwo^{*})^{3} -3z\parTwo^{*})\\
    &= x^{3} + a^{*} x^{2} +((\parTwo^{*})^{2}+z)x.
  \end{align*}
  Again, \ref{exptrig1} is refuted.

  For $l=2$, we claim that any composition
  $$
  f = v_{1} \circ T_{m}(x,z^{2}) \circ T_{2}(x,z) \circ v_{2}
  $$ 
  of the Second Case already occurs in the First Case. We have
  $T_{2}=x^{2}-2y$. Since $m$ is odd by \ref{eq:RST-1} and $p \nmid m$
  by \ref{eq:ab}, \ref{Tsqfree-1b} guarantees a monic $u\in F[x]$ of
  degree $d = (m-1)/2$ with $T_{m}(x,z^{2}) =T_{m}(x,(-z)^{2})=
  (x+2z)u^{2}-2z^{m}$. Then for $\tilde{u} = u \circ (x-2z)$ we have
  $$
  f = v_{1} \circ ((x+2z)u^{2}-2z^{m}) \circ (x^{2}-2z) \circ v_{2} =
  (v_{1}-2z^{m}) \circ x^{2} \tilde{u}^{2} (x^{2}) \circ v_{2},
  $$  
  which is of the form \ref{eq:mopo}, with $k = m-2d = 1$.
\end{proof}

\begin{remark}\label{rem:opara}
  Other parametrizations are possible. As an example, in the Second
  Case, for odd $q=p$, one can choose a nonsquare
  $z_{0}\in F = \mathbb{F}_{q}$ and $B=\{1,\dots,(q-1)/2 \}$. Then all $f$
  in \ref{eq:TN} can also be written as
$$
f=b^{-\dg}(x-T_{\dg}(a,z))\circ T_{\dg}(x,z)\circ(bx+a)
$$
with unique $(z, a, b) \in \{1, z_{0}\} \times F \times B=Z$. To wit,
let $z,a \in F$ with $z\neq 0$. Take the unique $(z^{*}, a^{*}, b) \in
Z$, so that $z^{*}=b^{2}z$ and $a^{*}=ab$. Then $z^{*}$ is determined
by the quadratic character of $z$, and $b$ by the fact that every
square in $F^{\times}$ has a unique square root in $A$; the other one
is $-b\in F^{\times}\setminus A$. \ref{Tsqfree-4} says that
$$
b^{\dg}T_{\dg}(x,z)=T_{\dg}(bx,z^{*}),
$$
\begin{align*}
(x-T_{\dg}(a,z))\circ T_{\dg}(x,z)\circ
(x+a)&=b^{-\dg}(x-T_{\dg}(a^{*},z^{*}))\circ T_{\dg}(bx,z^{*})\circ(x+a)\\
&=b^{-\dg}(x-T_{\dg}(a^{*},z^{*}))\circ T_{\dg}(x,z^{*})\circ (bx+a^{*}),
\end{align*}
as claimed. If $F$ is algebraically closed, as in \cite{zan93}, we can take
$z=1$. The reduction from finite fields to this case is provided by
\cite{sch00c}, Section 1.4, Lemma 2.
\end{remark}

\begin{remark}
Given just $f \in F[x]$, how can we determine whether Ritt's Second
Theorem applies to it, and if so, compute $(w,a)$ or $(z,a)$, as
appropriate? We may assume $f$ to be monic and original of degree
$n$. The divisor $l$ of $n$ might be given as a further input, or we
perform the following for all divisors $l$ of $n$ with $2 \leq l \leq
\sqrt{n}$ and $\gcd(l,n/l)=1$. If $p \nmid n$, the task is easy. We
compute decompositions
$$
f = g \circ h = g^{*} \circ h^{*}
$$
with $~deg h = ~deg g^{*}=l$ and all components monic and original. If
one of these decompositions does not exist, Ritt's Second Theorem does
not apply; otherwise the components are uniquely determined. If
$h_{l-1}$ is the coefficient of $x^{l-1}$ in $h$, then $a =
h_{l-1}/l$ in \ref{eq:mopo}. Furthermore,
\begin{align*}
g(-a^{l}) &= -a^{kl}w^{l}(a^{l}),\\
g \circ (x-a^{l})-g(-a^{l})&= x^{k}w^{l},
\end{align*}
from which $w$ is easily determined via an $x$-adic Newton iteration
for extracting an $l$th root of the reversal of the left hand side,
divided by $x^{k}$.
Actually only a single Newton step is required to compute the root
modulo $x^{2}$.

If the Second Case applies, then by \ref{Tsqfree-2} the three highest
coefficients in $f$ are
\begin{align*}
f  &= x^{n}+f_{n-1}x^{n-1}+f_{n-2}x^{n-2}+O(x^{n-3})\\
 &= (x+a)^{n}-nz(x+a)^{n-2}+O(x^{n-4})\\
&= x^{n}+na x^{n-1}+\bigl(\frac{n(n-1)}{2}a^{2}-nz \bigr) x^{n-2}+O(x^{n-3});
\end{align*}
this determines $a$ and $z$.
\end{remark}

\begin{remark}\label{rem:rid}
  If $p\nmid\dg$, then we can get rid of the right hand component
  $x+\parTwo$ by a further normalization. Namely, when
  $f=x^{\dg}+\sum_{{0\leq i<\dg}}{f_{i}x^{i}},$ then
  $f\circ(x+\parTwo)=x^{\dg}+(\dg \parTwo+f_{\dg-1})x^{\dg-1}+O(x^{\dg-2})$. We
  call $f$ \emph{second-normalized} if $f_{\dg-1}=0$.  (This has been
  used at least since the times of Cardano and Tartaglia.) For any
  $f,$ the composition $f\circ(x-f_{\dg-1}/\dg)$ is second-normalized,
  and if
  \begin{align}\label{eq:secnorm}
    \deg g = m\text{ and } f = g\circ
    h=x^{\dg}+mh_{\dg/m-1}x^{\dg-1}+O(x^{\dg-2})
  \end{align}
  is second-normalized, then so is $h$ (but not necessarily $g$).
\end{remark}

\begin{corollary}\label{cor:andsecond} In \ref{th:fifi}, if
  $p\nmid\dg$ and $f$ is second-normalized, then all claims hold with
  $\parTwo=0$. 
\end{corollary}

\begin{example}
  We note two instances of misreading Ritt's Second Theorem.
  \cite{boddeb09} claim in the proof of their Lemma 5.8 that $t \leq
  q^{5}$ in the situation of \ref{lem:div-1}. This contradicts the
  fact that the exponent $s+3$ of $q$ is unbounded. A
  second instance is in \cite{cor90}. The author claims that his
following example contradicts the
Theorem. He takes (in our language) positive integers $b$, $c$, $d$,
$t$, sets $m = bp^{c}+d$, and $l=p^{c}+1$, elements $h_{0}, \ldots,
h_{t} \in F$, where $c < p$ and $tl \leq m$ and $F$ is a field of
characteristic $p >0$, and
\begin{align*}
h & = \sum_{0 \leq i \leq t}h_{i}x^{m-il},\\
g^{*} & = \sum_{o \leq i, j \leq t} h_{i}h_{j}x^{m-ip^{c}-j}.
\end{align*}
Then
$$
x^{l} \circ h = g^{*} \circ x^{l},
$$
provided that all $h_{i}$ are in $\mathbb{F}_{p^{n}}$. If $d > b$, we
have $m=bl+(d-b)$, so that $s=b$ and $k=d-b$.

Applying \ref{th:fifi}, we find $w = \sum_{0 \leq i \leq
  t}h_{i}x^{b-i}$ and $a=0$. Then
\begin{align*}
h & = x^{4}w(x^{l}),\\
g^{*}& = x^{k}w^{l}.
\end{align*}
Thus the example falls well within Ritt's Second Theorem. \cite{zan93}
points out that this was also remarked by A. Kondracki, a student of
Andrzej Schinzel.
\end{example}

For the arguments below, it is convenient to assume $F$ to be perfect.
Then each element of $F$ has a $p$th root, where $p\geq 2$ is the
characteristic. Any finite field is perfect.

For the next result, we have to make the first condition in
\ref{eq:unidet} more explicit.

\begin{lemma}\label{lem:root}
  Let $F$ be a perfect field, $l$ and $m$ positive integers
  with $\gcd(l,m)=1$, $m=ls+k$ and $s=tp+r$ divisions with remainder,
  so that $1 \leq k < l$ and $0 \leq r < p$,
 and $w \in F[x]$ monic of degree $s$. Then
  \begin{equation}\label{eq:lxw}
   p \nmid l  \text{ and } kw + lxw'= 0  \Longleftrightarrow p \mid m \text{
      and } \exists u \in F[x] \quad w=x^{r} u^{p}, u \text{ monic}.
  \end{equation}
\end{lemma}
If the conditions in \ref{eq:lxw} are satisfied, then $u$ is uniquely
determined. 
\begin{proof}
  For ``$\Longrightarrow$'', we denote by $w^{(i)}$ the $i$th
  derivative of $w$. By induction on $i\geq 0$, we find that
  \begin{align*}
    (k+il) w^{(i)}+ lxw^{(i+1)}  &= 0,\\
    (k+il) w^{(i)}(0) &= 0.
  \end{align*}
  Now $p \nmid s-i$ for $0 \leq i < r$, $p \mid m = k+ls= ~lc
  (kw+lxw')$, and $p \nmid l$.  Thus
  $$
  p \nmid m -(s-i)l=k+ls-ls+il=k+il
  $$
  for $0 \leq i < r$, and hence $w^{(i)}(0)= 0$ for these $i$.  Since
  $r < p$, this implies that the lowest $r$ coefficients of $w$
  vanish, so that $x^{r}\mid w$ and $v = x^{-r}w \in F[x]$. Then
  \begin{align*}
    lv'&= l(-rx^{-r-1}w+x^{-r}w')= x^{-r-1}(-lrw-kw)\\
    &= -x^{-r-1}w\cdot(lr+k)= -x^{-r-1}w \cdot(m-l(s-r))=0.
  \end{align*}
  This implies that $v'=0$ and $v = u^{p}$ for some
  $u\in F[x]$, since $F$ is perfect.

  For ``$\Longleftarrow$'', $p\nmid l$ follows from $\gcd(l,m)=1$, and
  we verify
  \begin{align*}
    kw + lxw' &= kx^{r}u^{p}+ lx \cdot rx^{r-1} u^{p} =
    x^{r} u^{p}(k+lr)\\
    &= w \cdot (m-l(s-r))=0.
  \end{align*}
The uniqueness of $u$ is immediate, since $x^{r}u^{p}= x^{r}
\tilde{u}^{p}$ implies $u = \tilde{u}$.
\end{proof}

We can now estimate the number of distinct-degree
collisions. If $p\nmid m$, the bound is exact. We use Kronecker's $\delta$ in
the statement. 
\begin{corollary}\label{lem:div-b}
  Let $\mathbb{F}_{q}$ be a finite field of characteristic $p$, let
  $l$ and $m$ be integers with $m>l\geq 2$ and $\gcd(l,m)=1$, $n=lm$,
  $s= \lfloor m/l \rfloor$, and $t =\#(D_{\dg, l} \cap D_{\dg,m} \cap
  D_{\dg}^{+})$.  Then the following hold.
  \begin{enumerate}
  \item\label{lem:div-1/2} \label{lem:div-1} If $p \nmid \dg$, then
    $$
    t = (q^{s+3}+ (1-\delta_{l,2})
    (q^{4}-q^{3}))(1-q^{-1}),
    $$
    $$
    q^{s+3}(1-q^{-1})\leq t \leq (q^{s+3}+ q^{4})(1-q^{-1}).
    $$
  \item\label{lem:div-1/4} If $p \mid l$, then $t=0$.
  \item\label{lem:div-1/3} If $p \mid m$, then
    $$
    t \leq (q^{s+3}-q^{\lfloor s/p\rfloor +3})(1-q^{-1}).
    $$
  \end{enumerate}
\end{corollary}
\begin{proof}
  \short\ref{lem:div-1/2} The monic original polynomials $f \in
  D_{\dg,l} \cap D_{\dg,m} \cap D_{\dg}^{+}=T$ fall either into the
  First or the Second Case of Ritt's Second Theorem. In the First
  Case, such $f$ are injectively parametrized by $(w,\parTwo)$ in
  \ref{th:fifi-1}.  Condition \ref{eq:unidet} is satisfied, since $p
  \nmid m= k+ls =~lc(kw+lxw')$. Thus there are $q^{s+1}$ such
  pairs. Allowing composition by an arbitrary linear polynomial on the
  left, we get $q^{s+3}(1-q^{-1})$ elements of $T$. In the Second
  Case, we have the parameters $(z,\parTwo), q^{2}(1-q^{-1})$ in number,
  from \ref{th:fifi-2}. Composing with a linear polynomial yields a total of
  $q^{4}(1-q^{-1})^{2}$. Furthermore, \ref{th:fifi-3} says that $t$
  equals the sum of the two contributions if $l \geq 3$, and it equals
  the first summand for $l=2$; in the letter case, we have
  $p\neq 2$. Both claims in \short\ref{lem:div-1} follow.

  \short\ref{lem:div-1/4} \ref{eq:unidet} and \ref{eq:ab} are never
  satisfied, so that $t=0$.

  \short\ref{lem:div-1/3} We have essentially the same situation as in
  \short\ref{lem:div-1/2}, with $p \nmid l$ and $(w,\parTwo)$
  parametrizing our $f$ in the First Case, albeit not
  injectively. Thus we only obtain an upper bound. The first condition
  in \ref{eq:unidet} holds if and only if $w$ is not of the form
  $x^{r} u^{p}$ as in \ref{eq:lxw}. We note that $\deg u= (s-r)/p =
  \lfloor s/p \rfloor$ in \ref{eq:lxw}, so that the number of
  $(w,\parTwo)$ satisfying \ref{eq:unidet} equals $q^{s+1}-q^{\lfloor
    s/p\rfloor+1}$. Since $p \mid m \mid \dg$, \ref{eq:ab} does not
  hold, and there is no non-Frobenius decomposition in the Second
  Case.\qed
\end{proof}

\begin{example}
  We note two instances of misreading Ritt's Second Theorem.
  \cite{boddeb09} claim in the proof of their Lemma 5.8 that $t \leq
  q^{5}$ in the situation of \ref{lem:div-1}. This contradicts the
  correct statement, where the exponent $s+3$ of $q$ is unbounded. A
  second instance is in \cite{cor90}. The author claims that his
  following example contradicts the Theorem. He takes (in our
  language) positive integers $b$, $c$, $d$, $t$ and elements $h_{0},
  \ldots, h_{t} \in F$ and sets $m = bp^{c}+d$ and $l=p^{c}+1$, where
  $c < p$, $tl \leq m$, and $F$ is a field of characteristic $p
  >0$. Then for
\begin{align*}
h & = \sum_{0 \leq i \leq t}h_{i}x^{m-il},\\
g^{*} & = \sum_{o \leq i, j \leq t} h_{i}h_{j}x^{m-ip^{c}-j}.
\end{align*}
we have
$$
x^{l} \circ h = g^{*} \circ x^{l},
$$
provided that all $h_{i}$ are in $\mathbb{F}_{p^{n}}$. If $d > b$, we
have $m=bl+(d-b)$, so that $s=b$ and $k=d-b$.
Applying \ref{th:fifi}, we find $w = \sum_{0 \leq i \leq
  t}h_{i}x^{b-i}$ and $a=0$. Then
\begin{align*}
h & = x^{k}w(x^{l}),\\
g^{*}& = x^{k}w^{l}.
\end{align*}
Thus the example falls well within Ritt's Second Theorem. \cite{zan93}
points out that this was also remarked by A. Kondracki, a student of
Andrzej Schinzel.
\end{example}

\begin{lemma}\label{lem:propdivi}
  Let $F$ be a perfect field, let $l$, $m \geq 2$ be integers
  for which $p$ divides $n= lm$, and let $g$ and $h$ in $F[x]$ have
  degrees $l$ and $m$, respectively. Then the following hold.
  \begin{enumerate}
  \item\label{lem:propdivi-1} $g \circ h \in D_{\dg}^{\varphi}
    \Longleftrightarrow g'h'=0 \Longleftrightarrow g \in
    D_{l}^{\varphi}$ or $h \in D_{m}^{\varphi}$,
  \item\label{lem:propdivi-2} $\#D_{\dg}^{\varphi} =
    q^{\dg/p+1}(1-q^{-1})$,
  \item\label{lem:propdivi-3}
    \begin{equation*}
      \#D_{\dg,l}^{\varphi}
      \begin{cases}
        = \#D_{\dg/p,l} & \text{if } p \nmid l,\\
       = \#D_{\dg/p,l/p} & \text{if } p \nmid m,\\
       \leq \#D_{\dg/p,l} + \#D_{\dg/p, l/p} & \text{always}.
      \end{cases}
    \end{equation*}
  \end{enumerate}
\end{lemma}
\begin{proof}
  \bare\ref{lem:propdivi-1} is clear. For \ref{lem:propdivi-2}, all
  Frobenius compositions are of the form $g^{*} \circ x^{p}$ with
  $g^{*} \in P^{=}_{\dg/p}$, and $g^{*}$ is uniquely determined by the
  composition. In \ref{lem:propdivi-3}, if $p \nmid l$, then $p \mid
  m$, and according to \ref{eq:frob}, any $g \circ h \in
  D_{\dg,l}^{\varphi}$ can be uniquely rewritten as $g \circ h^{*} \circ
  x^{p}$, with $h^{*} \in P^{0}_{m/p}$. If $p \nmid m$, then the
  corresponding argument works. For the third line, we may assume that
  $p$ divides $l$ and $m$, and then have both possibilities above for
  Frobenius compositions.
\end{proof}

A particular strength of Zannier's and Schinzel's result in \ref{RST} is that,
contrary to earlier versions, the characteristic of $F$ appears only
very mildly, namely in \ref{eq:RST-2}. We now elucidate the case
excluded by \ref{eq:RST-2}, namely $g'(g^{*})'=0$, which is mentioned
in \cite{zan93}, page 178. This case can only occur when $p \geq
2$. We recall the Frobenius power $\varphi_{j} \colon F[x]\rightarrow
F[x]$ from \ref{rem:coll}.
\begin{lemma}\label{fact}
  In the above notation, assume that $(l,m,g,h,g^{*},h^{*})$ and $f$
  satisfy \ref{eq:RST-1}, \ref{eq:RST-3}, and \ref{eq:intersec}, and
  that $F$ is perfect.
  \begin{enumerate}
  \item\label{fact-0} The following are equivalent:
    \begin{enumerate}
    \item $f $ is a Frobenius composition,
    \item $ f' = 0$,
    \item $g'(g^{*})' = 0$.
    \end{enumerate}
  \item\label{fact-1} If $g'=0$, then $p \nmid l$ and $(g^{*})' \neq
    0$, and there exist positive integers $j$ and
    $M$, and monic original $G$, $G^{*}$, $H^{*} \in F[x]$ so that
    \begin{align}\label{al:fact}
      \begin{aligned}
        & m =p^{j}M, \deg G = \deg H^{*} = M, \deg G^{*}=l,\\
       & g= x^{p^{j}} \circ G, g^{*} \circ x^{p^{j}} = x^{p^{j}} \circ
       G^{*}, h^{*} = x^{p^{j}} \circ H^{*},\\
        & G' (G^{*})' \neq 0, G \circ h=G^{*} \circ H^{*},\\
        & f = x^{p^{j}} \circ G \circ h = x^{p^{j}} \circ (G^{*} \circ
        H^{*}).
      \end{aligned}
    \end{align}
    In particular, $(l, M, G, h, G^{*}, H^{*})$ satisfies
    \ref{eq:RST-1} through \ref{eq:RST-3} if $M > l$, and $(M,l,G^{*},
    H^{*}, G,h)$ does if $2 \leq M < l$. If $M = 1$, then $G$ and
    $H^{*}$ are linear.
  \item\label{fact-2} If $(g^{*})' = 0$, then $p \nmid m$ and $g' \neq
    0$, and there exist positive integers $d$ and $L$, and monic
    original $G,H, G^{*} \in F[x]$ with
    \begin{align}\label{eq:fgx}
    l&=p^{d}L, p \nmid L, g = \varphi_{d}(G), h = x^{p^{d}}\circ H, g^{*}=
    x^{p^{d}} \circ G^{*},\nonumber\\
    G'(G^{*})' &\neq 0,\\
     G \circ H& = G^{*} \circ h^{*}, f = x^{p^{d}} \circ G \circ H.\nonumber
    \end{align}
    with $\varphi_{d}$ from \ref{rem:coll}.
   In particular, $(L,m, G, H, G^{*}, h^{*})$ satisfies
    \ref{eq:RST-1} through \ref{eq:RST-3} if $L \geq 2$.
  \item\label{fact-4} The data derived in \short\ref{fact-1} and
    \short\ref{fact-2} are uniquely determined. Conversely, given such
    data, the stated formulas yield $(l,m,g,h,g^{*}, h^{*})$ and $f$
    that satisfy \ref{eq:RST-1}, \ref{eq:RST-3}, and \ref{eq:intersec}.

  \end{enumerate}
\end{lemma}
\begin{proof}
  \short\ref{fact-0} If $f = x^{p} \circ G$ is a Frobenius
  composition, then $f'= 0$. We have
  \begin{equation}\label{eq:g*}
    f' = (g' \circ h) \cdot h' = ((g^{*})' \circ h^{*}) \cdot (h^{*})'.
  \end{equation}
  If (b) holds, then $p\mid \deg f = \dg = lm$, hence $p \mid l$ or
  $p \mid m$. In the case $p \mid l$, \ref{eq:RST-1} implies that $p
  \nmid m$ and $g'(h^{*})' \neq 0$, hence $h'= (g^{*})'=0$ by
  \ref{eq:g*}. Symmetrically, $p \mid m $ implies that $g' =
  (h^{*})'=0$, so that (c) follows in both cases.

  If (c) holds, say $g'=0$, then the coefficient of $x^{i}$ in $g$
  is zero unless $p \mid i$. Since $F$ is perfect, every element
  has a $p$th root, and it follows
  that $g = x^{p}\circ G$ for some $G \in F[x]$. Thus $g$ is a
  Frobenius composition, and so is $f=g \circ h$.

  \short\ref{fact-1} Let $j \geq 1$ be the largest integer for which
  there exists some $G \in F[x]$ with $g = x^{{p^{j}}} \circ G$. Then
  $j$ and $G$ are uniquely determined, $G$ is monic and original, $G'
  \neq 0$, $p^{j} \mid m$, $\deg G =mp^{-j}=M$, and $p \nmid l$ by
  \ref{eq:RST-1}. Furthermore, we have
  \begin{equation}\label{eq:Gh}
    g^{*} \circ h^{*} = g \circ h = x^{p^{j}}\circ G \circ h.
  \end{equation}

  Writing $h^{*} = \sum_{1 \leq i \leq m} h_{i}^{*}x^{i}$ with
  $h_{m}^{*}= 1$, we let $I=\{i \leq m \colon h^{*}_{i} \neq 0\}$ be the
  support of $h^{*}$. Assume that there is some $i \in I$ with $p^{j}
  \nodiv i$, and let $k$ be the largest such $i$. Then $k < m$,
  $m(l-1)+k$ is not divisible by $p^{j}$, the coefficient of
  $x^{m(l-1)+k}$ in $(h^{*})^{l}$ is $lh^{*}_{k}$, and in $g^{*} \circ
  h^{*}$ it is $ ~lc(g^{*}) \cdot lh^{*}_{k}\neq 0$; see $E_{1}$ in
  \ref{lem:int}. This contradicts \ref{eq:Gh}, so that the assumption
  is false and $h^{*}=(H^{*})^{p^{j}}$ for a unique monic original
  $H^{*} \in F[x]$, of degree $M =m p^{-j}$.

  Setting $G^{*}=\varphi_{j}^{-1}(g^{*})$, we have $\deg G^{*} = \deg
  g^{*}=l$ and hence $(G^{*})' \neq 0$, $ x^{p^{j}}\circ G^{*}=
  \varphi_{j} (G^{*})\circ x^{p^{j}}$, and
  $$
  x^{p^{j}}\circ G \circ h=g \circ h= f = g^{*} \circ h^{*}=
  \varphi_{j}(G^{*})\circ x^{p^{j}} \circ H^{*}=x^{p^{j}} \circ G^{*}
  \circ H^{*},
  $$
  $$
  G\circ h= G^{*} \circ H^{*}.
  $$

  \short\ref{fact-2} Since $p \mid l = \deg g^{*}$, \ref{eq:RST-1}
  implies that $p \nmid m$, $g' \neq 0$, and $g' \circ h \neq 0$. In
  \ref{eq:g*}, we have $f'=0$ and hence $h' =0$. There exist monic
  original $G_{1}$, $H_{1} \in F[x]$ with $g^{*}= x^{p} \circ G_{1}$,
  $h = x^{p} \circ H_{1}$, and
  \begin{align*}
    x^{p} \circ G_{1} \circ h^{*}  &= f= g \circ x^{p}\circ
    H_{1}=x^{p}
    \circ \varphi_{1}^{-1}(g) \circ H_{1},\\
    G_{1} \circ h^{*}  &= \varphi_{1}^{-1}(g) \circ H_{1}.
  \end{align*}
  If $G_{1}' = 0$, then $H'_{1}=0$ and we can continue this
  transformation. Eventually we find an integer $j \geq 1$ and monic
  original $G_{j}$, $H_{j} \in F[x]$ with $p^{j} \mid l$, $g^{*}=
  x^{p^{j}} \circ G_{j}$, $h = x^{p^{j}} \circ H_{j}$, and $G_{j}'
  \neq 0$. We set $G = \varphi_{j}^{-1}(g), G^{*}= G_{j}$, and
  $H=H_{j}$. Then $G'(G^{*})'\neq 0,\; \deg G^{*} = \deg H = L, \deg G
  = m$. As above, we have
  $$
  G^{*}\circ h^{*}= G_{j} \circ h^{*} = \varphi_{j}^{-1}(g) \circ
  H_{j}= G \circ H,
  $$
  $$
  f = (x^{p^{d}} \circ G^{*}) \circ h^{*} = g \circ (x^{p^{d}} \circ
  H) = x^{p^{d}} \circ G \circ H.
  $$

  According to \short\ref{fact-2}, $d$ is the multiplicity of $p$ in
  $l$. We now show that $j=d$. We set $l^{*} = lp^{-j}$. If $l^{*}
  \geq 2$, then the above collision satisfies the assumptions
  \ref{eq:RST-1} through \ref{eq:RST-3}, with $l^{*} < m$ instead of
  $l$. Thus \ref{th:fifi} applies.

  In the First Case, \ref{eq:unidet} shows that $p \nmid l^{*}$. It
  follows that $j = d$ and $l^{*}=L$. In the Second Case, we have $p
  \nmid l^{*}m = lp^{-j}m$ by \ref{eq:mopo}, so that again $j=d$ and
  $l^{*}=L$. In the remaining case $l^{*}=1$, we have $L=1$ and
  $G^{*}=H=x$. 

  \short\ref{fact-4} The uniqueness of all quantities is clear.
\end{proof}

We need some simple properties of the Frobenius
map $\varphi_{j}$ from \ref{eq:frob}.

\begin{lemma}\label{lem:Frobmap}

  Let F be a field of characteristic $p\geq2$, $f$, $g\in F[x]$, $\parTwo\in
  F$, let $ i$, $j \geq 1$, and denote by $f'$ the derivative of $f$. Then
  \begin{enumerate}
  \item \label{lem:Frobmap-1}
    $\varphi_{j}(fg)=\varphi_{j}(f)\varphi_{j}(g)$,
  \item\label{lem:Frobmap-2}
    $\varphi_{j}(f^{i})={\varphi_{j}(f)}^{i}$,
  \item\label{lem:Frobmap-3} $\varphi_{j}(f\circ
    g)=\varphi_{j}(f)\circ \varphi_{j}(g)$,
  \item\label{lem:Frobmap-4}
    $\varphi_{j}(f(a))=\varphi_{j}(f)(\parTwo^{p^{j}})$,
  \item\label{lem:Frobmap-5} $\varphi_{j}(f')=\varphi_{j}(f)'$.
  \end{enumerate}
\end{lemma}
\begin{proof}\ref{lem:Frobmap-1} is immediate, and
  \ref{lem:Frobmap-2} follows. For \ref{lem:Frobmap-3}, we write
  $f=\sum{f_{i}x^{i}}$ with all $ f_{i}\in
  F$. Then
 $$\varphi_{j}(f\circ
 g)=\varphi_{j}(\sum{f_{i}g^{i}})=\sum{f_{i}^{p^{j}}\varphi_{j}(g^{i})}=
 \varphi_{j}(f)\circ\varphi_{j}(g).
$$
\ref{lem:Frobmap-4} is a special case of \ref{lem:Frobmap-3}. For
\ref{lem:Frobmap-5}, we have
  $$
   \varphi_{j}(f{'})=\varphi_{j}(\sum{if_{i}x^{i-1}})=\sum{i^{p^{j}}f_{i}^{p^{j}}x^{i-1}}=\sum{if_{j}^{p^{j}}x^{i-1}}=\varphi_{j}(f)'.
  \qed $$
\end{proof}

Our next goal is to get rid of the assumption \ref{eq:RST-2}, namely
that $g'(g^{*})' \neq 0$, in \ref{th:fifi}. This is achieved by the
following result. Its statement is lengthy, and the simple version
is: if \ref{eq:RST-2} is violated, remove the component $x^{p}$ from the
culprit as long as you can. Then \ref{th:fifi} applies.

\begin{theorem}\label{lem:div2a}
  Let $F$ be a perfect field of characteristic $p \geq 0$. Let $m > l \geq 2$
  be integers with $\gcd(l,m)=1$, set $\dg = lm$ and let $f,g,h,g^{*},
  h^{*} \in F[x]$ be monic original of degrees $n$, $m$, $l$, $l$,
  $m$, respectively, with $f=g \circ h = g^{*} \circ h^{*}$. Then the
  following hold.
  \begin{enumerate}
  \item\label{lem:div2a-i} If $g'=0$, then there exists a uniquely
    determined positive integer $j$ so that $p^{j}$ divides $m$ and
    either (\bare\ref{lem:div2a-i/1}) or (\bare\ref{lem:div2a-i/2})
    hold; furthermore, (\bare\ref{lem:div2a-i/3}) is true. We set $M=p^{-j}m$.
    \begin{enumerate}
    \item\label{lem:div2a-i/1} (First Case)
      \begin{enumerate}
      \item If $M>l$, then there exist a monic $W \in F[x]$ of
        degree $S = \lfloor M/l \rfloor$ and $\parTwo \in F$ so that
        $$
        KW +lxW'\neq 0
        $$
        for $K = M-l \lfloor M/l \rfloor$, and all conclusions of
        \ref{th:fifi-1}, except \ref{eq:unidet} and $k<l$, hold for $k
        = p^{j}K$, $s=p^{j}S$, and $w=W^{p^{j}}$.
        Conversely, any $W$ and $a$ as above yield via these formulas
        a collision satisfying \ref{eq:RST-1}, \ref{eq:RST-3} and
        \ref{eq:intersec}, with $g'=0$. If $p \nmid M$, then $W$ and
        $a$ are uniquely determined by $f$ and $l$.
      \item If $M <l$, then there exist a monic $W \in F[x]$ of
        degree $S = \lfloor l/M \rfloor$ and $\parTwo \in F$ so that
     $$
     f = (x-\parTwo^{kM}w^{M}(\parTwo^{M})) \circ x^{kM}
     w^{M}(x^{M}) \circ (x+\parTwo),
     $$
    $$
        KW +lxW'\neq 0
        $$
        for $K = l-M\lfloor l/M\rfloor$, and all conclusions of
        \ref{th:fifi-1}, with $l$ replaced by $M$ and excepting
        \ref{eq:unidet} and the division with remainder, hold for $k =
        p^{j}K$, $s=p^{j}S$, and $w=W^{p^{j}}$.  Conversely, any $W$
        and $a$ as above yield via these formulas a collision
        satisfying \ref{eq:RST-1}, \ref{eq:RST-3} and
        \ref{eq:intersec}, with $g'=0$. Furthermore, $W$ and
        $a$ are uniquely determined by $f$ and $l$.
      \item If $m=p^{j}$, then $g=h^{*}= x^{p^{j}}$ and $g^{*}=
        \varphi_{j}(h)$.
      \end{enumerate}
    \item\label{lem:div2a-i/2} (Second Case) $p \nmid M$, and all
      conclusions of \ref{th:fifi-2} 
      hold, except \ref{eq:ab}. 
    \item\label{lem:div2a-i/3}
      Assume that $M \geq 2$, and let $f$ be a collision of the Second
      Case. Then $f$ belongs to the First Case if and only if $~min(l,M)=2$.
    \end{enumerate}
  \item\label{lemdiv2a-ii} If $(g^{*})'=0$, then there exists a unique
    positive integer $d$ such that $p^{d} \mid l$, $p \nmid
    p^{-d}l=L$, and either (\bare\ref{lemdiv2a-ii/1}) or
    (\bare\ref{lemdiv2a-ii/2}) holds; furthermore,
    (\bare\ref{lemdiv2a-ii/3}) is true. 
\begin{enumerate}
\item\label{lemdiv2a-ii/1}(First Case) There exist a monic $ w \in
  F[x]$ of degree $\lfloor m/L\rfloor$ and $\parTwo \in F $ so that
    \begin{align*}
      f&=(x-a^{kl}w^{L}(a^{l}))\circ x^{kl}w^{L}(x^{l})\circ(x+a),\\
      g&=(x-\parTwo^{kl}w^{L}(\parTwo^{l})) \circ x^{k}w^{L} \circ (x+\parTwo^{l}),\\
      h&=(x-\parTwo^{l}) \circ x^{l} \circ (x+\parTwo),\\
      g^{*}&=(x-\parTwo^{kl}w^{L}(\parTwo^{l}))\circ x^{l} \circ
      (x+\parTwo^{k}\varphi_{d}^{-1} (w)(\parTwo^{L})),\\
      h^{*}&=(x-\parTwo^{k} \varphi_{d}^{-1}(w)(\parTwo^{L}))\circ
      x^{k}\varphi_{d}^{-1}(w)(x^{L})\circ(x+\parTwo),
    \end{align*}
    where $m = L \lfloor m/L\rfloor + k$. The quantities $w$ and $a$
    are uniquely determined by $f$ and $l$. Conversely, any $w$ and
    $a$ as above yield via these formulas a collision satisfying
    \ref{eq:RST-1}, \ref{eq:RST-3}, and
    \ref{eq:intersec}. Furthermore, $kw+lxw' \neq 0$.

\item\label{lemdiv2a-ii/2} (Second Case) There exist $z,a\in F$ with
    $z\neq 0$ for which all conclusions of \ref{th:fifi-2} hold, except
    \ref{eq:ab}. Conversely, any $(z,a)$ as above yields a collision
    satisfying \ref{eq:RST-1}, \ref{eq:RST-3} and \ref{eq:intersec}.
     \item\label{lemdiv2a-ii/3} When $L\geq 3$, then
       (\bare\ref{lemdiv2a-ii/1}) and (\bare\ref{lemdiv2a-ii/2}) are
       mutually exclusive. For $L\leq 2$, (\bare\ref{lemdiv2a-ii/2}) is
       included in (\bare\ref{lemdiv2a-ii/1}).
  \end{enumerate}
\end{enumerate}
\end{theorem}

\begin{proof}
  \short\ref{lem:div2a-i} We take the quantities $j$, $M$, $G$,
  $G^{*}$, $H^{*}$ from \ref{fact-1} and apply \ref{th:fifi} to the
  collision $G\circ h=G^{*}\circ H^{*}$ in \ref{al:fact}. We start
  with the First Case (\ref{th:fifi-1}). If $M > l$, it yields a monic
  $W \in F [x]$ of degree $\lfloor M/l \rfloor$ and $\parTwo \in F$
  with
  \begin{align}
    &G\circ h=G^{*}\circ h^{*}=(x-\parTwo^{*})\circ x^{Kl}W^{l}(x^{l})\circ (x+\parTwo), \nonumber \\
    & KW + lxW' \neq 0,
  \end{align}
  where $K=M-l\lfloor M/l \rfloor$ and $\parTwo^{*}=\parTwo^{Kl}W^{l}(\parTwo^{l})$.  We
  set $k=p^{j}K$ and $w=W^{p^{j}}$. Then
  \begin{align*}
    f&=g \circ  h =G^{p^{j}}\circ h = x^{p^{j}}\circ G \circ h \\
    &= x^{p^{j}}\circ(x-\parTwo^{*})\circ x^{Kl}W^{l}(x^{l})\circ(x+\parTwo)\\
    &=\bigl( x-(\parTwo^{*})^{p^{j}}\bigr) \circ
    x^{p^{j}Kl}(W^{p^{j}})^{l}(x^{l})\circ(x+\parTwo)\\
    &= (x-\parTwo^{kl}w^{l}(\parTwo^{l}))\circ x^{kl}w^{l}(\parTwo^{l})\circ(x+\parTwo).
  \end{align*}

Furthermore, we have
$$
ls + k = lp^{j}\lfloor M/l \rfloor +p^{j}(M-l\lfloor M/l\rfloor)=m.
$$

If $2 \leq M < l$, we have to reverse the roles of $M$ and $l$ in the
application of \ref{th:fifi-1}. Thus we now find a monic $W \in F[x]$
of degree $\lfloor l/M \rfloor$ and $\parTwo\in F$ with
  $$
  G \circ h = (x-a^{*}) \circ x^{KM}W^{M}(x^{M}) \circ (x+\parTwo),
  $$
  with $K = l-M \lfloor l/M\rfloor$, $ a^{*}=\parTwo^{KM}W^{M}(\parTwo^{M})$, and
  $KW +MxW' \neq 0$. We set $k=p^{j}K$ and $w=W^{p^{j}}$. Then
  \begin{align*}
    f  &= x^{p^{j}} \circ G \circ h = \varphi_{j}(x-a^{*}) \circ
    x^{p^{j}} \circ x^{KM}W^{M}(x^{M}) \circ (x+\parTwo)\\
   &= (x-\parTwo^{kM}w^{M}(\parTwo^{M})) \circ x^{kM} w^{M}(x^{M})
    \circ (x+\parTwo).
  \end{align*}
Furthermore we have 
$$
Ms+k = Mp^{j}\lfloor l/M \rfloor + p^{j}(l-M \lfloor l /M \rfloor) = p^{j}l.
$$  
Since $p \nmid l$, $W$ and $a$ are uniquely determined.

If $M=1$, then $g=x^{p^{j}}$, $f=x^{p^{j}} \circ h =
  \varphi_{j}(h)
  \circ x^{p^{j}}$, and $g^{*}= \varphi_{j}(h)$ by \ref{cor:inj-1}.

  In the Second Case of \ref{th:fifi}, we use $T_{p^j} = x^{p^j}$ from
  \ref{Tsqfree-3}. Now \ref{th:fifi-2} provides $z, \parTwo \in F$
  with $z\neq 0$ and
  \begin{align*}
  G \circ h &= G^{*}\circ H^{*}= (x-T_{lM}(\parTwo,z))\circ
  T_{lM}(x,z)\circ(x+\parTwo),\\
  G &= (x-T_{lM}(a,z)) \circ T_{M}(x,z^{l}) \circ (x+T_{l}(a,z)).
\end{align*}
Since $G' \neq 0$, we have $p\nmid M$, and hence $p \nmid lM$. Thus
$z$ and $a$ are uniquely determined. Furthermore
\begin{align*}
    f&=g\circ h = x^{p^{j}} \circ G\circ h\\
    &=(x^{p^{j}}-(T_{lM}(\parTwo,z))^{p^{j}})\circ T_{lM}(x,z)\circ (x+\parTwo)\\
    &=(x- T_{\dg}(\parTwo,z))\circ x^{p^{j}}\circ T_{lM}(x,z)\circ (x+\parTwo)\\
    &=(x-T_{\dg}(\parTwo,z))\circ T_{\dg}(x,z)\circ(x+\parTwo).
  \end{align*}

In \short\ref{lem:div2a-i/3}, we have $p \nmid lM = p^{-j}\dg$. By
\ref{th:fifi-3}, $G \circ h$ belongs to the First Case if and only if
$~min \{l,M\}=2$.

\short\ref{lemdiv2a-ii} We take $d$, $L$, $G$, $H$, $G^{*}$ from
\ref{fact-2}, and apply \ref{th:fifi} to the collision $G \circ H =
G^{*} \circ h^{*}$. In the First Case, this yields a monic $W \in
F[x]$ of degree $\lfloor m/L \rfloor$ and $\parTwo \in F$ so that the
conclusions of \ref{th:fifi-1} hold for these values, with $k =
m-L \cdot \lfloor m / L \rfloor$. We set $w =
\varphi_{d}(W)$. Then
  \begin{align*}
    \deg G  &= \deg (x^{k}W^{L})= (m-L\cdot \lfloor m / L
    \rfloor) + L \cdot \lfloor m / L \rfloor = m,\\ 
    g  &= \varphi_{d}(G) = \varphi_{d}\bigl ((x-\parTwo^{kL}W^{L}(\parTwo^{L}))\circ
    x^{k}W^{L}
    \circ (x+\parTwo^{L})\bigr )\\
    &= \varphi_{d}(x-\parTwo^{kL}W^{L}(\parTwo^{L})) \circ
    \varphi_{d}(x^{k}W^{L})
    \circ \varphi_{d}(x+\parTwo^{L})\\
    &=(x-\parTwo^{kl}w^{L}(\parTwo^{l})) \circ x^{k}w^{L} \circ (x+\parTwo^{l}).\\
    h &= x^{p^{d}} \circ H = x^{p^{d}} \circ (x-\parTwo^{L}) \circ x^{L}
    \circ
    (x+\parTwo)\\
    &= (x-\parTwo^{l}) \circ x^{l} \circ (x+\parTwo),\\
    g^{*} &= x^{p^{d}} \circ G^{*} = x^{p^{d}} \circ
    (x-\parTwo^{kL}W^{L}(\parTwo^{L}))
    \circ x^{L} \circ (x+\parTwo^{k}W(\parTwo^{L}))\\
    &= (x-\parTwo^{kl} W^{p^{d}L}(\parTwo^{L})) \circ x^{l} \circ (x+\parTwo^{k}W(\parTwo^{L}))\\
    &= (x-\parTwo^{kl}w^{L}(\parTwo^{l}))\circ x^{l} \circ (x+\parTwo^{k}
    \varphi_{d}^{-1}(w)(\parTwo^{L})),\\
    h^{*} &= (x-\parTwo^{k}W(\parTwo^{L})) \circ x^{k}W(x^{L}) \circ (x+\parTwo)\\
    &= (x-\parTwo^{k} \varphi_{d}^{-1}(w)(\parTwo^{L})) \circ x^{k}
    \varphi_{d}^{-1}(w)(x^{L}) \circ (x+\parTwo),\\
f &= (x-a^{kl}w^{L}(a^{l})) \circ x^{kl}w^{L}(x^{l}) \circ (x+a).
  \end{align*}

Furthermore, \ref{lem:Frobmap} implies that
$$
kw + lxw'= k \varphi_{d}(W)+lx \varphi_{d}(W)'=
\varphi_{d}(kW+lxW')\neq 0.
$$
In the Second Case, \ref{th:fifi-2} provides $z,a\in F$ with $z\neq 0$
and
\begin{align*}
g&=\varphi_{d}(G)=\varphi_{d}\bigl((x-T_{mL}(a,z))\circ
T_{m}(x,z^{L})\circ(x+T_{L}(a,z))\bigr)\\
&=\bigl(x-\varphi_{d}(T_{mL}(a,z))\bigr)\circ
\varphi_{d}(T_{m}(x,z^{L}))\circ\bigl(x+\varphi_{d}(T_{L}(a,z))\bigr)\\
&=(x-T_{mL}(a,z)^{p^{d}})\circ
T_{m}(x,(z^{L})^{p^{d}})\circ(x+T_{L}(a,z)^{p^{d}})\\
&=(x-T_{\dg}(a,z))\circ T_{m}(x,z^{l})\circ (x+T_{l}(a,z)),\\
h&=x^{p^{d}}\circ H=x^{p^{d}}\circ(x-T_{L}(a,z))\circ
T_{L}(x,z)\circ(x+a)\\
&=(x-T_{l}(a,z))\circ T_{l}(x,z)\circ (x+a),\\
g^{*}&=x^{p^{d}}\circ G^{*}=x^{p^{d}}\circ(x-T_{Lm}(a,z))\circ
T_{L}(x,z^{m})\circ(x+T_{m}(a,z))\\
&=(x-T_{\dg}(a,z))\circ x^{p^{d}}\circ T_{L}(x,z^{m})\circ
(x+T_{m}(a,z))\\
&=(x-T_{\dg}(a,z))\circ T_{l}(x,z^{m})\circ (x+T_{m}(a,z)),\\
h^{*}&=(x-T_{m}(a,z))\circ T_{m}(x,z)\circ (x+a),\\
f&=(x-T_{\dg}(a,z))\circ T_{\dg}(x,z)\circ (x+a).
\end{align*}

\short\ref{lemdiv2a-ii/3} follows from \ref{th:fifi-3} for $L\geq
2$. If $L=1$, then $l=p^{d}$ and $k=0$ in \short\ref{lemdiv2a-ii/1}. For
any
$$
f=(x-T_{\dg}(a,z))\circ T_{\dg}(x,z)\circ (x+a)
$$
in \short\ref{lemdiv2a-ii/2}, we take $w=T_{m}(x,z^{p^{d}})$. Then 
\begin{align*}
T_{\dg}(x,z)&=T_{m}(x,z^{p^{d}})\circ T_{p^{d}}(x,z)=w\circ
x^{p^{d}},\\
f&=(x-w(a^{l}))\circ w(x^{l})\circ (x+a),
\end{align*}
which is an instance of \short\ref{lemdiv2a-ii/1}.
\end{proof}

If $p \nmid \dg$, then the case where $\gcd (l,m) \neq 1$ is reduced
to the previous one by the following result of \cite{tor88a}. We will
only use the special case where $l=l^{*}$ and $m = m^{*}$.

\begin{fact}
Suppose we have a field $F$ of characteristic $p \geq 0$, integers $l,
l^{*}, m, m^{*}\\ \geq 2$ with $p \nmid lm$, monic original
polynomials $g, h, g^{*}, h^{*} \in F[x]$ of degrees $m, l, l^{*},
m^{*}$, respectively, with $g \circ h = g^{*} \circ
h^{*}$. Furthermore, let $i = \gcd (m, l^{*})$ and $j=\gcd
(l,m^{*})$. Then the following hold.
\begin{enumerate}
\item\label{fact:monopo-1}
There exist monic original polynomials $u, v, \tilde{g}, \tilde{h},
\tilde{g}^{*}, \tilde{h}^{*} \in F[x]$ of degrees $i, j, m/i, l/j,
l^{*}/i, m^{*}/j$, respectively, so that
\begin{align}\label{al:monopo}
g &= u \circ \tilde{g},\nonumber\\
h&= \tilde{h} \circ v,\\
g^{*}& = u \circ \tilde{g}^{*},\nonumber \\
h^{*}&= \tilde{h}^{*} \circ v.\nonumber
\end{align}
\item\label{fact:monopo-2}
Assume that $l = l^{*} < m= m^{*}$. Then $i=j$ and $m/i, l/i,
\tilde{f}=\tilde{g} \circ \tilde{h}, \tilde{g}, \tilde{h},
\tilde{g}^{*}, \tilde{h}^{*}$ satisfy the assumptions of \ref{th:fifi}.
\end{enumerate}
\end{fact}
\begin{proof}
  \short\ref{fact:monopo-1} \cite{tor88a} proves the claim if $F$ is
  algebraically closed, but without the condition of being monic original. Thus we have four decompositions
  \ref{al:monopo} over an algebraic closure of $F$. We may choose all
  six components in \ref{al:monopo} to be monic original. They are
  then uniquely determined. Since $p \nmid \dg$, decomposition is
  rational; see \cite{sch00c}, I.3, Theorem 6, and \cite{kozlan89} or
  \cite{gat90c} for an algorithmic
  proof. It follows that the six components are in $F[x]$. 

\short\ref{fact:monopo-2} We have $\gcd(l/i, m/i)=1$, and
$$
f = (u \circ \tilde{g}) \circ (\tilde{h} \circ v) = (u \circ
\tilde{g}^{*})\circ (\tilde{h}^{*} \circ v).
$$
The uniqueness of tame decompositions (\ref{cor:inj}) implies that
$\tilde{g} \circ \tilde{h} = \tilde{g}^{*} \circ \tilde{h}^{*}$. The
other requirements are immediate.
\end{proof}

\citeauthor{tor88a}'s result, together with the preceding material,
determines $D_{\dg,l} \cap D_{\dg,m}$ completely, if $p \nmid \dg =
lm$.
\begin{corollary}\label{thm:FFC}
  Let $\mathbb{F}_{q}$ be a finite field of characteristic $p$, and
  let $m > l \geq 2$ be integers with $p \nmid \dg = lm$, $i =
  \gcd(l,m)$ and $s=\lfloor m/l \rfloor$.  Let $t=\#(D_{\dg,l} \cap
  D_{\dg,m})$. Then the following hold.
\begin{enumerate}
\item\label{thm:FFC-1}\begin{align*}
 t=
\begin{cases}
q^{2l+s-1}(1-q^{-1}) & \text{if } l \mid m,\\
q^{2i}(q^{s+1}+(1- \delta_{l,2})
(q^{2}-q))(1-q^{-1}) & \text{otherwise}.
\end{cases}
\end{align*}
\item\label{thm:FFC-2} 
$$
t\leq 2q^{2l+s-1}(1-q^{-1}).
$$
\end{enumerate}
\end{corollary}

\begin{proof}
\short\ref{thm:FFC-1} Let $T = D_{\dg,l} \cap D_{\dg,m} \cap
D_{\dg}^{0}$ consist of the monic original polynomials in the
intersection, and similarly $U =
D_{\dg/i^{2}, l/i} \cap D_{\dg/i^{2}, m/i}\cap D^{0}_{\dg/i^{2}}$. Then
\ref{fact:monopo-2} implies that $T = P_{i}^{0} \circ U \circ
P_{i}^{0}$, using $G \circ H = \{g \circ h \colon g \in G, h \in H\}$
for sets $G, H \subseteq F[x]$. Furthermore, the composition maps
involved are injective. Thus
\begin{align*}
 \#T& = (\#P_{i}^{0})^{2}\cdot\# U = q^{2i-2} \cdot\#U,\\
 \#(D_{\dg,l} \cap D_{\dg,m}) &= q^{2}(1-q^{-1}) \cdot q^{2i-2}\cdot\#U.
\end{align*}

If $l \nmid m$, then $l/i \geq 2$ and from \ref{lem:div-1} we have
$$
\#U = \frac{q^{-2}}{1-q^{-1}}
\cdot(q^{s+3}+(1-\delta_{l,2})(q^{4}-q^{3}))(1-q^{-1}),
$$
which implies the claim in this case. If $l \mid m$, then $l/i = 1$ and
\ref{lem:div-b} is inapplicable. Now
\begin{align*}
U &= D_{m/l,1}\cap D_{m/l, m/l}\cap P^{0}_{m/l}= P^{0}_{m/l},\\
\#U  &= \#P^{0}_{m/l}= q^{m/l-1}= q^{s-1},
\end{align*}
which again shows the claim.

\short\ref{thm:FFC-2} We have $q^{2}\leq q^{s+1}$, and if $l\nmid m$,
then $2i\leq l\leq 2l-2$.
\end{proof}

This result shows that there are more polynomials in the intersection
when $l^{2}\mid \dg$ than otherwise.

We now have determined the size of the intersection if either $p\nmid
\dg$ or $~gcd(l,m)=1$. It remains a challenge to do this with the same
precision when both
conditions are violated. The following approach yields a rougher estimate.
\begin{theorem}\label{thm:div2a}
  Let $F$ be a field of characteristic $p\geq 2$, let $l, m ,\dg
  \geq 2$ be integers with $p\mid \dg = lm$, and set $T =D_{\dg, l}
  \cap D_{\dg,m} \cap D_{\dg}^{+}$.  Then the following hold.
  \begin{enumerate}
  \item\label{lem:div-6} If $p \nmid l$, then for any monic original
    $f \in T$ there exist monic original $g^{*}$ and $h^{*}$ in
    $F[x]$ 
    of degrees $l$ and $m$, respectively, with $f = g^{*} \circ
    h^{*}$, $(g^{*})'(h^{*})' \neq 0$, and $0\leq\deg(h^{*})' < m-l$.
    
  \item\label{lem:div-7} If $p \mid l$, then for any monic original
    $f\in T$ there exist monic original $g$ and $h\in F[x]$ of
    degrees $m$ and $l$, respectively, with $f=g\circ h$ and $\deg
    g'\leq \ m-(m+1)/l$.
  \end{enumerate}
\end{theorem}
\begin{proof} We take a collision \ref{eq:intersec} and its derivative
  \ref{eq:g*}.  Since $f \in D_{\dg}^{+}$, we have $f' \neq 0$.

  \short\ref{lem:div-6} Since $p\mid m$, we have $\deg g' \leq m-2$, $(h^{*})' \neq 0$, and
  $\deg(h^{*})' \geq 0$, so that
  \begin{align*}
    \dg - m + \deg(h^{*})' &=(l-1) \cdot m + \deg(h^{*})' = \deg
    f'\\
   &\leq (m-2)\cdot l+l-1 = \dg-l-1,\\
   0&\leq\deg(h^{*})' < m-l.
  \end{align*}
  
  \short\ref{lem:div-7} We have $g'h' \neq 0$,  $\deg(g^{*})'\leq l-2$, $\deg h' \geq 0$,
  and
  \begin{align*}
    l \cdot \deg g' &\leq l \cdot \deg g' + \deg h' = \deg f'
    \\
    &\leq (l-2)\cdot m + m - 1 = lm - m -1,
    \\
    \deg g' &\leq m - \frac{m+1}{l}.\qed
  \end{align*}
\end{proof}
We deduce the following upper bounds on $\#T$.

\begin{corollary}
  \label{cor:ffchar}
  Let $\mathbb{F}_{q}$ be a finite field of characteristic $p$, $l$ a
  prime number dividing $m >l$, assume that $p \mid \dg = lm$, and set
  $t =\#(D_{\dg, l} \cap D_{\dg,m}\cap D_{n}^{+})$.  Then the
  following hold.
  \begin{enumerate}
  \item\label{cor:ffchar-2} If $p \nmid l$, then
    $$
    t \leq q^{m+\lceil l/p \rceil}(1-q^{-1}).
    $$
  \item 
    \label{cor:ffchar-3}
    If $p\mid l$, we set $c=\lceil (m-l+1)/l\rceil$. Then
    $$ 
    t\leq q^{m+l-c+\lceil c/p\rceil}(1-q^{-1}).
    $$
    If $l\mid m$, then $c=m/l$.
 \end{enumerate}
\end{corollary}
\begin{proof}
  \short\ref{cor:ffchar-2} Any $h^{*}$ permitted in \ref{lem:div-6}
  has nonzero coefficients only at $x^{i}$ with $p \mid i$ or $i \leq
  m-l$. Since $p\mid m$, the number of such $i$ is $m-l+\lceil l/p
  \rceil$. Taking into account that $h^{*}$ is monic, the number of
  $g^{*} \circ h^{*}$, composed on the left with a linear polynomial,
  is at most
  $$
  q^{2}(1-q^{-1}) \cdot q^{l-1} \cdot q^{m-l+\lceil l/p \rceil-1}=
  q^{m+\lceil l/p \rceil}(1-q^{-1}).
  $$

  \short\ref{cor:ffchar-3} The polynomials $g$ permitted in
  \ref{lem:div-7} are monic of degree $m$ and satisfy
  \begin{align*}
    \deg g' &\leq m - \frac{m+1}{l},\\
    \deg g' &\leq m-2.
  \end{align*}
  Thus $p\mid m$, and $g$ has nonzero coefficients only at $ x^{i}$
  with $i \leq m$ and $p \mid i$ or $1\leq i \leq m-c$. The number of
  such $i$ is $m-c+\lceil c/p\rceil$. By composing with a linear
  polynomial on the left and by $h$ on the right and using that $g$ is
  monic, we find
  $$
  t \leq q^{2}(1-q^{-1}) \cdot q^{m-c+ \lceil c/p\rceil-1}
  \cdot q^{l-1} = q^{m+l-c+ \lceil c/p \rceil}(1-q^{-1}).
  $$
  If $l\mid m$, then $c=m/l-1+\lceil 1/l\rceil =m/l$.
\end{proof}
For perspective, we also note the following lower bounds on
$\#T$. Unlike the results up to \ref{thm:FFC}, there is a substantial
gap between the upper and lower bounds.

\begin{corollary}
  \label{cor:ffcharb}
  Let $\mathbb{F}_{q}$ be a finite field of characteristic $p$, $l$ a
  prime number dividing $m >l$, assume that $p \mid \dg = lm$, and set
  $t =\#(D_{\dg, l} \cap D_{\dg,m}\cap D_{n}^{+})$.  Then the
  following hold.
  \begin{enumerate}
  \item
    \label{cor:ffcharb-1}
    If $p \neq l$ divides $m$ exactly $d \geq 1$ times, then
    $$
    q^{2l+m/l-1}(1-q^{-1})(1-q^{-m/l})(1-q^{-1}(1+q^{-p+2}
    \frac{(1-q^{-1})^{2}}{1-q^{-p}})) \leq t
    $$
    if $l \nmid p^{d}-1$. Otherwise we set $\mu=~gcd(p^{d}-1,l)$,
    $r^{*}=(p^{d}-1)/\mu$ and have
    \begin{align*}
      & q^{2l+m/l-1}(1-q^{-1})\bigl((1-q^{-1}(1+q^{-p+2}
      \frac{(1-q^{-1})^{2}}{1-q^{-p}}))(1-q^{-m/l})\\
      & -q^{-m/l-r^{*}+2}
      \frac{(1-q^{-1})^{2}(1-q^{-r^{*}(\mu-1)})}{1-q^{-r^{*}}}
      (1+q^{-r^{*}(p-2)})\bigr) \leq t.
    \end{align*}
\item\label{cor:ffcharb-4} If $p=l$, $p\nmid m/p$, and $m$ has no
    prime divisor smaller than $p$, then
    $$
    q^{2p+m/p-1}(1-q^{-1})^{2}(1-q^{-p+1}) \leq t.
    $$
  \end{enumerate}
\end{corollary}

\begin{proof}
  \short\ref{cor:ffcharb-1} For any monic original $g, w, h \in
  \mathbb{F}_{q}[x]$ of degrees $l, m/l, l$, respectively, we have $g
  \circ w \circ h \in D_{\dg,l} \cap D_{\dg,m} \cap D^{0}_{\dg}$. We now
  estimate the number of such compositions.

 Since $p \nmid l =
  \deg g$, \ref{cor:inj-1} implies that the composition map $(g, w
  \circ h) \mapsto g \circ w \circ h$ is injective. To estimate from
  below the number $N$ of $w \circ h$, we use \ref{th:decom} with
  $r=p^{d}$, $a=m/lp^{d}$, $k=m/l$, $ \tilde{m}=l \neq r$, $\mu = \gcd
  (r-1,l)$, and $r^{*}= (r-1)/\mu$.  (Here $\tilde{m}$ is the value
  called $m$ in \ref{th:decom}, whose name conflicts with the present
  value of $m$.) 

  If $\mu = 1$, we obtain from \ref{th:decom-1}
  $$ 
  N \geq q^{l+m/l-2}(1-q^{-m/l})(1-q^{-1}(1+q^{-p+2}
  \frac{(1-q^{-1})^{2}}{1-q^{-p}})).
  $$

  If $\mu\neq 1$, \ref{th:decom-2} says that
  \begin{align*}
    N \geq \, & q^{l+m/l-2}\bigl((1-q^{-1}(1+q^{-p+2}
    \frac{(1-q^{-1})^{2}}{1-q^{-p}}))(1-q^{-m/l})\\
    & -q^{-m/l-r^{*}+2}
    \frac{(1-q^{-1})^{2}(1-q^{-r^{*}(\mu-1)})}{1-q^{-r^{*}}}
    (1+q^{-r^{*}(p-2)})\bigr),
  \end{align*}
  where we have used the simplification of \ref{frac:leqp}.  (We note
  that \ref{cor:decom} provides a simplified bound if $r^{*} \geq 2$ and
  $p>\mu$; when $p >l$, then these two inequalities hold unless $l=2$
  and $r=3$.)

  We compose these $w \circ h$ with $v \circ g$ on the left, where $v$
  is linear and $g$ monic original of degree $l$. This gives the lower
  bound
  $$
  q^{2}(1-q^{-1})\cdot q^{l-1} \cdot N = q^{l+1}(1-q^{-1})N
  $$
  on $t$, as claimed.

  Thus $g$ has nonzero coefficients only at $x^{i}$ with $p\mid i$ or
  $i\leq ap-a$. It follows that
$$
t\leq q^{a-1+ap-(a-\left\lfloor a/p\right\rfloor)}\cdot
q^{p-1}=q^{ap+p-a+\left\lfloor a/p\right\rfloor-2}.
$$

  \short\ref{cor:ffcharb-4} Clearly, $t$ is at least the number of $v
  \circ g \circ w \circ h$ with $v$ linear and $g, w, h \in F[x]$
  monic original of degrees $p$, $m/p$, $p$, respectively.

  We first bound the number $t^{*}$ of $h^{*}= w \circ h$ with
  $h^{*}_{m-1} \neq 0$. We denote as $h_{p-1}$ the second highest
  coefficient of $h$. Then $h^{*}_{m-1}= m/p \cdot h_{p-1}$, and
  $h^{*}_{m-1}$ vanishes if and only if $h_{p-1}$ does. By
  \ref{cor:inj-1}, $\gamma_{m, m/p}$ is injective, so that
  $$
  t^{*}= q^{m/p-1} \cdot q^{p-1}(1-q^{-1})= q^{m/p+p-2}(1-q^{-1}).
  $$
  We now consider $g\circ h^{*}$ as input to \ref{algoWd}.

  We have $r=p\neq m$ and $\mu= \gcd(p-1,m)=1$. In the proofs of
  \ref{th:decom-1} and \ref{cor:decom-1}, no special properties of $h$
  are used, except \ref{eq:pc}. In the notation used there, we have
  $i_{0}\in\mathbb{N}$ if and only if $p-1\mid(\kappa-1)m$. Now
  $\kappa<p$ and $m$ has no divisors less than $p$, so that
  $i_{0}\notin\mathbb{N}$ and \ref{eq:pc} holds vacuously for all
  $h$. Thus the lower bound also applies when we replace the number
  $q^{m-1}(1-q^{-1})$ of all possible second components by
  $t^{*}$. Thus
  $$
  t\geq q^{p+m} (1-q^{-1}) (1-q^{-p})
  (1-q^{-1}(1+q^{-p+2}\frac{(1-q^{-1})^{2}}{1-q^{-p}}))\cdot\frac{q^{m/p+p-2}(1-q^{-1})}{q^{m-1}(1-q^{-1})}$$
  $$=q^{2p+m/p-1}(1-q^{-1})(1-q^{-p})(1-q^{-1}(1+q^{-p+2}\frac{(1-q^{-1})^{2}}{1-q^{-p}}))$$
  $$=q^{2p+m/p-1}(1-q^{-1})^{2}(1-q^{-p+1}).\qed$$
\end{proof}

\begin{example}\label{ex:p2}
  We study the particular example $p=l=2$ and $m=6$, so that $n=12$. Let
  $t_{1}= t \cdot q^{-2}(1-q^{-1})^{-1}$ denote the number of monic
  original polynomials in $D_{12,2} \cap D_{12,6} \cap
  D_{12}^{+}$. Then \ref{cor:ffchar-3} says that $t_{1} \leq
  q^{5}$. By coefficient comparison, we now find a better
  bound. Namely, we are looking for $g \circ h = g^{*} \circ h^{*}$
  with $g, h, g^{*}, h^{*} \in \mathbb{F}_{q}[x]$ monic original of
  degrees $2,6,6,2$, respectively. (We have reversed the usual degrees
  of $g$, $h$ and $g^{*}$, $h^{*}$ for notational convenience.)  We
  write $h = \sum_{i}h_{i}x^{i}$, and similarly for the other
  polynomials.  Then we choose any $h_{2}, h_{4}, h_{5} \in
  \mathbb{F}_{q}$, and either $g_{1}$ arbitrary and $h_{1}= uh_{5}$,
  or $h_{1}$ arbitrary and $g_{1} = h_{5}(h_{1}+uhs)$, where $u =
  h_{5}^{4} + h_{5}^{2} h_{4}+h_{2}$. Furthermore, we set $h_{3}=
  h_{5}^{3}$ and $h_{1}^{*}= h_{5}$. Then the coefficients of $g^{*}$
  are determined. If $g'(g^{*})'\neq 0$, then the above constitute a
  collision, and by comparing coefficients, one finds that these are
  all. Their number is at most $2q^{4}$, so that $t_{1} \leq 2q^{4}$
  and $t \leq 2q^{6}(1-q^{-1})$.

For an explicit description of $g$, we set $u_{2}= h_{4} + h_{5}^{2}$.
In the first case, where $h_{1}= uh_{5}$, we have
$$
g^{*}= x^{6}+u_{2}^{2}x^{4}+g_{1}x^{3}+(u^{2}+u_{2}g_{1})x^{2}+g_{1}ux.
$$
In the second case, we have
$$
g^{*}= x^{6}+u_{2}^{2}x^{4}+h_{5}(h_{1}+uh_{5})x^{3}+(u_{2}h_{1}h_{5}+
uh_{2})x^{2}+h_{1}(h_{1}+uh_{5})x.
$$
In both cases, $g_{1}= g' \neq 0$ implies that $(g^{*})' \neq 0$.
\end{example}

\cite{gie88}, Theorem 3.8, shows that there exist polynomials of
degree $n$ over a field of characteristic $p$ with super-polynomially
many decompositions, namely at least $\dg ^{\lambda \log n}$ many,
where $\lambda = (6 \log p)^{-1}$.


\section{Counting tame decomposable polynomials}\label{sec:cdup}
This section estimates the dimension and number of decomposable
univariate polynomials. We start with the dimension of decomposables
over an algebraically closed field. Over a finite field,
\ref{thm:Estimate} below provides a general upper bound on the number
in \short\ref{thm:Estimate-1}, and an almost matching lower bound.
The latter applies only to the tame case, where $p \nmid \dg$, and
both bounds carry a relative error term. Lower bounds in the more
difficult wild case are the subject of \ref{sec:cgdp}.

\cite{gie88} was the first work on our counting problem. He proves (in
his Section 1.G and translated to our notation) an upper bound of
$d(\dg) q^{2+\dg/2}$ $(1-q^{-1})$ on the number of decomposable
polynomials, where $ d(\dg)$ is the number of divisors of $\dg$. This
is mildly larger than our bound of about $2q^{l+\dg/l}(1-q^{-1})$, in
\ref{thm:Estimate-1}, with its dependence on $l$ replaced by the
``worst case'' $l=2$, as in the \whole\ref{cor:Fq} \ref{cor:Fq-6}.  With the
same replacement, Giesbrecht's thesis contains the upper bound in the
following result, which is the geometric bound for our current
problem.
\begin{theorem}\label{thm:geo}
  Let $F$ be an algebraically closed field, $\dg \geq 2$, and $l$ the
  smallest prime divisor of $\dg$.  Then $D_{\dg} = \varnothing$ if
  $\dg$ is prime, and otherwise
  $$
  \text{dim } D_{\dg} = l + \dg/l.
  $$
\end{theorem}
\begin{proof}
  We may assume that $\dg $ is composite. By \ref{cor:inj}, the fibers
  of $\gamma _{\dg,l}$ are finite, and hence
  $$
  \text{dim} D_{\dg} \geq \text{dim} D_{\dg,l}= \text{dim} (P^{=}_{l}
  \times P_{\dg/l}^{0})= l+ \dg/l.
  $$
  
  Now $D_{\dg,\dg/l} $ has the same dimension, and $D_{\dg,e}$ has
  smaller dimension for all other divisors $e$ of $\dg$. \qed

  The argument for \ref{lem:div-1} shows that if $\dg$ is composite,
  $p \nmid \dg$, and $l^{2} \nodiv \dg$, then $\text{dim}(D_{\dg,l} \cap
  D_{\dg,\dg /l} )\leq \lfloor \dg /l^{2} \rfloor +3 < l+ \dg/l$. Thus
  $\gamma_{\dg,l}$ and $\gamma_{\dg,\dg/l}$ describe two different
  irreducible components of $D_{\dg}$, both of dimension $l+\dg/l$.

  \cite{zan08} studies a different but related question, namely
  compositions $f= g \circ h$ in $ \mathbb{C} [x]$ with a
  \emph{sparse} polynomial $f$, having $t$ terms. The degree is not
  bounded.  He gives bounds, depending only on $t$, on the degree of
  $g$ and the number of terms in $h$.  Furthermore, he gives a
  parametrization of all such $f$, $g$, $h$ in terms of varieties (for
  the coefficients) and lattices (for the exponents).
\end{proof}

We now present a generally valid upper bound on the number of decomposables
and a lower bound in the tame case $p \nmid \dg$.
\begin{theorem}
  \label{thm:Estimate}
  Let $\mathbb{F}_{q}$ be a field of characteristic $p$ and with $q$
  elements, and $\dg \geq 2$. Let $l$ and $l_{2}$ be the smallest and
  second smallest nontrivial divisors of $\dg$, respectively (with
  $l_{2}=1$ if $\dg=l$ or $\dg=l^{2}$), $s =\lfloor \dg/l^{2}\rfloor$,
  and
  \begin{align}
    \label{thm:small}
    \alpha_{\dg}&=
     \begin{cases}
      0&\text{if } \dg=l,\\
      {q^{2l}(1-q^{-1})}& \text{if }\dg = l^{2},\\
      2q^{l+\dg/l}(1-q^{-1})&\text{ otherwise},
    \end{cases}
    \\
    \nonumber c&=\frac{(\dg-ll_{2})(l_{2}-l)}{ll_{2}},
    \\
     \beta_{\dg} &=
    \begin{cases}
      0&\textrm{if }\dg  \in \{l,l^{2},l^{3}, ll_{2}\}, \\
      \displaystyle \frac{q^{-c}}{1-q^{-1}} & \text{otherwise},
    \end{cases}
\nonumber \\
    \beta_{\dg}^{*}&= q^{-l-\dg/l+s+3},\\
   t &=
    \begin{cases}
      0 & \text{if } \dg\in\{l,l^{2}\},\\
      \#(D_{\dg,l} \cap D_{\dg,\dg/l})& \text{otherwise}.
    \end{cases}
  \end{align}
  Then the following hold.
  \begin{enumerate}
  \item\label{thm:Estimate-1} $\# D_{\dg}\leq
    \alpha_{\dg}(1+ \beta_{\dg})$. If
     $\dg\notin\{l^{2},l^{3}\}$, then $\# D_{\dg}\leq
    \alpha_{\dg}(1-\alpha_{\dg}^{-1}t+\beta_{\dg})$ .
  \item\label{thm:Estimate-2} $\#I_{\dg} \geq \# P^{=}_{\dg}
    -2\alpha_{\dg}$. 
  \item\label{thm:Estimate-3} If $p \nmid \dg$ and $l^{2} \nmid \dg$,
    then
    $$
\alpha_{\dg}(1-q^{-\dg/l+l+s-1})\leq\alpha_{\dg} (1- \beta^{*}_{\dg}) \leq \#D_{\dg} \leq\alpha_{\dg}(1-\frac{\beta_{\dg}^{*}}{2}+ \beta_{\dg}).
    $$
  \item\label{thm:Estimate-4} If $p \nmid \dg$, then
    $$
\alpha_{\dg} (1- q^{-\dg/l+l+s-1}) \leq\#D_{\dg} \leq
\alpha_{\dg}(1-\frac{\beta_{\dg}^{*}}{2}+ \beta_{\dg}). 
   $$
  \item\label{thm:Estimate-2/2} If $p \neq l$, then $\#D_{l^{2}}=
    \alpha_{l^{2}}$ and $\#D_{l^{3}}=
    \alpha_{l^{3}}(1-q^{-(l-1)^{2}}/2)$.
  \item\label{thm:Estimate-2/3} If $p\nmid \dg\neq l^{2}$ and $\dg/l$
    is prime, then
    $$
    \# D_{\dg}=\alpha_{\dg}\bigl(1-\frac 1 2
    q^{-\dg/l-l+3}(q^{s}+(1-\delta_{l,2})(q-1))\bigr).
    $$
  \end{enumerate}
\end{theorem}
\begin{proof}
  When $\dg=l$ is prime, then $D_{\dg} = \varnothing$ and all claims
  are clear (reading $\alpha_{\dg}^{-1}t$ as $0$). We may now assume
  that $\dg$ is composite.

  \short\ref{thm:Estimate-1} The claim for $\dg\in\{l^{2},l^{3}\}$
  follows from \short\ref{thm:Estimate-2/2}, and we now exclude these
  cases. We write $u(e)=e+\dg/e$ for the exponent in
  \ref{cor:inj-1}. We have the two largest subsets $D_{\dg,l}$ and
  $D_{\dg,\dg/l}$ of $D_{\dg}$, both of size at most
  \begin{equation}\label{eq:DD}
    \frac{\alpha_{\dg}}{2}= q^{u(l)}(1-q^{-1})= q^{l+\dg/l}(1-q^{-1})=
    \#(P^{=}_{l} \times P^{0}_{\dg/l})=
    \#(P^{=}_{\dg/l} \times P^{0}_{l}).
  \end{equation}

  Their joint contribution to $\#D_{\dg}$ is at most
  \begin{equation}\label{eq:joint}
    \alpha_{\dg}-t.
  \end{equation}  

  Since $\dg$ is not $l$ or $l^{2}$, we have $l < l_{2} \leq \dg/l$,
  and $l_{2}$ is either $l^{2}$ or a prime number larger than $l$.
  The index set $E$ in \ref{substack} consists of all proper divisors
  of $\dg$.  If $\dg=ll_{2}$, then $E=\{l,l_{2}\}$, and from
  \ref{eq:joint} we have
  $$
  \#D_{\dg} \leq \alpha_{\dg} -t.
  $$

  We may now assume that $\dg \neq ll_{2}$.
  For any $e \in E$, we have $u(e)= e + \dg/e =
  u(\dg/e)$. Furthermore
  \begin{equation}\label{e'}
    u(e)- u(e')= \frac{(\dg-ee')(e'-e)}{ee'}
  \end{equation}
  holds for $e,e' \in E$, and in particular
  \begin{equation}\label{E}
    u(l)-u(l_{2})
    = (\dg-ll_{2})(l_{2}-l)/ll_{2}=c.
  \end{equation}
  Considered as a function of a real variable $e$, $u$ is convex on
  the interval $[1..\dg]$, since $\partial^{2}u/\partial e^{2} =
  2\dg/e^{3}>0$. Thus $u(l)-u(e) \geq c$ for all $e \in E_{2} = E
  \smallsetminus \{l,\dg/l\}$.  Then
  \begin{align*}
    \sum_{e \in E_{2}}q^{u(e)-u(l)} &=q^{-c} \sum_{e \in
      E_{2}}q^{u(e)-u(l) +c}\\
     &< q^{-c}\cdot 2\sum_{i \geq 0}q^{-i}=\frac{2q^{-c}}{1-q^{-1}},
  \end{align*}
  since each value $u(e)$ is assumed at most twice, namely for $e$ and
  $\dg/e$, according to \ref{e'}.  Using \ref{eq:joint}, it follows
  for $\dg \neq l^{2}$ that
  \begin{equation}\label{eq:e}
    \begin{aligned}
      \#D_{\dg}+t &\leq \sum_{e \in E} \#D_{\dg,e}\leq \sum_{e \in
        E}q^{u(e)}(1-q^{-1}) \\
      &\leq q^{l+\dg/l}(1-q^{-1})(2+ \sum_{e \in
        E_{2}}q^{u(e)-u(l)})\\
      &\leq q^{l+\dg/l}(1-q^{-1})(2+ \frac{2q^{-c}}{1-q^{-1}}) =
      \alpha_{\dg}(1+ \beta_{\dg}).
    \end{aligned}
  \end{equation}
  
  This implies the claim in \short\ref{thm:Estimate-1}.
  
  \short\ref{thm:Estimate-2} follows from $\beta_{\dg} \leq 1$.
  
  For \short\ref{thm:Estimate-3}, we have $D_{\dg,l} \cup
  D_{\dg,\dg/l} \subseteq D_{\dg}$. Since $p \nmid \dg$, both
  $\gamma_{\dg,l}$ and $\gamma_{\dg, \dg/l}$ are injective, by
  \ref{cor:inj-1}. From \ref{lem:div-1}, we find
  \begin{align*}
    \#D_{\dg}  &\geq \#D_{\dg,l} + \#D_{\dg,\dg/l}- \#(D_{\dg,l} \cap
    D_{\dg,\dg/l})\\
     &\geq 2q^{l+\dg/l}(1-q^{-1})- (q^{s+3}+q^{4})(1-q^{-1})\\
     &= \alpha_{\dg}(1-\frac{q^{s+3}+q^{4}}{2q^{l+\dg/l}}) \geq \alpha_{\dg}(1-\frac{q^{s+3}}{q^{l+\dg/l}})=
    \alpha_{\dg}(1-\beta^{*}_{\dg}),\\
 \#D_{\dg}&\leq\alpha_{\dg}(1-\frac{q^{s+3}(1-q^{-1})}{\alpha_{\dg}}+\beta_{\dg})=\alpha_{\dg}(1-\frac{\beta_{\dg}^{*}}{2}+\beta_{\dg}).
  \end{align*}

Furthermore, we have $1\leq s \leq n/l^2$ (since $n$ is composite),
  $s+3\geq 4$, $l \geq 2$, and hence
  $$
  -l-\frac n l + s + 3 \leq -\frac n l + l +s
  -1.
  $$
It follows that
  $$
  \beta^* \leq q^{-n/l+l+s-1}.
  $$

  \short\ref{thm:Estimate-4} For the lower bound if $l^{2}\mid \dg$,
  we replace the upper bound from \ref{lem:div-1} by the one
  from \whole\ref{thm:FFC-2}.

  In \short\ref{thm:Estimate-2/2}, for $\dg=l^{2}$, we have $D_{\dg}=
  D_{\dg,l}$ and
  $$
  \#D_{\dg} = q^{l+\dg/l}(1-q^{-1})= \alpha_{\dg},
  $$ using  the injectivity
  of $\gamma_{l^{2},l}$ (\ref{cor:inj-1}). When $\dg=l^{3}$, then
  \ref{thm:FFC} says that
 \begin{align*} 
  t&=q^{3l-1}(1-q^{-1}),\\
\# D_{l^{3}}&=\alpha_{l^{3}}(1-\frac{t}{\alpha_{l^{3}}})=\alpha_{\dg}(1-\frac{q^{-(l-1)^{2}}}{2}).
\end{align*}
  This shows \short\ref{thm:Estimate-2/2}. For \short\ref{thm:Estimate-2/3},
  we replace the bound on $\#(D_{\dg,l}\cap D_{\dg,\dg/l})$ by its
  exact value from \ref{lem:div-1}.
\end{proof}

\cite{boddeb09} state an upper bound as in \ref{thm:Estimate-1}, with
an error term which is only $O(\dg)$ worse than $\beta_{\dg}$.
\begin{remark}\label{remark:twice}
  How often does it happen that the smallest prime factor $l$ of $\dg$
  actually divides $n$ at least twice? The answer: almost a third of
  the time.

  For a prime $l$, let
  $$
  S_{l}= \{\dg \in \mathbb{N}\colon l^{2}\mid \dg, \forall \text{
    primes } r<l \quad r \nmid \dg\},
  $$
  so that $\bigcup_{l} S_{l}$ is the set in question. The union is
  disjoint, and its density is
  $$
  \sigma = \sum_{l} \frac{1}{l^{2}} \prod_{r<l}(1-\frac{1}{r}) \approx
  0.330098.
  $$
  If we take a prime $p$ and further ask that $p \nmid \dg$, then we
  have the density
  $$
  \sigma_{p}= \sigma - \frac{1}{p^{2}}
  \prod_{r<p}(1-\frac{1}{r})-\frac{1}{p} \sum_{l<p} \frac{1}{l^{2}}
  \prod_{r<l}(1-\frac{1}{r}).
  $$
  The correction terms $\sigma - \sigma_{p}$ are $\approx 0.25,
  0.13889, 0.07444$ for $p=2,3,5$, respectively.
\end{remark}

The upper and lower bounds in \ref{thm:Estimate-1} and
\short\ref{thm:Estimate-3} have distinct relative error estimates. We
now compare the two.
\begin{proposition}\label{prop:l}
  In the notation of \ref{thm:Estimate}, assume that $\dg \neq l,
  l^{2}, ll_{2}$.
  \begin{enumerate}
  \item If $l_{2} \leq l^{2}$, then $\beta_{\dg} >
    \beta_{\dg}^{*}$. If furthermore $l^{2} \nmid \dg$ and $p \nmid
    \dg$, then
    $$
    | \#D_{\dg}- \alpha_{\dg} | \leq \alpha_{\dg}\beta_{\dg}.
    $$
  \item If $l_{2} \geq l^{2}+l$, then $\beta_{\dg} \leq
    \beta_{\dg}^{*}$. If furthermore $l^{2} \nmid \dg$ and $p \nmid
    \dg$, then
    $$
    | \#D_{\dg}- \alpha_{\dg}| \leq \alpha_{\dg}\beta_{\dg}^{*}.
    $$
  \end{enumerate}
\end{proposition}
\begin{proof}
  We let $\mu=-\log_{q}(1-q^{-1})$ and $\sigma = \dg/l^{2}-s$, so that
  $0 < \mu\leq 1$, $0\leq\sigma\leq 1-1/l<1$, and
  \begin{align*}
\beta_{\dg}&=q^{-c+\mu},\\ 
\beta_{\dg}^{*}&=q^{-l-\dg/l+\dg/l^{2}-\sigma+3}.
\end{align*}
 Furthermore,
  \begin{equation}\label{eq:lambda}
    \begin{aligned}
      \beta_{\dg}\leq \beta_{\dg}^{*}  &\Longleftrightarrow ll_{2}(l+
      \frac{\dg}{l} - \frac{\dg}{l^{2}}+\sigma+ \mu-3) \leq (\dg
      -ll_{2})(l_{2}-l)\\
      &\Longleftrightarrow ll_{2} (l_{2}+\sigma+\mu-3) \leq \frac{\dg}{l} (l_{2}-l^{2}).
    \end{aligned}
  \end{equation}
  We note that $l_{2} >l_{2}+ \sigma +\mu-3>0$. If $l_{2}\leq l^{2}$,
  it follows that $\beta_{\dg} > \beta_{\dg}^{*}$. If $l_{2}\geq
  l^{2}+l$, then $a=\dg/ll_{2}$ is a proper divisor of $\dg$, since
  $\dg \neq ll_{2}$. It follows that $a \geq l_{2}$, since $a=l$ would mean
  that $l^{2}$ is a divisor of $\dg$ with $l < l^{2} < l_{2}$,
  contradicting the minimality of $l_{2}$. Then
  $$
  \frac{\dg}{l}(l_{2}-l^{2}) \geq l_{2}^{2}\cdot l >ll_{2}(l_{2}+\sigma+\mu-3),
  $$
  and $\beta_{\dg} \leq \beta_{\dg}^{*}$.

  The claims about $\#D_{\dg}$ follow from \ref{thm:Estimate}.
\end{proof}

There remains the ``gray area'' of $l^{2} < l_{2} <l^{2}+l$, where
\ref{eq:lambda} has to be evaluated.  The three equivalent properties
in \ref{eq:lambda} hold when $\dg$ has at least four prime factors,
and do not hold when $\dg = ll_{2}$.

We can simplify the bounds of \ref{thm:Estimate}, at the price of a
slightly larger relative error. 
\begin{corollary}\label{cor:l}
  We assume the notation of \ref{thm:Estimate}.
  \begin{enumerate}
  \item\label{cor:l-1} If $\dg$ is prime, then $D_{\dg}= \varnothing$.
  \item\label{cor:l-2} For all $\dg$, we have
    \begin{equation}\label{eq:betan}
      \#D_{\dg} \leq \alpha_n (1+ q^{-\dg/3l^{2}}).
    \end{equation}
  \item\label{cor:l-3} If $p \nmid \dg$, then
    $$
    \left|\#D_{\dg}- \alpha_{\dg}\right| \leq \alpha_{\dg} \cdot q
    ^{-n/3l^{2}}.
    $$
  \end{enumerate}
\end{corollary}
\begin{proof}
  \short\ref{cor:l-1} follows from \ref{thm:Estimate-1}, since
  $\alpha_{\dg}=0$. For \short\ref{cor:l-2}, we claim that
  $\beta_{\dg}\leq q^{-\dg/3l^{2}}$. The cases where $\dg \in
  \{l,l^{2},ll_{2}\}$ are trivial, and we may now assume that $a =
  \dg/ll_{2} \geq 2$. We set $\mu = -\log_{q} (1-q^{-1})$, so that
  $0< \mu \leq 1$ and $\beta_{\dg} = q^{-c + \mu}$. 

We have
$$
\frac{3l^{3}+3l}{3l-2}\geq\frac{3l^{2}}{3l-1}.
$$  
If
  \begin{equation}\label{eq:3l}
    l_{2}\geq \frac{3l^{2}+3l}{3l-2}= l+\frac{5}{3}+\frac{10}{9l-6},
  \end{equation}
  then $l_{2}-l-l_{2}/3l \geq 0$ and
  \begin{align}\label{a}
    a(l_{2}-l- \frac{l_{2}}{3l}) \geq 2(l_{2}-l-\frac{l_{2}}{3l}) \geq
    l_{2}-l+1,\nonumber\\
    (a-1)(l_{2}-l)-1 \geq \frac{al_{2}}{3l}= \frac{\dg}{3l^{2}},
  \end{align}
  from which the claim follows. \ref{eq:3l} is satisfied except when
  $(l,l_{2})$ is $(2,3)$, $(2,4)$ or $(3,5)$.

  In the first case, \ref{a} is satisfied for $a \geq 4$, and in the
  other two for $a \geq 3$. The latter always holds in the case $(3,5)$,
  and we are left with $\dg \in \{12,16,18\}$.  For these values of
  $\dg$, we use a direct bound on the sum in \ref{eq:e}, namely
  $$
  \sum_{e \in E_{2}} q^{u(e)-u(l)} \leq \#E_{2}\cdot q^{-c}= 2
  \epsilon q^{-c},
  $$
  where $\epsilon = \#E_{2}/2$, so that
  $$
  \#D_{\dg} \leq \alpha_{\dg}(1+ \epsilon q^{-c})-t.
  $$
  The required values are given in \ref{tab:para}. In all cases, we
  conclude from \ref{thm:Estimate-1} that $\#D_{\dg} \leq \alpha_{\dg}(1+q^{-\dg/3l^{2}})$.

  \begin{table}[h!]
    $$
    \begin{array}{l|c|c|c}
      \dg  & 12 &16 &18 \\\hline
      \epsilon & 1 & 1/2 & 1\\
      c & 1 & 2 & 2\\
      \dg/3l^{2}& 1 & 4/3 & 3/2
    \end{array}
    $$
    \caption{Parameters for three values of $\dg$.}
    \label{tab:para}
  \end{table}

  \short\ref{cor:l-3} Our claim is that $q^{-\dg/l+l+s-1} \leq
  q^{-\dg/3l^{2}}$. Since $\dg\geq l^{2}$, we have
\begin{align*}
l^{2}(3l-3) &\leq l^{2}(3l-2)\leq\dg(3l-2),\\
2\dg+3l^{3}&\leq 3l\dg+3l^{2},\\
\frac{\dg}{3l^{2}}+l+s  &\leq \frac{\dg}{3l^{2}}+l+\frac{\dg}{l^{2}}=\frac{2\dg}{3l^{2}}+l\leq\frac{\dg}{l}+1.
\end{align*}
This proves the claim, and \short\ref{cor:l-3} follows from
\short\ref{cor:l-2} and \ref{thm:Estimate}.
\end{proof}

\section{Counting general decomposable polynomials}\label{sec:cgdp}

\ref{thm:Estimate} provides a satisfactory result in the tame case,
where $p \nmid \dg$. Most of the preparatory work in Sections
\bare\ref{sec:usd} and \bare\ref{sec:collcomp} is geared towards the
wild case. The upper bound of \ref{thm:Estimate-1} still holds. We now
present the resulting lower bounds.

\begin{figure}[h!]
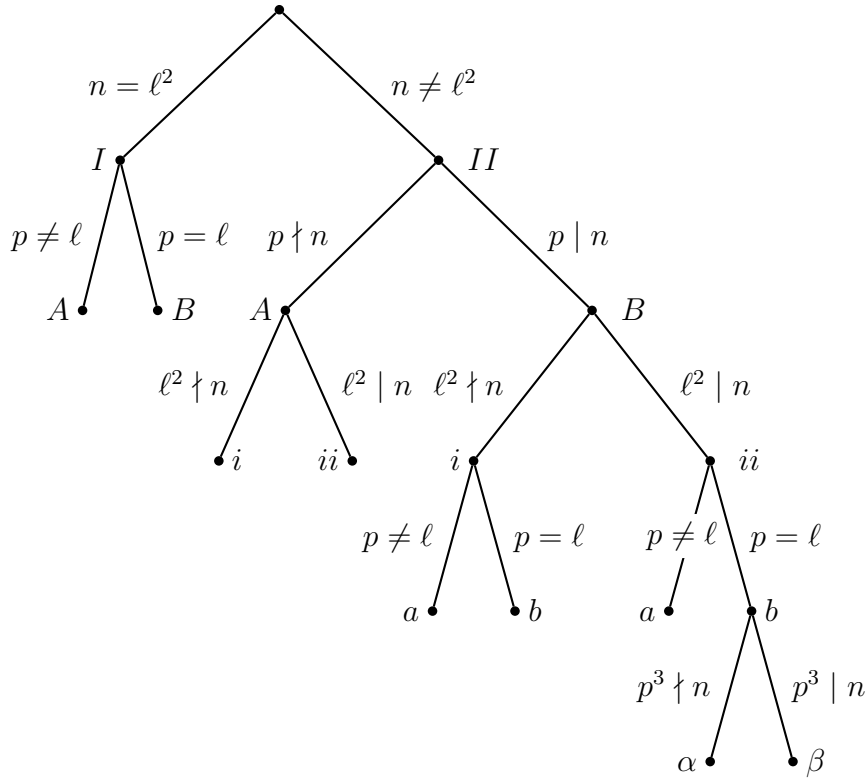

  \[
  \begin{psTree}[treesep=0.7cm,levelsep=2cm]{\Tdot}
    \begin{psTree}[treesep=1cm]{\Tdot~[tnpos=l]{I}^{n=l^{2}}}
      \Tdot~[tnpos=l]{A}^{p\neq l} \Tdot~[tnpos=r]{B}_{p=l}
    \end{psTree}
    \begin{psTree}[treesep=1.3cm]{\Tdot~[tnpos=r]{\enspace II}_{n\neq
          l^{2}}}
      \begin{psTree}[treesep=1cm]{\Tdot~[tnpos=l]{A}^{p\nmid n}}
        \Tdot~[tnpos=r]{i}^{l^{2}\nmid n}
        \Tdot~[tnpos=l]{ii}_{l^{2}\mid n}
      \end{psTree}
      \begin{psTree}[treesep=1.3cm]{\Tdot~[tnpos=r]{\enspace B}_{p\mid
            n}}
        \begin{psTree}[treesep=1.1cm]{\Tdot~[tnpos=l]{i}^{l^{2}\nmid
              n}}
          \Tdot~[tnpos=l]{a}^{p \neq l} \Tdot~[tnpos=r]{b}_{p = l}
        \end{psTree}
        \begin{psTree}[treesep=1.1cm]{\Tdot~[tnpos=r]{\enspace
              ii}_{l^{2}\mid n}}
          \Tdot~[tnpos=l]{a}\ncput*{\rput*(-0.1,0){p \neq l}}
          \begin{psTree}[treesep=1.1cm]{\Tdot~[tnpos=r]{b}_{p = l}}
            \Tdot~[tnpos=l]{\alpha}^{p^{3} \nmid n}
            \Tdot~[tnpos=r]{\beta}_{p^{3} \mid n}
          \end{psTree}
        \end{psTree}
      \end{psTree}
    \end{psTree}
  \end{psTree}
  \]
  \caption{The tree of case distinctions for estimating $\#D_{n}$.}
  \label{fig:tree}
\end{figure}

We have to deal with an annoyingly large jungle of case
distinctions. To keep an overview, we reduce it to the single tree of
\ref{fig:tree}. Its branches correspond to the various bounds on
equal-degree collisions (\ref{cor:decom}) and on distinct-degree
collisions (Corollaries \bare\ref{lem:div-b}, \bare\ref{thm:FFC}, and
\bare\ref{cor:ffchar}).  Since at each internal vertex, the two
branches are complementary, the leaves cover all possibilities. We
use a top down numbering of the vertices according to the branches; as
an example, II.B.ii.b.$\beta$ is the rightmost leaf at the lowest
level.  Furthermore, if a branching is left out, as in II.B, then a
bound at that vertex holds for all descendants, which comprise three
internal vertices and five leaves in this example.

\begin{table}[h!]
  $$
  \begin{array}{l|l|l}
    \text{leaf in} & & \text{up-} \\
    \text{\ref{fig:tree}}
    & \text{lower bound on $\#D_{\dg}/\alpha_{\dg}$}
    & \text{per}\\ \hline
    \text{I.A} & 1 & 1\\
    \text{I.B} & 
    \frac 1 2 (1+\frac{1}{p+1})(1-q^{-2})+{q^{-p}} > 1/2 & 1\\
    \text{II.A.i} & 1- \beta_{\dg}^{*} \geq 1-q^{-n/l-l+n/l^2+3} &\\
    \text{II.A.ii} & 1-q^{-\dg/l+l+\dg/l^{2}-1}/2 &\\
    
    \text{II.B.i.a} & 1- (q^{-1}+q^{-p+1}+q^{-\dg/l-l+\dg/l^{2}+3})/2 &\\
    \text{II.B.i.b} & 1- (q^{-1}-q^{-p})/2 & 1\\
    \text{II.B.ii.a} & 1- (q^{-1}+q^{-p+1}-q^{-p}+q^{-l+1})/2 & \\
    \text{II.B.ii.b.$\alpha $} & 
    \frac 1 2 (\frac{3}{2}
    + \frac{1}{2p+2}
    -{q^{-1}} 
    -\frac{q^{-2}} 2 (1+\frac 1 {p+1}) 
    -\frac{q^{-p+1}}{1-q^{-p}})
    & \\
    \text{II.B.ii.b.$\beta$} & 1-q^{-1}-q^{-p+1} & 1\\
  \end{array}
  $$
  \caption{The bounds at the leaves of \ref{fig:tree}.}
  \label{tab:leaves}
\end{table}

\begin{theorem}\label{thm:multi}
  Let $\mathbb{F}_{q}$ be a finite field of characteristic $p$ with
  $q$ elements, and $l$ the smallest prime divisor of the composite
  integer $\dg \geq 2$. Then we have the following bounds on
  $\#D_{\dg}$ over $\mathbb{F}_{q}$.
  \begin{enumerate}
  \item\label{thm:multi-1} If the ``upper'' column in \ref{tab:leaves}
    contains a $1$, then
    $$
    \# D_{\dg} \leq \alpha_{\dg}.
    $$
  \item\label{thm:multi-2} The lower bounds in \ref{tab:leaves} hold.
  \end{enumerate}
\end{theorem}
\begin{proof}
  We recall $D_{\dg,e}$ from \ref{eq:Zan} and $\alpha_{\dg}$ from
  \ref{thm:small}, the superscript $+$ for non-Frobenius from
  \ref{al:DD}, and set at each vertex
  $$
  \nu = \frac{\#D_{\dg}}{\alpha_{\dg}}, \; \nu_{0}=
  \frac{\#D^{+}_{\dg,l}}{\alpha_{\dg}}, \; \nu_{1}=
  \frac{\#D^{+}_{\dg,\dg/l}}{\alpha_{\dg}}, \; \nu_{2}=
  \frac{\#(D^{+}_{\dg,l} \cap D^{+}_{\dg,\dg/l})}{\alpha_{\dg}}, \;
  \nu_{3}= \frac{\#D^{\varphi}_{n}}{\alpha_{n}} .$$ Then $\nu =
  \nu_{0}+\nu_{3}$ if $\dg = l^{2}$, and otherwise
  \begin{equation}\label{eq:D}
    \nu_{0}+ \nu_{1}- \nu_{2} + \nu_{3} \leq \nu \leq 1 + \beta_{\dg} - \nu_{2}-\nu_{3}.
  \end{equation}
  In the lower bound, $\nu_{0}+\nu_{1}-\nu_{2}$ counts the
  non-Frobenius compositions of the dominant contributions $D_{\dg,l}$
  and $D_{\dg,\dg/l}$, and $\nu_{3}$ adds the Frobenius compositions.
  In the upper bound, $1-\nu_{2}$ bounds the two dominant
  contributions from above, $\beta_{\dg}$ accounts for the
  non-dominant contributions. We may subtract $\nu_{3}$ since the
  Frobenius compositions have been counted twice, in $D_{\dg,p}$ and
  $D_{\dg,\dg/p}$; of course, $\nu_{3}$ is nonzero only if $p\mid\dg.$

  The proof proceeds in two stages. In the first one, we indicate for
  some vertices $V$ bounds $\lambda_{i}(V)$ with the following
  properties:
  $$
  \nu_{0} \geq \lambda_{0}, \; \nu_{1} \geq \lambda_{1}, \;
  \lambda_{2} \geq \nu_{2} \geq \lambda_{4}.
  $$
  Such a bound at $V$ applies to all descendants of $V$.  The value
  $\lambda_{4}$ only intervenes in the upper bound on $\nu$, and we
  sometimes forego its detailed calculation and simply use
  $\lambda_{4}=0$. In the second stage, we assemble those bounds for
  each leaf, according to
  \ref{eq:D}.

  Throughout the proof, $d\geq 0$ denotes the multiplicity of $p$ in
  $\dg$, and $s = \lfloor \dg/l^{2}\rfloor.$
  In the first stage, we use \ref{thm:Estimate-2/2} at I.A:
  \begin{align*}
    \nu (\text{I.A}) = 1.
  \end{align*}
  At I.B, we have from \ref{ex:com}
  \begin{align*}
    \lambda_0(\text{I.B})  &\geq \frac 1 2
    (1+\frac{1}{p+1})(1-q^{-2})+{q^{-p}} .
  \end{align*}
  Furthermore,
  $$
  (1+\frac 1 {p+1} )(1-q^{-2}) \geq (1+\frac 1 {p+1} )(1-p^{-2}) =
  1+\frac {p-2} {p^2} \geq 1,
  $$
  so that $\lambda_0(\text{I.B}) > 1/2$. \ref{lem:propdivi-2} says
  that 
  $$
  \lambda_{3}(\text{I.B})=q^{-p+1}.
  $$  
  From \ref{cor:inj-1}, we
  have
  $$
  \lambda_{0}(\text{II.A})= \lambda_{1}(\text{II.A})= \frac{1}{2},
  $$
  and since $p \nmid \dg$,
  $$
  \nu_{3} (\text{II.A}) = 0.
  $$

  Vertex II.A.i has been dealt with in \ref{lem:div-1}:
  \begin{align*}
    \lambda_{2}(\text{II.A.i}) &=\beta_{\dg}^{*}\geq \frac{1}{2}q^{-\dg/l-l}(q^{s+3}+q^{4}),\\
    \lambda_{4}(\text{II.A.i}) &= \frac 1 2 q^{-n/l-l+s+3}.
  \end{align*}
  
  Since $l \mid \dg/l$, \ref{thm:FFC} yields
  \begin{align*}
    \lambda_{2}(\text{II.A.ii})= \lambda_{4}(\text{II.A.ii}) =
    \frac{1}{2} q^{-\dg/l+l+s-1}.
  \end{align*}
   
  Since $p \mid \dg$ at II.B, \ref{lem:propdivi-2} implies that
  $$
  \nu_{3}\text{(II.B)} = \frac{1}{2}q^{-l-\dg/l+\dg/p+1}.
  $$

  We now let $V$ be one of II.B.i.a or II.B.ii.a. Then we have
  $$
  \lambda_{0}(V)= \frac{1}{2},
  $$
  by \ref{cor:inj-1}.  Applying \ref{cor:decom} to $D_{\dg,\dg/l}$ at
  $V$, we have $d\geq 1$,$ r=p^{d} \neq l=m$, $k=\dg/l$, and
  \begin{equation}\label{eq:mu}
    \mu = \gcd(p^{d}-1,l)\text{ is either }1 \text{ or }l.
  \end{equation}
  In the first case, where $\mu=1$, we have
  $$
  \nu_{1}(V) \geq \frac{1}{2} (1-q^{-1}(1+q^{-p+2}
  \frac{(1-q^{-1})^{2}}{1-q^{-p}})) (1-q^{-\dg/l})
  $$
  from \ref{cor:decom-1}. In the second case, where $\mu=l$, we have
  $p > l = \mu \geq 2$. We first assume that $r\neq 3.$ Then
  $r-1=p^{d}-1$ is not a prime number, and $r^{*}= (r-1)/l \geq 2$, so
  that the last bound in \ref{cor:decom-2} applies and
  $$
  \nu_{1}(V)\geq \frac{1}{2} \bigl((1-q^{-1}(1+q^{-p+2}
  \frac{(1-q^{-1})^{2}}{1-q^{-p}}) \bigr) (1-q^{-\dg/l}) -
  \frac{2}{3}q^{-\dg/l}(1-q^{-1})^{2}.
  $$
  If $r=3$, then $p=3$, $\mu=l=2$, $r^*=1$, and according to the
  second bound in \ref{cor:decom-2}, we have to replace the last
  summand above by
  $$
  -\frac 1 2 q^{-n/l+1}(1-q^{-1})^2 (1+q^{-1}).
  $$
  Since $2/3\leq q(1+q^{-1})/2$, the latter term dominates in absolute
  value the one for $r\neq 3$. Its value is at least $q^{-\dg/l+1}/2$, and we find for $\mu=l$ that
  \begin{align*}
    \nu_{1}(V)  &\geq
    \frac{1}{2} - \frac {q^{-1}} 2 (1+q^{-p+2}(1-q^{-1})) \\
    & \quad-\frac {q^{-n/l}} 2 (1-q^{-1}-q^{-p+1} \frac {(1-q^{-1})^2}{1-q^{-p}} +q) \\
    &\geq \frac{1}{2} -\frac {q^{-1}} 2 (1+q^{-p+2}) +\frac {q^{-p}}
    2 - \frac{q^{-n/l}(q+1)} 2.
  \end{align*}
  Thus we may take the last value as $\lambda_{1}\text{(II.B.i.a) and
  } \lambda_{1}\text{(II.B.ii.a)}$.  Furthermore, \ref{lem:div-1/3}
  yields
  \begin{align*}
    \lambda_{2}(\text{II.B.i.a})  =
    \frac{1}{2}q^{-\dg/l-l}(q^{s+3}-q^{\lfloor s/p\rfloor +3}).
  \end{align*}

  When $V$ is II.B.i.b or II.B.ii.b, we have for $\lambda_{0}$ in the
  notation of \ref{cor:decom} that $k = r = p \neq \dg/p=m$ and $\mu =
  \gcd(p-1, \dg/p)=1$, since all proper divisors of $\dg/p$ are at
  least $l=p$. Thus we may apply \ref{cor:decom-1} to find
  \begin{align*}
    \lambda_{0}(V)&= \frac{1}{2} (1-q^{-p})(1-q^{-1}(1+q^{-p+2}
    \frac{(1-q^{-1})^{2}}{1-q^{-p}}))\\
    &= \frac{1}{2}(1-q^{-1}-q^{-p+1}+q^{-p}).
  \end{align*}
  At II.B.i.b, we have $p \nmid \dg/p$, so that \ref{cor:inj-1} for
  $D_{\dg,\dg/p}$ implies
  $$
  \lambda_{1}(\text{II.B.i.b})= \frac{1}{2},
  $$
  and \ref{lem:div-1/4} yields
  $$
  \lambda_{2}\text{(II.B.i.b)}= \lambda_{4}\text{(II.B.i.b)}= 0.
  $$


At II.B.ii.a, we have $l<p$, and \ref{cor:ffchar-2} says that
  \begin{align*}
    \lambda_{2}(\text{II.B.ii.a}) &= \frac{1}{2}q^{-l+\lceil
      l/p\rceil}= \frac{1}{2}q^{-l+1}.
  \end{align*}

  At II.B.ii.b.$\alpha$, we have $k=\dg/p$ and $r=p=z=m$ in
  \ref{cor:decom-3} for $D_{\dg,\dg/p}$, so that 
\begin{align*}
    \lambda_{1}(\text{II.B.ii.b.}\alpha)&=\frac{1}{2} (1-q^{-1})
    (\frac 1 2 + \frac{1+q^{-1}}{2p+2}
    +\frac{q^{-1}} 2 \\
    & \quad-q^{-n/p} \frac{1-q^{-p+1}}{1-q^{-p}} - q^{-p+1}\frac{1-q^{-1}}
    {1-q^{-p}}).
  \end{align*}
Furthermore, from \ref{cor:ffchar-3} we
  have 
  \begin{align*}
    \lambda_{2}(\text{II.B.ii.b}) &= \frac{1}{2}q^{-\dg/p^{2} + \lceil
      \dg/p^{3}\rceil}. 
  \end{align*}
  At II.B.ii.b.$\beta$, we have for $D_{\dg,\dg/p}$ that $k=\dg/p$,
  $r=p^{d-1} \neq p=m$, since $d \geq 3$, and $\mu= \gcd(r-1,m)=
  \gcd(p^{d-1}-1,p)=1$, so that \ref{cor:decom-1} yields
  \begin{align*}
    \lambda_{1}(\text{II.B.ii.b}.\beta) &= \frac{1}{2}
    \bigl(1-q^{-1}(1+q^{-p+2}
    \frac{(1-q^{-1})^{2}}{1-q^{-p}})\bigr) (1-q^{-\dg/p})\\
    &= \frac{(1-q^{-1})(1-q^{-p+1})(1-q^{-\dg/p})}{2(1-q^{-p})}.
  \end{align*}
  \ref{cor:ffcharb-4} says that
  \begin{align*}
    \lambda_{4}(\text{II.B.ii.b}.\alpha)&=
    \frac{1}{2}q^{-\dg/p+p+\dg/p^{2}-1}(1-q^{-1})(1-q^{-p+1}).
  \end{align*}


  We find the following bounds on $\nu$ at the leaves.

  I.A:
  $$
  \nu= \lambda_{0}(\text{I.A}) = 1,
  $$

  I.B: We have $\lambda_{3}(\text{I.B})=q^{-p+1}$, and all Frobenius
  compositions except $x^{p}\circ x^{p}$ are collisions. Thus
  $$
   1-q^{-p+1}(1-q^{-p+1}) \geq \nu \geq \frac 1 2
   (1+\frac{1}{p+1})(1-q^{-2})+{q^{-p}} > 1/2. 
  $$

  II.A.i:
  \begin{align*}
    \nu  &\leq 1 + \beta_{\dg}- \lambda_{4}(\text{II.A.i}) =
    1 + \beta_{\dg}-\frac 1 2 q^{-n/l-l+s+3} \leq 1+ \beta_{\dg},\\
    \nu  &\geq \lambda_{0}(\text{II.A})+
    \lambda_{1}(\text{II.A})-\lambda_{2}(\text{II.A.i})=
    1-\beta_{\dg}^{*}.
  \end{align*}

  II.A.ii:
  \begin{align*}
    \nu &\leq 1 + \beta_{\dg}- \lambda_{4}(\text{II.A.ii})=1+
    \beta_{\dg}-
    \frac{1}{2} q^{-\dg/l+l+\dg/l^{2}-1} \leq 1 + \beta_{\dg},\\
    \nu  &\geq \lambda_{0}(\text{II.A})+ \lambda_{1}(\text{II.A})- \lambda_{2}(\text{II.A.ii})\\
    &=\frac{1}{2} +
    \frac{1}{2}- \frac{1}{2} q^{-\dg/l+l+s-1}= 1- \frac{1}{2}q^{-\dg/l+l+s-1}.
  \end{align*}

  II.B.i.a: 

  For the lower bound, we find
  \begin{align}
    \nonumber
    \nu&\geq\lambda_{0}(\text{II.B.i.a})+\lambda_{1}(\text{II.B.i.a})-\lambda_{2}(\text{II.B.i.a})+\nu_{3}(\text{II.B})\\\nonumber
    &=\frac{1}{2}+\frac{1}{2}(1-q^{-1}(1+q^{-p+2})+q^{-p}-q^{-\dg/l}(q+1))\\\nonumber
    &\quad-\frac{1}{2}q^{-\dg/l-l}(q^{s+3}-q^{\lfloor s/p\rfloor +3})+\frac{1}{2}q^{-l-\dg/l+\dg/p+1}\\
    \label{eq:lowss} &\geq 1-\frac{1}{2}
    (q^{-1}+q^{-p+1})+\frac{q^{-p}}{2}-\frac {q^{-\dg/l}}2(q +
    1+q^{s-l+3}-q^{\dg/p-l+1}).
  \end{align}

  At the present leaf, we have $n=alp$ with $p>l\geq 2$ and $a\geq 1.$
  Thus $\dg/l\geq p$ and
  $$q^{-p}\geq q^{-\dg/l}.$$ 
  Furthermore, $\dg/p\geq l$ and
  $$q^{\dg/p-l+1}\geq q.$$
  It follows that
  \begin{align}\label{eq:follow}
    \nu \geq 1 -\frac{1}{2}(q^{-1}+q^{-p+1}+q^{-\dg/l-l+s+3}).
  \end{align}

  II.B.i.b: 
  $$
  \nu \leq 1 + \beta_{\dg}-
  \lambda_{4}(\text{II.B.i.b})-\nu_{3}(\text{II.B})= 1+\beta_{\dg}-0-
  \frac{1}{2} q^{-p+1}.
  $$

  We claim that $\beta_{\dg} \leq \frac{1}{2}q^{-p+1}$, so that $\nu
  \leq 1$. We may assume that $\dg \notin \{l^{2},ll_{2}\}$, since
  otherwise $\beta_{\dg}=0$. Setting $\mu = \log_{q}(2/(1-q^{-1}))$, we
  have $0 < \mu \leq 2$ and $2\beta_{\dg}= q^{-c+\mu} \leq q^{-c+2}$,
  so that it suffices to show
  $$
  l-1 = p-1 \leq c-2 = \frac{(\dg-ll_{2})(l_{2}-l)}{ll_{2}}-2.
  $$

  Abbreviating $a = \dg/ll_{2}$, this is equivalent to
  \begin{equation}\label{eq:al}
    \frac{l+1}{l_{2}-l}+1 \leq a.
  \end{equation}
  Since $p=l$ and $p^{2}\nmid\dg$, we have $l \nmid a$ and $a \geq
  l_{2} > l$, by the minimality conditions on $l$ and $l_{2}$. If
  $l_{2} \geq l+2$, \ref{eq:al} holds. If $l_{2}=l+1$, then $l=2$ and
  $a \geq 4$ is required for \ref{eq:al}. Since $2 \nmid a$, it
  remains the case $a=3$, corresponding to $\dg = 18$ and $p=2$. One
  checks that $\beta_{18}\leq \frac{1}{2}q^{-1}$ for $q \geq 4$. For
  $q=2$, we have to go back to \ref{eq:e} and check that
  $\nu_{3}=q^{10}(1-q^{-1})$ and
  $$
   \# D_{18}\leq \alpha_{18}-\nu_{3} +2q^{9}(1-q^{-1})= \alpha_{18}.
  $$ 
For the lower bound, we have
  \begin{align*}
    \nu  &\geq \lambda_{0}(\text{II.B.i.b}) +
    \lambda_{1}(\text{II.B.i.b})-
    \lambda_{2}(\text{II.B.i.b})+\nu_{3}(\text{II.B})\\
     &=\frac{1}{2}(1-q^{-1}-q^{-p+1}+q^{-p})+ \frac{1}{2}-0+\frac{1}{2}q^{-p+1}\\
     &=1- \frac{1}{2}(q^{-1}-q^{-p}).
  \end{align*}


  At II.B.ii.a, we have
    \begin{align*}
    \nu&\geq\lambda_{0}\bigl(\textrm{II.B.ii.a}\bigr)
    +\lambda_{1}\bigl(\textrm{II.B.ii.a}\bigr)-\lambda_{2}\bigl(\textrm{II.B.ii.a}\bigr)+\nu_{3}\bigl(\text{II.B}\bigl)\\
    &=\frac{1}{2}+\frac{1}{2}-\frac{q^{-1}}{2}(1+q^{-p+2})+\frac{q^{-p}}{2}-\frac{q^{-\dg/l}(q+1)}{2}\\
    &\quad-\frac{q^{-l+1}}{2}+\frac{q^{-l-\dg/l+\dg/p+1}}{2}\\
    &=1-\frac{1}{2}(q^{-1}+q^{-p+1})+\frac{q^{-p}}{2}-\frac{q^{-l+1}}{2}+\frac{q^{-\dg/l}}{2}(q^{\dg/p-l+1}-q-1).
  \end{align*}
  Since $\dg=al^{2}p$ with $a\geq 1$, we have $\dg/p\geq l^{2}>l+1$,
  and
  \begin{align*}
    q^{\dg/p-l+1}&>q^{2}>q+1,\\
    \nu &>1-\frac{1}{2}(q^{-1}+q^{-p+1}-q^{-p}+q^{-l+1}).
  \end{align*}

II.B.ii.b.$\alpha$:
\begin{align}
  \nu &\geq
  \lambda_{0}(\text{II.B.ii.b})+
  \lambda_{1}(\text{II.B.ii.b.$\alpha$}) -
  \lambda_{2}(\text{II.B.ii.b}) + \nu_{3}(\text{II.B})
  \nonumber \\
  &= 
  \begin{aligned}[t]
    \frac{1}{2}(&1-q^{-1}-q^{-p+1}+q^{-p}) +\frac{1}{2} (1-q^{-1})
    (\frac 1 2 + \frac{1+q^{-1}}{2p+2} +\frac{q^{-1}} 2
    \\
    &-q^{-n/p} \frac{1-q^{-p+1}}{1-q^{-p}} - q^{-p+1}\frac{1-q^{-1}}
    {1-q^{-p}}) - \frac{1}{2} q^{-n/p^{2}+\lceil n/p^{3}\rceil }
    +\frac{1}{2}q^{-p+1}
  \end{aligned}
  \nonumber \\
  \label{IIBiib1}
  &=
  \begin{aligned}[t]
    \frac 1 2 \bigl( &\frac{3}{2} + \frac{1}{2p+2} -{q^{-1}}
    -\frac{q^{-2}} 2 (1+\frac 1
    {p+1})+\frac{q^{-p}(2-q^{-1}-q^{-p})}{1-q^{-p}}\\
    &-q^{-n/p^{2}+\lceil n/p^{3}\rceil}
    -q^{-n/p} \frac{(1-q^{-1})(1-q^{-p+1})}{1-q^{-p}} \bigr).
  \end{aligned}
\end{align}

  We have $\dg=ap^{2}$ with $a>p$ and all prime divisors of $a$ larger
  than $p$. If $p\geq 3$, then $a\geq p+2$ and
  \begin{align}
    a&\geq p+2>p+1+\frac{1}{p-1}=\frac{p^{2}}{p-1},\nonumber\\
    a&\geq p+\frac{a}{p},\nonumber\\
    a&\geq p+\left\lceil\frac{a}{p}\right\rceil,\nonumber\\
    \label{alimine:1}
    q^{-p} &\geq q^{-\dg /p^{2}+\lceil{\dg /p^{3}}\rceil}.
  \end{align}

  We may now assume that $p=2$. If $a\geq 5$, then
  $$
  a-\frac{a}{2}=\frac{a}{2}\geq 2 =p,
  $$
  and \ref{alimine:1} again holds. In the remaining case $p=2$ and
  $a=3$, we have $n=12$ and \ref{alimine:1} is
  false. 
  Furthermore, we have $p < \dg/p$ and
  \begin{align*}
    \frac{q^{-p}}{1-q^{-p}} & > q^{-\dg/p}\frac{(1-q^{-1})(1-q^{-p+1})}{1-q^{-p}},\\
    q^{-p+1}(1-q^{-1})^{2} & = q^{-p+1}-(2-q^{-1})q^{-p}\geq q^{-p+1}-
    2q^{-p},
  \end{align*}
  so that for $n\neq 12$ the following holds:
  \begin{align*}
    \nu \geq \frac 1 2 \bigl(\frac{3}{2} + \frac{1}{2p+2} -{q^{-1}}
    -\frac{q^{-2}} 2 (1+\frac 1 {p+1}) -
    \frac{q^{-p+1}}{1-q^{-p}}\bigr).
  \end{align*}

  For $\dg = 12$, we have calculated in \ref{ex:p2} that
  $\lambda_{2}(\text{II.B.ii.b})= t/\alpha_{12} \leq q^{-2}= q^{-p}$,
  and we may use this to the same cancellation effect as
  \ref{alimine:1}, so that the last inequality also holds for $\dg=12$.

  II.B.ii.b.$\beta$: 
  \begin{align}\label{eq:numneu}
    \nu &\geq \lambda_{0}(\text{II.B.ii.b})+
    \lambda_{1}(\text{II.B.ii.b.$\beta$})-
    \lambda_{2}(\text{II.B.ii.b})+\nu_{3}(\text{II.B})\nonumber
    \\
    &=\frac{1}{2}(1-q^{-1}-q^{-p+1}+q^{-p}) + \frac 1 2
    \frac{(1-q^{-1})(1-q^{-p+1})(1-q^{-\dg/p})}{1-q^{-p}}\nonumber\\
    &\quad-\frac{1}{2}q^{-\dg/p^{2}+\left\lceil\dg/p^{3}\right\rceil}
    +\frac{1}{2}q^{-p+1}\nonumber
    \\
    &= 1-q^{-1}- \frac{q^{-p+1}}{2}\cdot
    \frac{(1-q^{-1})^{2}}{1-q^{-p}} + \frac{q^{-p}}{2}\\
    & \quad\quad  - \frac {q^{-\dg/p}(1-q^{-1}-q^{-p+1}+q^{-p})} {2(1-q^{-p})}
    -\frac{1}{2}q^{-\dg/p^{2}+\dg/p^{3}}.\nonumber
  \end{align}

  Since $n \geq p^{3}$, we have
  \begin{align}\label{eq:numne}
  \dg/p &\geq p^{2}>p,\nonumber\\
  q^{-p} &>q^{-\dg/p},\nonumber\\
  \dg(p-1) &\geq p^{3}(p-1),\nonumber\\
  -p+1 &\geq - \frac{\dg}{p^{2}}+ \frac{\dg}{p^{3}},\nonumber\\
  \nu &\geq 1-q^{-1}-\frac{q^{-p+1}}{2}- \frac{1}{2}q^{-\dg/p^{2}+\dg/p^{3}}
    \geq 1-q^{-1}-q^{-p+1}. \qed
  \end{align}
 \end{proof}
Except at I.B and II.B.ii.b.$\alpha$, the lower bounds are of the
satisfactory form $1-O(q^{-1})$. The leaf I.B is discussed in
\ref{ex:com}.  For small values of $q$, the entry in \ref{tab:leaves}
at II.B.ii.b.$\alpha$ provides the lower bounds in
\short\ref{tab:Lowerbounds}.

\begin{table}[h!]
  $$
  \begin{array}{c|l}
    q  & \#D_{\dg}/\alpha_{\dg} \geq \\\hline
    2& 1/6 > 0.1666 \\
    3 & 259/468 > 0.5534  \\
    4 & 133/240 > 0.5541\\
    5 & 106091/156200 > 0.6791 \\
    7 & 56824055/80707116 > 0.7040 \\
    8 & 2831/4032 > 0.7021 \\
    9 & 88087/117936 > 0.7469
  \end{array}
  $$
  \caption{Lower bounds at the leaf II.B.ii.b.$\alpha$, where
    $l^{2}=p^{2}\parallel \dg\neq p^{2}$.}
  \label{tab:Lowerbounds}
\end{table}

The multitude of bounds, driven by the estimates of \ref{sec:usd} and
\bare\ref{sec:collcomp}, is quite confusing.  The \ref{cor:Fq} in the
introduction provides simple and universally applicable estimates.
Before we come to its proof, we note that for special values, in
particular for small ones, of our parameters one may find better
bounds in other parts of this paper.

\begin{proof}[\ref{cor:Fq}] \ref{cor:Fq-6} follows from $2\leq l\leq
  \sqrt \dg$. The first upper bound on $\# D_{\dg}$ in \ref{cor:Fq-1}
  follows from \ref{cor:l-2}. It remains to deduce the lower
  bounds. Starting with the last claim, we note that \ref{cor:Fq-5} is
  \ref{cor:l-3}. In the assumption of \ref{cor:Fq-4}, the leaves I.B
  and II.B.ii.b.$\alpha$ are disallowed. We claim that \ref{thm:multi}
  implies
  \begin{align}\label{eq:numine}
    \nu \geq 1- 2 q^{-1}
  \end{align}
  at all leaves but these two. Leaf I.A is clear. At II.A.i, we have
  $n=al$, where $a>l$ and all prime factors of $a$ are larger than
  $l$.  When $a\geq l+2$, then
  \begin{align*}
    \frac n l - \frac n {l^2} = a(1-\frac 1 l ) &\geq (l+2)(1-\frac 1
    l ) = l+1-\frac 2 l \geq l,
    \\
    \beta^{*}_{\dg}&\leq q^{-n/l-l+n/l^2+3} \leq q^{3-2l} \leq
    q^{-1},\\
    \nu&\geq 1-\beta^{*}_{\dg}\geq 1-q^{-1}.
  \end{align*}
  When $a=l+1$, then $l=2$, $a=3$, $n=6$, and by \ref{thm:Estimate-3}
  we have again
  $$
   \frac{\#D_6}{\alpha_6} \geq 1-\beta^*_6 = 1-q^{-1}.
  $$
  At II.A.ii, we have $n=al^2$ with $a\geq l$ and
  \begin{align*}
  \frac n l - \frac n {l^2}& = a(l-1)\geq l(l-1) \geq l,\\
  q^{-n/l+l+n/l^2-1}& \leq q^{-1},\\
  \nu&\geq1-q^{-1}/2.
  \end{align*}
  
  At II.B.i.a, we consider the inequality
  \begin{align}\label{equ:inequalii}
    -\frac{\dg}{l}-l+s+3\leq -1,
  \end{align}
  with $s=\left\lfloor \dg/l^{2}\right\rfloor\leq \dg/l^{2}$. It holds
  for $l\geq 3$. When $l=2$, it holds for $\dg\geq 8$, and one checks
  it for $\dg=6$. Now $\dg=4$ is case II.B and excepted here. 
  Thus \ref{equ:inequalii} holds in all cases at II.B.i.a, and
  \ref{eq:follow} implies that $\nu\geq 1-3q^{-1}/2>1-2q^{-1}$.  
  
  \ref{eq:numine} is clear for II.B.i.b and II.B.ii.b.$\beta$. At
  II.B.ii.a, we have $p>l\geq 2$, and \ref{eq:numine} follows from
  \ref{tab:leaves}. This concludes the proof of \ref{cor:Fq-4}.

  In \ref{cor:Fq-2}, the second inequality follows from
  $(3-2q^{-1})\cdot(1-q^{-1})/4>1/2$ when $q\geq5$. For the first
  inequality, we have $1-2q^{-1} \geq (3-2q^{-1})/4$ when $q>5$. Thus
  it remains to prove \ref{cor:Fq-2} at II.B.ii.b.$\alpha$. It is
  convenient to show \ref{cor:Fq-1} and \ref{cor:Fq-2} together at
  this leaf.

  We have for $p\geq 3$ and $q\geq 5$ that
  \begin{align*}
    1-q^{-3}&\geq q^{-2}(3q+4)>q^{-2}(3p+4)-q^{-5}(p+2)\\
    &=q^{-2}(p+2)(1-q^{-3})+q^{-2}(2p+2),\\ \quad
    \frac{1}{2p+2}&>\frac{q^{-2}(p+2)}{2p+2}+\frac{q^{-2}}{1-q^{-3}}\geq
    \frac{q^{-2}}{2}(1+\frac{1}{p+1})+\frac{q^{-p+1}}{1-q^{-p}},
  \end{align*}
  and from \ref{tab:leaves}
  \begin{align}\label{eq:nufrom}
    \nu&\geq\frac{1}{2}(\frac{3}{2}+\frac{1}{2p+2}-q^{-1}-\frac{q^{-2}}{2}(1+\frac{1}{p+1})-\frac{q^{-p+1}}{1-q^{-p}})\\
    &>\frac{3}{4}-\frac{q^{-1}}{2}=\frac{3-2q^{-1}}{4}.\nonumber
  \end{align}



  For the remaining cases $q=3$ or $p=2$, we use \ref{IIBiib1}. At
  the current leaf, we can write $\dg=ap^{2}>p^{2}$ with all prime divisors
  of $a$ greater than $p$, and split the lower bound into two
  summands:
  \begin{align*}
    \nu_{q}&= \frac 1 2 \bigl( \frac{3}{2}
    + \frac{1}{2p+2}-{q^{-1}} -\frac{q^{-2}} 2 (1+\frac 1{p+1})+\frac{q^{-p}(2-q^{-1}-q^{-p})}{1-q^{-p}}\bigr),\\
    \quad \epsilon_{q,\dg}&= \frac 1 2 \bigl(q^{-a+\lceil a/p
      \rceil}
    +q^{-ap} \frac{(1-q^{-1})(1-q^{-p+1})}{1-q^{-p}} \bigr),
  \end{align*}
  so that $\nu\geq \nu_{q}-\epsilon_{q,\dg}$, and $\epsilon_{q,\dg}$
  is monotonically decreasing in $a$.

  For $q=3$, we have $a\geq 5$,
  \begin{align*}
    \nu_{3}&=\frac{203}{27\cdot 13}>0.5783,\\ \quad \nonumber
    \epsilon_{3,\dg}&\leq\frac{1}{2}(3^{-a+\lceil
      a/3\rceil}+\frac{8}{13}\cdot 3^{-3a})\leq \epsilon_{3,45}= \frac
    1 2(3^{-5+2}+\frac{8}{13}\cdot 3^{-15})\\\quad \nonumber
    &=\frac{1}{54}+\frac{4}{13}\cdot 3^{-15}<0.0186,\\\quad \nonumber
    \nu&\geq\nu_{3}-\epsilon_{3,\dg}>0.5598>1/2.\quad\nonumber
  \end{align*}
  For $p=2$, we find
  \begin{align*}
    \nu_{q}&=\frac{5}{6}-q^{-1}+\frac{q^{-2}} 6,\\ \quad\nonumber
    \epsilon_{q,\dg}&=\frac 1 2 (q^{-(a-1)/2}+q^{-2a}\cdot
    \frac{1-q^{-1}}{1+q^{-1}}).\nonumber
  \end{align*}
  When $q\geq 8$ and $\dg\geq 28$, so that $a\geq 7$, we have
  \begin{align*}
    \frac{q^{-2}}{6}&\geq\frac 1 2
    (q^{-3}+q^{-14}\cdot\frac{1-q^{-1}}{1+q^{-1}})=\epsilon_{q,28}\geq\epsilon_{q,\dg},\\
    \nu&\geq\nu_{q}-\epsilon_{q,\dg}\geq \frac 5 6 -q^{-1}\geq
    \frac{3-2q^{-1}}{4}.
  \end{align*}
  For the remaining values $q\in \{ 2,4 \}$ or $\dg\in \{ 12,20 \}$, we note
  the values
  \begin{align*}
     \nu_{2}&=\frac 3 8, \\
    \nu_{4}&=\frac{19}{32},\\
    \epsilon_{q,12}&=\frac 1 2(q^{-1}+q^{-6}\cdot
    \frac{1-q^{-1}}{1+q^{-1}}),\\
    \epsilon_{q,20}&=\frac 1
    2(q^{-2}+q^{-10}\cdot\frac{1-q^{^{-1}}}{1+q^{-1}}).
  \end{align*}
  We find that $\nu\geq(3-2q^{-1})/4$ for $q\geq 8$ and $\dg=20$, and
  for $q\geq 16$ and $\dg =12$. \ref{tab:numforsm} shows that this
  also holds for $(q,\dg)=(8,12).$ When $q=4$, we have $\nu\geq 1/2$
  for $\dg\geq 20$ by the above, and according to \ref{tab:numforsm}
  also for $\dg =12$.

  \begin{table}[h!]
    $$
    \begin{array}{r|r|r|r}
      q,\dg
      &\# D_{\dg}
      & \alpha_{\dg}
      &\# D_{\dg}/\alpha_{\dg} \geq\\ \hline
      2,4&6&8&0.7500\\
      2,8 & 36 & 64 & 0.5625\\
      2,12 & 236 & 256&0.9218\\
      2,16 & 762 & 1\,024 & 0.7441\\
      2,20 &3\, 264&4\, 096&0.7968\\
      2,24 & 14\,264 & 16\,384 & 0.8706\\
      2,28&49\,920 & 65\,536 &0.7617\\
      2,36 & 821\,600 & 1\,048\,576& 0.7835\\
      4,4&132&192&0.6875\\
      4,12& 100\, 848&98\, 304&1.0258\\
      8,4&2\,408&3\,584&0.6718\\
      8,12& 30\, 382\, 016 &29\, 360\, 128& 1.0348\\
      16,4&41\,040&61\,440&0.6679\\
      32,4& 677\,536& 1\,015\,808&0.6669\\
      64,4& 11\,011\,392& 16\,515\,072&0.6667\\
      128,4& 177\,564\,288& 266\,338\,304&0.6666\\
      256,4& 2\,852\,148\,480&4\,278\,190\,080&0.6666\\
      3,9&414&486&0.8518\\
      9,9&450\, 792&472\,392&0.9542\\
      5,5&7\,798\,100&7\,812\,500&0.9981\\
    \end{array}
    $$
    \caption{Decomposable polynomials of degree $\dg$ over $\mathbb{F}_{q}$.}
    \label{tab:numforsm}
  \end{table}
 
  When $q=2$, the values above only show that $\nu\geq 1/4$ for
  $\dg\geq 28$. However, a different and simple approach gives a
  better bound for $\dg=4a$ with an odd $a\geq 3$ over
  $\mathbb{F}_{2}$. We exploit the special fact that
  $x^{2}+x\in\mathbb{F}_{2}[x]$ is the only quadratic original
  polynomial that is not a square.

  Any $g\in\mathbb{F}_{2}[x]$ is uniquely determined by
  $f=g\circ(x^{2}+x)$, due to the uniqueness of the Taylor
  expansion. The number of original $g$ of degree $2 a$ and that are
  not a square is $2^{2a-1}-2^{a-1}$, and by composing with a linear
  polynomial on the left, we have $\# D^{+}_{\dg,\dg/2}=
  2^{2a}-2^{a}=2^{\dg/2}-2^{\dg/4}$.
  Similary, $(x^{2}+x)\circ h=(x^{2}+x)\circ h^{*}$ with $h\neq h^{*}$
  implies that $-1=h^{*}-h$, so that one of the two polynomials is not
  original. Thus $\gamma_{\dg,2}$ is also injective on the original
  polynomials, and $\#
  D_{\dg,2}^{+}=2^{\dg/2}-2^{\dg/4}$. Furthermore, \ref{cor:ffchar-3}
  says that
  $$
  t=\# (D_{\dg,2}^{+}\cap D_{\dg,\dg/2}^{+})\leq 2^{\dg/4+\lceil
    \dg/8\rceil+1} = 2^{3\dg/8+3/2}.
  $$ 

  The number of Frobenius compositions (that is, squares) of degree $\dg$ equals
  $\#D_{\dg}^{\varphi}=2^{2a}$, and $\alpha_{\dg}=2^{\dg/2+2}$.  It
  follows that
  \begin{align}
    \label{eq:prota}
    \# D_{\dg}&\geq \# D^{+}_{\dg,2}+\# D_{\dg,\dg/2}^{+}-t+\#
    D_{\dg}^{\varphi}\nonumber\\
    &\geq 2\cdot
    2^{\dg/2}(1-2^{-\dg/4})-2^{3\dg/8+ 3/2}+2^{\dg/2}\nonumber\\
    &=(\frac{3}{4}-2^{-\dg/8-1/2}-2^{-\dg/4-1})\alpha_{\dg},\\
    \nu&\geq\frac{3}{4}-2^{-5/2-1/2}-2^{-5-1}=\frac{39}{64}>0.6093>1/2\nonumber
  \end{align}
  for $\dg\geq 20$. Using \ref{tab:numforsm} for $\dg=12$, we find
  $\nu>1/2$ also for $q=2$, and hence for all values at leaf
  II.B.ii.b.$\alpha$. Now it only remains to prove $\nu\geq 1/2$ in
  \ref{cor:Fq-1}. The leaf II.B.ii.b.$\alpha$ has just been dealt
  with. Since $1-q^{-1}\geq 1/2$ for all $q$, the claim follows from
  the previous bounds at the leaves I.A, II.A.i, II.A.ii, and
  II.B.i.b. At II.B.i.a, we have shown $\nu\geq 1-3q^{-1}/2\geq 1/2$
  for $q\geq 3$; since $p\neq l$ and hence $p\geq 3$ at this leaf, the
  claim follows. Similarly, we have at II.B.ii.a that $q\geq p\geq 3$
  and $\nu\geq 1-\frac 1 2(q^{-1}+q^{-l+1}+q^{-p+1}-q^{-p})\geq
  1-q^{-1}-q^{-2}\geq 1/2$. Now remain the two leaves I.B and
  II.b.ii.b.$\beta$.

  At leaf I.B, we have $\dg =p^{2}$ and
$$
q^{2}\geq q+2\geq p+2,
$$
$$
\frac{1}{p+1}+2q^{-p}> q^{-2}(1+\frac{1}{p+1}).
$$
From \ref{ex:com} we find
$$
\nu\geq \frac{1}{2}(1+\frac{1}{p+1})(1-q^{-2})+q^{-p}\geq\frac 1 2.
$$
\ref{tab:numforsm} gives the exact values of $\nu$ for $p=2$ and
$q\leq 256$.

At the final leaf II.B.ii.b.$\beta$, we have $l=p$ and $p^{3}\mid\dg$. The
lower bound in \ref{tab:leaves} implies $\nu\geq 1/2$ for $q\geq
4$. When $q=3$, \ref{eq:numne} yields
$$
\nu\geq 1-\frac 1 3-\frac 1 {9}=\frac 5 9>\frac 1 2.
$$
For $q=2$, we have from \ref{eq:numneu}
$$
\nu\geq\frac 1 2 +\frac 1{24}-\frac{2^{-\dg/2-1}}{3}-2^{-\dg/8-1}.
$$
When $\dg\geq 32$, this shows $\nu\geq 1/2$. For the smaller values
$8,16,$ and $24$ of $\dg$, the data in \ref{tab:numforsm} are sufficient.
\qed
\end{proof}
Two features are worth noting. Firstly, our lower bounds are rather
pessimistic when $q=2$, yielding for $\dg=12$ that $\nu\geq 
47/384>0.1223$ by \ref{IIBiib1}, $\nu\geq 3/16= 0.1875$ from the special
argument, compared to $\nu=59/64>0.9218$ from our
experiments. Secondly, our lower bounds are strictly increasing in
$\dg$, while the experiments show a decrease in $\nu$ from $\dg=12$ to
$\dg=20$. Both features show that more work is needed to understand
the case $ p=l$ and $p^{2}\parallel \dg$, where the latter means that
$p^{2}\mid\dg$ and $p^{3}\nmid\dg$.

Much effort has been spent here in arriving at precise bounds, without
asymptotics or unspecified constants. We now derive some conclusions
about the asymptotic behavior.  There are two parameters: the field
size $q$ and the degree $\dg$. When $\dg$ is prime, then $\#
D_{\dg}=\alpha_{\dg}=0$, and prime values of $\dg$ are excepted in the
following. We consider the asymptotics in one parameter, where the
other one is fixed, and also the special situations where
$~gcd(q,n)=1$. Furthermore, we denote as ``$q,\dg\longrightarrow
\infty $'' the set of all infinite sequences of pairwise distinct
$(q,\dg)$. The cases $p^{2}\parallel\dg$ are the only ones where
\ref{tab:leaves} does not show that $\nu\longrightarrow 1$.
\begin{theorem}\label{thm:consid}
  Let $\nu_{q,\dg}=\# D_{\dg}/\alpha_{\dg}$ over $\mathbb{F}_{q}$. We
  only consider composite $\dg$.
  \begin{enumerate}
  \item\label{thm:consid-1} For any $q$, we have
    $$
    \underset{\dg\to\infty}{\lim\sup} ~ {\nu_{q,\dg}}=1,
    $$
    $$
    \lim_{\substack{\dg\to\infty \\ \gcd(q,\dg)=1}}{\nu_{q,\dg}}=1,
    $$
    \begin{align*}
      \frac 1 2 &\leq \nu_{q,\dg} \text{ for any }
        \dg,\\
      \frac{3-2q^{-1}}{4}&\leq \nu_{q,\dg} \text{ for any }\dg,\text{ if }
        q\geq 5.
    \end{align*}
  \item\label{thm:consid-2} Let $\dg$ be a composite integer and $l$
    its smallest prime divisor. Then
    $$
    \underset{q\to\infty}{\lim\sup} ~{\nu_{q,\dg}}=1,
    $$
    \begin{align*}
      \underset{{q\to\infty}}{\lim\inf} 
      ~{\nu_{q,\dg}}\begin{cases}\geq\frac 1 2
        (1+\frac{1}{l+1})\geq\frac 2 3 &\text{ if }\dg=l^{2},\\ 
        \geq\frac 1 4 (3+\frac{1}{l+1})\geq \frac 5 6 &\text{ if }
          l^{2}\parallel\dg \text{ and }\dg\neq l^{2},\\
        =1 &\text{ otherwise,}
      \end{cases}
    \end{align*}
    $$
    \lim_{\substack{q\to\infty\\ \gcd(q,\dg)=1}}{\nu_{q,\dg}}=1.
    $$
  \item\label{thm:consid-3} For any sequence $q,\dg \rightarrow
    \infty$, we have
    $$\frac 1 2 \leq \underset{q,\dg\to\infty}{\lim\inf} ~{\nu_{q,\dg}}\leq
    \underset{q,\dg\to\infty}{\lim\sup} 
    ~{\nu_{q,\dg}}=1, 
    $$
    $$
    \lim_{\substack{q,\dg\to\infty\\ \gcd(q,\dg)=1}}{\nu_{q,\dg}}=1.
    $$
  \end{enumerate}
\end{theorem}
\begin{proof} \short\ref{thm:consid-1} We start with an upper
  bound. The conclusions of the Main Theorem are too weak for our current
  purpose, and we have to resort to \ref{thm:Estimate}. For the
  special $\dg$ which are a square or a cube of primes, or a product
  of two distinct primes, \ref{thm:Estimate-1} says that
  $\nu_{q,\dg}\leq 1$. For the other values, we set $d=\dg/ll_{2}$,
  and the upper bound on the $\lim\sup$ follows if we show that
  $c=(d-1)(l_{2}-l)$ is unbounded as $\dg$ grows, since then
  $\beta_{\dg}=q^{-c}/(1-q^{-1})$ tends to zero, and $\nu_{q,\dg}\leq
  1+\beta_{\dg}$. Since $l_{2}-l\geq 1$, it is sufficient to show the
  unboundedness of $d$. When $\dg=l^{e}$ is a power of a prime, we may assume by
  the above that $e\geq 4$. Then $l_{2}=l^{2}$, $l\leq\dg^{1/4}$ and
  $d=l^{e-3}\geq l^{e/4}=\dg^{1/4}$ is unbounded.

  If $\dg=l^{e}l^{e_{+}}_{+}$ has exactly two prime factors $l<l_{+}$,
  we may assume that $e+e_{+}\geq 3$.  If $e=1$, then
  $l_{2}=l_{+}$, $ e_{+}\geq 2$, and $d=l_{+}^{e_{+}-1}\geq
  l_{+}^{(e_{+}+1)/3}>\dg^{1/3}$. We now assume that $e \geq 2$. Then
  \begin{align*}
    l_{2}=\begin{cases}
      l^{2}&\text{ if }l^{2}<l_{+},\\
      l_{+}&\text{ otherwise,}
    \end{cases}
  \end{align*}
  \begin{align}\label{primfacd}
    d =\begin{cases}
      \frac{\dg}{l^{3}}&\text{ if }l^{2}<l_{+},\\
      \frac{\dg}{ll_{+}}&\text{ otherwise.}
    \end{cases}
  \end{align}

  We first treat the case where $l^{2}< l_{+}$. If $e=2$, then
  $$
  d=l_{+}^{e_{+}}/l>l_{+}^{e_{+}-1/2}\geq
  l_{+}^{(1+e_{+})/4}>\dg^{1/4}.
  $$
If $e=3$, then
  $$
  d=l_{+}^{e_{+}}>l_{+}^{(e_{+}+3/2)/3}>(l^{2})^{1/2}l_{+}^{e_{+}/3}=\dg^{1/3}.
  $$
If $e\geq 4$, then $d=l^{e-3}l_{+}^{e_{+}}\geq
l^{e/4}l^{e_{+}}_{+}>\dg^{1/4}$. Next we deal with $l_{+}>l^{2}$. If
$e=1$, we have $e_{+}\geq 2$, and then
  $$
  d=l^{e_{+}-1}_{+}\geq l_{+}^{(e_{+}+1)/3}>\dg^{1/3}.
  $$
If $e_{+}=1$ we have $e\geq 2$, and then
  $$
  d=l^{e-1}\geq l^{(e+2)/4}>l^{e/4}l_{+}^{1/4}=\dg^{1/4}.
  $$
In the remaining case, where $e$, $e_{+}\geq 2$, we have
  $$
  d=l^{e-1}l_{+}^{e_{+}-1}\geq l^{e/2}l_{+}^{e/2}=\dg^{1/2}.
  $$

  In the last case, $\dg = l^{e}l_{+}^{e_{+}}l_{++}^{e_{++}}\cdots$
  has at least three distinct prime factors $l<l_{+}<l_{++}<\cdots$,
  and
  \begin{align*}
    d=\begin{cases}
      \frac{\dg}{l^{3}}&\text{if }e\geq 2\text{ and }l^{2}<l_{+},\\
      \frac{\dg}{ll_{+}}&\text{ otherwise.}
    \end{cases}
  \end{align*}

  If $e=e_{+}=1$, then $ll_{+}<\dg^{2/3}$ and
  $d\geq\dg^{1/3}$. Otherwise, we apply the previous argument to
  $\dg^{*}=l^{e}l_{e}^{e_{+}}=\dg/m$ and $d^{*}=d/m$, where
  $m=l_{++}^{e_{++}}\cdots=\dg l^{-e}l_{+}^{-e_{+}}$. Then $d^{*}$
  equals the value $d$ defined above for $n^{*}$, and
  $$
  d=d^{*}m\geq (n^{*})^{1/4}m>\dg^{1/4}.
  $$

  In all cases, $d$ is unbounded if $\dg$ is. Thus
  $\lim\sup_{\dg\to\infty}{\nu_{q,\dg}} \leq 1$, and
  \ref{thm:Estimate-2/2} for $\dg=l^{2}$ implies that
  $\lim\sup_{\dg\to\infty} \geq 1$.

  If we only consider $\dg$ with $\gcd(q, \dg)=1$, then
  \ref{thm:Estimate-2/3} says that
  $$
  \nu_{q,\dg}\geq 1-2q^{-\dg/l+l+\dg/l^{2}-1} \geq
  1-q^{-\dg/l+l+\dg/l^{2}}.
  $$
  When $\dg$ is the product of two prime numbers, then $\nu_{q,\dg}$
  tends to $1$ for these special $\dg$. We may now assume that $n$ has
  at least three prime factors. Then $\dg \geq l^{3}$, and
  \begin{align*}
    - \frac{\dg}{l}+l+\frac{\dg}{l^{2}} &= -
    \frac{\dg}{l}(1-\frac{1}{l})+l
    \leq -\frac{\dg}{2l}+l \leq -\frac{\dg}{2\dg^{1/3}}+\dg^{1/3}\\
    &= -\frac{n^{2/3}}{2}+ \dg^{1/3} \leq - n^{1/2}
  \end{align*}
  for $\dg \geq 512$, say. The second claim in
  \short\ref{thm:consid-1} follows. The other two inequalities are in
  the Main Theorem.

\short\ref{th:fifi-2} The first claim follows from \ref{cor:l-2},
since $\dg\geq l^{2}$ and hence $\nu_{q,\dg}\leq 1+q^{-1/3}$. For the
other claims, we consider two subsequences of $q$: $q=l^{e}$ with
$e\rightarrow \infty$, and $q$ with $~gcd(q,l)=1$; we denote the
latter as $q'$. For $\dg=l^{2}$, the lower bound follows from the
entry at I.B in \ref{tab:leaves}, and for $l^{2}\parallel \dg \neq
l^{2}$ from the entry at II.B.ii.b.$\alpha$. In all other cases, the
\ref{cor:Fq} guarantees that $\nu_{l^{e}, \dg}$ and $\nu_{q',\dg}$
tend to $1$; see also \ref{eq:numine}.

\short\ref{th:fifi-3} We take some infinite sequence of $(q,\dg)$ for
which $\nu_{q,\dg}$ tends to $s=\lim\sup$. If all $q$ occurring in the
sequence are bounded, then \short\ref{th:fifi-1} implies that $s\leq
1$. Otherwise, $\nu_{q,\dg}\leq 1+q^{-1/3}$ is sufficient. The same
case distinction yields the lower bound on the limit, using the Main
Theorem \short\ref{thm:Estimate-2/3}. The lower bound on $\lim\inf$
follows from \short\ref{thm:Estimate-1}.
\end{proof}
\begin{example}\label{ex:wehal}
Let $p^{2}\parallel\dg$ and $\dg\neq p^{2}$. We study $D_{\dg}$ over
$\mathbb{F}_{q}$, using the notation of (the proof of)
\ref{thm:Estimate}. We have $l=p<l_{2}\leq p^{2}$,
\begin{align*}
c=\frac{(\dg-ll_{2})(l_{2}-l)}{ll_{2}}\geq\frac{\dg-l(l+1)}{l(l+1)}\geq\frac{\dg}{2l^{2}}.
\end{align*}
With
$$
E_{2}=\{ e\in\mathbb{N}\colon e\mid\dg,l_{2}\leq e\leq \dg/l_{2} \},
$$
we have
\begin{align*}
\sum_{e\in E_{2}} \# D_{\dg,e}&\leq \sum_{e\in E_{2}} q^{u(e)}(1-q^{-1})\leq q^{u(l)}(1-q^{-1})\frac{2q^{-c}}{1-q^{-1}}\\
&=\frac{q^{-c}}{1-q^{-1}}\cdot \alpha_{\dg}\leq 2q^{-\dg/2l^{2}}\cdot \alpha_{\dg}.
\end{align*}
We let
\begin{align*}
\lambda_{q,\dg}&=\frac{\#
  D_{\dg,p}^{+}+\#D_{\dg,\dg/p}^{+}}{\alpha_{\dg}},\\
t&=\#(D_{\dg,p}^{+}\cap D_{\dg,\dg/p}^{+}).
\end{align*}
Then
\begin{align*}
\nu_{q,\dg} &=\frac{\#D_{\dg}}{\alpha_{\dg}}\leq
\lambda_{q,\dg}-\frac{t}{\alpha_{\dg}}+\frac{\#D_{\dg}^{\varphi}}{\alpha_{\dg}}+
\frac{\sum_{e\in E_{2}} \#D_{\dg,e}}{\alpha_{\dg}}\\
&\leq\lambda_{q,\dg}+\frac{q^{\dg/p+1}(1-q^{-1})}{\alpha_{\dg}}+2q^{-\dg/2l^{2}}=\lambda_{q,\dg}+ \frac{q^{-p+1}}{2}+2q^{-\dg/2p^{2}}.
\end{align*}
On the other hand, \ref{cor:ffchar-3} says that
\begin{align*}
t&\leqq^{\dg/p+p-\dg/p^{2}+\lfloor \dg/p^{3}\rfloor +1}(1-q^{-1}),\\
\nu_{q,\dg}&\geq\lambda_{q,\dg}-\frac{t}{\alpha_{\dg}}+\frac{\#D_{\dg}^{\varphi}}{\alpha_{\dg}}\geq\lambda_{q,\dg}-\frac{1}{2}q^{-\dg/p^{2}+\lfloor\dg/p^{3}\rfloor+1}+
\frac{q^{-p+1}}{2}.
\end{align*}
For $p\geq 3$ we have
$$
-\frac{\dg}{p^{2}}+\frac{\dg}{p^{3}}+1\leq -\frac{\dg}{2p^{2}},
$$
$$
\left|\nu_{q,\dg}-(\lambda_{q,\dg}+q^{-p+1})\right|\leq 2q^{-\dg/2p^{2}}.
$$
We have presented some bounds on $\lambda_{q,\dg}$, but they are not
sufficient to determine its value in general, not even
asymptotically. However, for $q=2$ we have from \ref{eq:prota}
\begin{align}\label{al:asym}
  \lambda_{q,\dg}&=\frac{2^{\dg/2+1}(1-2^{-\dg/4})}{2 \cdot 2^{\dg/2+2}}=
  \frac{1-2^{-\dg/4}}{2}, \nonumber\\
  \frac{3}{4}-2^{-\dg/8-1/2}-2^{\dg/4-1} &\leq
  \nu_{2,\dg}\leq\frac{3}{4}+2^{-\dg/8+1}-2 ^{-\dg/4-1}.  \qed     
\end{align}
\end{example}

We have seen that $\nu_{q,\dg}$ tends to $1$ unless
$p^{2}\parallel\dg$. \ref{ex:wehal} suggests to use a correction
factor $\gamma$ so that $\nu_{q,\dg}/\gamma$ tends to $1$ also in those cases.
\begin{conjecture}\label{con:forprim}
For any prime $p$ and power $q$ of $p$ there
  exist $\gamma_{p},\delta_{q}\in\mathbb{R}$ so that
\begin{align*}
\lim_{e\longrightarrow\infty}\nu_{p^{e},p^{2}}&=\gamma_{p},\\
\lim_{\substack{\dg\longrightarrow\infty p^{2}\parallel\dg}} \nu_{q,\dg}&=\delta_{q}.
\end{align*}
\end{conjecture}
If true, this would imply that $\#
D_{p^{2}}\sim\gamma_{p}\alpha_{p^{2}}$ over extensions
$\mathbb{F}_{q}$ of $\mathbb{F}_{p}$, and $\#
D_{\dg}\sim\delta_{q}\alpha_{\dg}$ for growing $\dg$ with
$p^{2}\parallel \dg$. \ref{ex:com} shows that the first part is true
for $p=2$ and $\gamma_{2}= 2/3$, and \ref{al:asym} that the second
part holds for $q=2$ and $\delta _{2}=3/4$.

\cite{boddeb09} state without proof that $\# D_n \approx \frac 3 4
\alpha_n$ over $\F_2$ for even $n\geq 6$. Assuming a standard meaning
of the $\approx$ symbol, this is false unless $4\parallel \dg$, in
which case it is proven by \ref{al:asym}.

\begin{example}
  \ref{thm:multi-1} exhibits several situations where
  $\#D_{n}\leq\alpha_{n}$. One might wonder whether this always
  happens. We show that this is not the case. \ref{tab:numforsm} gives
  an example. More generally, we take three primes
  $2<l_{1}<l_{2}<l_{3}$, $n=l_{1}l_{2}l_{3}$, and an odd $q$ with
  $~gcd(n,q)=1$.  For $i\leq3$, we set
  \begin{align*}
    B_{i} &= D_{n,l_{i}}\cup D_{n,n/l_{i}},
    \\
    S_{i} &=\left \lfloor \frac{n}{l_{i}^{2}} \right \rfloor,
    \\
    t_{i} &= \frac{1}{2}(2q^{s_{i}+3}+q^{4}-q^{3})(1-q^{-1}).
  \end{align*}
  Then
  \begin{align*}
    D_{n} &= B_{1} \cup B_{2} \cup B_{3},
    \\
    \# B_{i} &= 2q^{n/l_{i}+l_{i}}(1-q^{-1})-t_{i}.
  \end{align*}
  For a permutation $\pi\in S_{3}$, we set
  \begin{align*}
    C_{\pi} &= \gamma_{\pi}(P_{l_{\pi 1}}^{=}\times P_{l_{\pi 2}}^{0}
    \times P_{l_{\pi 3}}^{0}),
    \\
    C &= \bigcup_{\pi\in S_{3}}C_{\pi},
  \end{align*}
  where $\gamma_{\pi}$ is the composition map for three
  components. Then for any $\pi\in S_{3}$
  $$
  \# C_{\pi} = q^{l_{1}+l_{2}+l_{3}-1}(1-q^{-1}).
  $$
  Now let $i \ne j$ and $f=g \circ h = g^{*} \circ h^{*} \in B_{i}
  \cap B_{j}$, with $\Set{ \deg g, \deg h } = \Set{ l_{i}, n/l_{i} }$
  and $\Set{\deg g^{*}, \deg h^{*} } = \Set{ l_{j}, n/l_{j} }$. To
  simplify notation, suppose that $i=1$ and $j=2$. We refine both
  decompositions into complete ones. Then for $g \circ h$, the set of
  degrees is either $\Set{l_{1},l_{2}l_{3}}$ or
  $\Set{l_{1},l_{2},l_{3}}$, and for $g^{*} \circ h^{*}$ it is either
  $\Set{l_{2},l_{1}l_{3}}$ or $\Set{l_{1},l_{2},l_{3}}$. 
  This set of degrees is unique, so that it equals
  $\{l_{1},l_{2},l_{3}\}$. It follows that $f \in C$ and $B_{i} \cap
  B_{j} \subseteq C$. Thus
  \begin{align}
    \# D_{n}  &\geq \sum_{1 \leq i\leq 3} \# B_{i} - \# C \nonumber\\
   &\geq (1-q^{-1}) \sum_{1 \leq i\leq 3} \left( 2q^{n/l_{i}+l_{i}}
      - \frac{1}{2}(2q^{s_{i}+3}+q^{4}) \right) -
    6q^{l_{1}+l_{2}+l_{3}-1}
    \nonumber\\
    \label{eq:zupf17}
    &= (1-q^{-1}) \left( 2 \sum_{1 \leq i\leq 3} q^{n/l_{i}+l_{i}} -
      \sum_{1 \leq i\leq 3} q^{s_{i}+3} - \frac{3}{2}q^{4} -
      6q^{l_{1}+l_{2}+l_{3}-1} \right).
  \end{align}
  Now suppose further that
  $$
  l_{3} \leq 2+(l_{1}-1)(l_{2}-1), \quad 5 \leq l_{2} \leq l_{1}^{2},
  \quad q \geq 7.
  $$
  Then
  \begin{align*}
    l_{1}+l_{2}+l_{3}-1  &\leq
    l_{1}+l_{2}+1+(l_{1}-1)(l_{2}-1)\\
    & \quad\quad = l_{1}l_{2}+2,\\
    6q^{l_{1}+l_{2}+l_{3}-1}  &\leq 6q^{l_{1}l_{2}+2} \leq
    q^{l_{1}l_{2}+3},
    \\
    4l_{3}  &\leq 10 (l_{3}-1) \leq l_{1}l_{2}(l_{3}-1),
    \\
    \frac{l_{1}l_{2}}{l_{3}}+4 &\leq l_{1}l_{2} < l_{1}l_{2}+l_{3},
    \\
    q^{l_{1}l_{2}/l_{3}+3} + \frac{3}{2}q^{4} +
    6q^{l_{1}+l_{2}+l_{3}-1} &< q^{l_{1}l_{2}+l_{3}} \left( q^{-1}+
      \frac{3}{2}q^{4-l_{3}}+q^{3-l_{3}}
    \right)\\
    & \quad\quad < 2q^{l_{1}l_{2}+l_{3}},
    \\
    \frac{l_{2}l_{3}}{l_{1}} &\leq l_{1}l_{3},
    \\
    \frac{l_{1}l_{3}}{l_{2}} &<l_{1}l_{3},
    \\
    q^{l_{2}l_{3}/l_{1}+3} + q^{l_{1}l_{3}/l_{2}+3}  &< (q^{3-l_{2}} +
    q^{3-l_{2}}) q^{l_{1}l_{3}+l_{2}} < q^{l_{1}l_{3}+l_{2}}.
  \end{align*}
  Finally, \ref{eq:zupf17} implies that
  \begin{align*}
    \frac{\#D_{n}}{1-q^{-1}} &\geq \frac{\alpha_{n}}{1-q^{-1}} + 2q^{l_{1}l_{3}+l_{2}} + 2q^{l_{1}l_{2}+l_{3}} - \sum_{1 \leq i\leq3} q^{\left\lfloor n/l_{i}^{2} \right\rfloor +3} -\frac{3}{2}q^{4}-6q^{l_{1}+l_{2}+l_{3}-1}\\
    &> \frac{\alpha_{n}}{1-q^{-1}}.
  \end{align*}
  As a small example, we take $l_{1}=3$, $l_{2}=5$, $l_{3}=7$, $q=11$,
  so that $n=105$ and $\alpha_{105}=2q^{38}(1-q^{-1})$. The lower
  bound in \ref{eq:zupf17} evaluates to
  \begin{align*}
    \#D_{105} &\geq \alpha_{105}+(1-q^{-1})(2(q^{26}+q^{22})-
    (q^{14}+q^{7}+q^{5} + \frac{3}{2}q^{4}+6q^{15}))
    \\
     &> \alpha_{105}+2q^{26}(1-q^{-1}).
  \end{align*}
  The general bounds of \ref{thm:Estimate-1} 
  and \ref{lem:div-1} specialize to
  \begin{align*}
    \#D_{105} &\leq
    \alpha_{105}(1+ \frac{q^{-12}} {1-q^{-1}})= \alpha_{105}+2q^{26}.
  \end{align*}
  The closeness of these two estimates indicates a certain precision in our
  bounds.
\end{example}
\begin{remark}
  We claim that if $p \nmid \dg$, then
  $$
  \#D_{\dg} \geq \alpha_{\dg}(1-q^{-1}).
  $$
  By \ref{cor:l-3}, this is satisfied if $\dg \geq 3l^{2}$. So we now
  assume that $\dg < 3l^{2}$. Then $\dg/l < 3l$, and all prime factors
  of $\dg/l$ are at least $l$. It follows that either $\dg=8$ or
  $\dg/l = l_{2}$ is prime. If $l_{2}=l$, then $\#D_{\dg}=
  \alpha_{\dg}$, by \ref{thm:Estimate-2/2}. Otherwise we have $s =
  \lfloor \dg/l^{2} \rfloor = \lfloor l_{2}/l \rfloor \leq \lfloor
  (3l-1) /l \rfloor \leq 2$ and from \ref{thm:Estimate-3} that
  $$
  \#D_{\dg} \geq \alpha_{\dg} (1-\beta_{\dg}^{*}) \geq
  \alpha_{\dg}(1-q^{-l-l_{2}+5}).
  $$
  It is now sufficient to show
  $$
  l+l_{2} \geq 6.
  $$
  This holds unless $n \in \{4,6,9 \}$, so that only $\dg = 6$ needs
  to be further considered. We have $\beta_{6}^{*} =
  q^{-2-3}(q^{1+3}+q^{4}-q^{3})/2 \leq q^{-1}$, and the claim follows
  from \ref{thm:Estimate-3}.
\end{remark}
\begin{open}
  \begin{itemize}
  \item Some polynomials have more than a polynomial number of
    decompositions. Can we find them in time polynomial in the output
    size? Or even a ``description'' of them in time polynomial in the
    input size? If not: prove (by a reduction) that this is hard?
  \item In the case where $p=l$ and $p^{2} \parallel \dg$, can one tighten
    the gap between upper and lower bounds in the Main Theorem \ref{cor:Fq-1},
    maybe to within a factor $1+O(q^{-1})$?
  \item Can one simplify the arguments and reduce the number of cases,
    yet obtain results of a quality as in the \ref{cor:Fq}? The bounds
    in \ref{th:decom} are based on ``low level'' coefficient
    comparisons. Can these results be proved (or improved) by ``higher
    level'' methods?
  \end{itemize}
\end{open}

\section{Acknowledgments}

Many thanks go to Jaime Guti\'errez for alerting me to Umberto
Zannier's paper, to Henning Stichtenoth for discussions and for
pointing out Antonia Bluher's work, to Laila El Aimani for some
computations, and to Konstantin Ziegler for drawing the tree. I
appreciate Igor Shparlinski's comments on \ref{remark:twice},
pointing out a notational infelicity, and thank Umberto Zannier for
correcting a misunderstanding. I appreciate the discussions with
Arnaud Bodin, Pierre D\`ebes, and Salah Najib about the topic, and in
particular the challenges that their work \cite{boddeb09} posed.

This work was supported by the B-IT Foundation and the Land
Nordrhein-Westfalen.


\bibliographystyle{cc2}%
\bibliography{\jobname}

\end{document}